\documentclass[11pt]{amsart}
\usepackage{format}
\newcommand\CH{{\mathcal{H} }}
\newcommand\bR{\mathbb R}
\newcommand\bZ{\mathbb Z}
\newcommand\bC{\mathbb C}
\newcommand\al{{\alpha}}
\newcommand\la{{\lambda}}

\newcommand\CG{{\mathcal{G}}}
\newcommand\CA{{\mathcal{A}}}
\newcommand\CZ{{\mathcal{Z}}}
\newcommand\CO{{\mathcal{O}}}
\newcommand\CM{{\mathcal{M}}}
\newcommand\CL{{\mathcal{L}}}

\newcommand\CE{{\mathcal{E}}}
\newcommand\CI{{\mathcal{I}}}
\newcommand\CJ{{\mathcal{J}}}

\newcommand\CV{{\mathcal{V}}}
\newcommand\CP{{\mathcal{P}}}
\newcommand\CT{{\mathcal{T}}}

\newcommand\ft{{\mathfrak{t}}}

\newcommand \bfG{{\mathbf G}}
\newcommand \bfT{{\mathbf T}}

\newcommand \bfP{{\mathbf P}}
\newcommand \bfU{{\mathbf U}}
\newcommand \bfL{{\mathbf L}}
\newcommand \Gal{{\mathrm {Gal}}}
\newcommand \sfk{{\mathsf k}}

\newcommand \bark{{\bar{\mathsf k}}}

\begin{document}

\title{Wavefront sets of unipotent representations of reductive $p$-adic groups II} 

%\date{\today}

\begin{abstract}
The wavefront set is a fundamental invariant of an admissible representation arising from the Harish-Chandra-Howe local character expansion. In this paper, we give a precise formula for the wavefront set of an irreducible representation of real infinitesimal character in Lusztig's category of unipotent representations in terms of the Deligne-Langlands-Lusztig correspondence. Our formula generalizes the main result of \cite{cmo}, where this formula was obtained in the Iwahori-spherical case. We deduce that for any irreducible unipotent  representation with real infinitesimal character, the algebraic wavefront set is a singleton, verifying a conjecture of M\oe glin and Waldspurger. In the process, we establish new properties of the generalized Springer correspondence in relation to Lusztig's families of unipotent representations of finite reductive groups.
\end{abstract}

\maketitle

\tableofcontents

\section{Introduction}

Let $\mathsf k$ be a nonarchimedean local field of characteristic $0$ and residue field $\FF_q$ of sufficiently large characteristic, and let $\mathbf{G}$ be a connected reductive algebraic group defined over $\sfk$, inner to split. Fix an algebraic closure $\bar{\mathsf{k}}$ of $\mathsf{k}$ and let $K \subset \bar{\mathsf{k}}$ be the maximal unramified extension of $\mathsf{k}$ in $\bar{\mathsf{k}}$. The main result in this paper is a formula for the wavefront set of an arbitrary irreducible unipotent representation of $\mathbf{G}(\sfk)$ of real infinitesimal character. To formulate our result, we will need to review some preliminary ideas.

\subsection{Wavefront sets} 

Let $X$ be an irreducible admissible representation of $\mathbf{G}(\sfk)$. The \emph{wavefront set} of $X$ is a fundamental invariant related to the Harish-Chandra-Howe local character expansion. In its classical form, the wavefront set of $X$ is a collection of nilpotent $\bfG(\sfk)$-orbits in the Lie algebra $\mathfrak g(\sfk)$, denoted $\WF(X)$. We will recall the definition in Section \ref{subsec:wavefrontsets}. In this paper, we consider two coarser invariants. The first of these invariants is the \emph{algebraic wavefront set}, denoted $\hphantom{ }^{\bar{\sfk}}\WF(X)$. This is a collection of nilpotent orbits in $\mathfrak{g}(\bark)$, see for example \cite[p. 1108]{Wald18}. The second invariant is a natural refinement $^K\WF(X)$ of $\hphantom{ }^{\bar{\sfk}}\WF(X)$ called the \emph{canonical unramified wavefront set}, defined recently in \cite{okada2021wavefront}. This is a collection of nilpotent orbits $\mathfrak g(K)$ (modulo a certain equivalence relation $\sim_A$). The relationship between these three invariants is as follows: the algebraic wavefront set $\hphantom{ }^{\bar{\sfk}}\WF(X)$ is deducible from the usual wavefront set $\WF(X)$ as well as the canonical unramified wavefront set $^K\WF(X)$. It is not known whether the canonical unramified wavefront set is deducible from the usual wavefront set, but we expect this to be the case.

\subsection{Representations with unipotent cuspidal support}

In \cite{Lu-unip1}, Lusztig defines the notion of an irreducible  \emph{unipotent representation} of $\mathbf{G}(\sfk)$. These are the quotients of the representations which are compactly induced from Deligne-Lusztig unipotent cuspidal representations of the reductive quotients of parahoric subgroups. We denote the set of such $\mathbf{G}(\sfk)$-representations by $\Pi^{\mathsf{Lus}}(\mathbf{G}(\sfk))$. We will now briefly recall the classification of $\Pi^{\mathsf{Lus}}(\mathbf{G}(\sfk))$ (see Section \ref{subsec:LLC} for a more detailed statement). While our main results are independent of isogeny (see in particular Remark \ref{r:real-indep}), we will assume in the introduction that $\bfG$ is of adjoint type. Let $G^{\vee}$ denote the complex points of the Langlands dual group of $\mathbf{G}$. Then $\Pi^{\mathsf{Lus}}(\mathbf{G}(\sfk))$ is in one-to-one correspondence with $G^{\vee}$-conjugacy classe of triples $(s,n,\rho)$ consisting of
\begin{itemize}
    \item a semisimple element $s \in G^{\vee}$,
    \item a nilpotent element $n \in \fg^{\vee}$ such that $\Ad(s)n=qn$, and
    \item an irreducible representation $\rho$ of the component group of the centralizer of $s$ and $n$ in $G^{\vee}$. 
\end{itemize}
We write $X(s,n,\rho)$ for the unipotent representation associated to the triple $(s,n,\rho)$. Conjugating if necessary, we can assume that $s$ belongs to a fixed maximal torus $T^{\vee} \subset G^{\vee}$. There is a canonical decomposition $T^{\vee} =  T^{\vee}_c  T^{\vee}_r$, where $T^{\vee}_c$ is the maximal compact subgroup of $T^{\vee}$ and $T^{\vee}_r$ is a real vector group. The representation $X(s,n,\rho)$ is said to have \emph{real infinitesimal character} if $s$ belongs to $T^{\vee}_r$.

\subsection{Aubert-Zelevinsky duality}

In \cite{Au}, Aubert defines an involution $X \mapsto \mathrm{AZ}(X)$ on the Grothendieck group of finite-length smooth $\mathbf{G}(\sfk)$-representations (generalizing an involution defined by Zelevinsky in \cite{zelevinsky-induced}). For details, we refer the reader to \cite[Section 2.7]{cmo}. After a sign normalization depending only on the Bernstein component, this involution takes irreducibles to irreducibles and preserves the set $\Pi^{\mathsf{Lus}}(\mathbf{G}(\sfk))$ of unipotent representations. Moreover, if $(s,n,\rho)$ is the parameter of a unipotent representation, then $\mathrm{AZ}(X(s,n,\rho)) = X(s,n',\rho')$.

\subsection{Main results}

Let $G$ be the complex connected reductive algebraic group associated to $\mathbf{G}$. Let $\mathcal N_o$ be the set of nilpotent orbits in the Lie algebra $\mathfrak g$ and let $\mathcal N_{o,\bar c}$ be the set of pairs $(\mathbb O,\bar C)$ consisting of a nilpotent orbit $\mathbb O\in \mathcal N_o$ and a conjugacy class $\bar C$ in Lusztig's canonical quotient $\bar A(\mathbb O)$ of the $G$-equivariant fundamental group $A(\OO)$ of $\OO$. Write $\mathcal N^\vee_o$ and $\mathcal N^\vee_{o,\bar{c}}$ for the corresponding sets for $G^\vee$. In \cite[Proposition 10.3]{Spaltenstein}, Spaltenstein defines a duality map
$$d: \cN^{\vee}_o \to \cN_o$$
An orbit $\OO \in \cN_o$ is in the image of $d$ if and only if it is special in the sense of Lusztig. In \cite{Acharduality}, Achar defines a related map
$$d_A: \mathcal N^\vee_{o,\bar c}\rightarrow \mathcal N_{o,\bar c}$$
This map generalizes $d$ in the following sense: for every orbit $\OO^{\vee} \in \mathcal{N}_o$, there is a conjugacy class $C'$ in $\bar{A}(d(\OO^{\vee}))$ such that
$$d_A(\OO^\vee,1)=(d(\OO^\vee),\bar C')$$
Furthermore, $\bar{C}'=1$ if $\OO^{\vee}$ is special. 

There is a natural bijection between $\mathcal{N}_{o,\bar{c}}$ and the set of $\sim_A$-equivalence classes of unramified nilpotent orbits $\mathcal N_o(K)$ of $\bfG(K)$, see \cite[Section 5.1]{okada2021wavefront}. Thus, we can regard $d_A(\OO^\vee, \bar{C})$ as an element of $\mathcal N_o(K)/\sim_A$.

\begin{theorem}[See Theorem \ref{thm:realwf} and Corollary \ref{cor:wfbound} below]\label{t:main} Let $X=X(s,n,\rho)$ be a unipotent representation of $\bfG(\mathsf k)$ with real infinitesimal character and let $\AZ(X) = X(s,n',\rho')$. 
%Suppose in addition that $s'$ is real, in the sense of Remark \ref{r:real}. 
\begin{enumerate}
    \item The canonical unramified wavefront set $^K\WF(X)$ is a singleton, and
\[^K\WF(X) = d_A(\OO^{\vee}_{\AZ(X)},1),\]
where $\OO^{\vee}_{\AZ(X)}$ is the $G^\vee$-orbit of $n'$ and $d_A$ is the duality map defined by Achar. In particular,
\[\hphantom{ }^{\bar{\sfk}}\WF(X) = d(\OO^\vee_{\AZ(X)}).\]
\item Suppose $X=X(q^{\frac 12 h^\vee},n,\rho)$ where $h^\vee$ is the neutral element of a Lie triple attached to a nilpotent orbit $\OO^\vee\subset \mathfrak g^\vee$. Then
\[d_A(\OO^{\vee}, 1) \leq_A \hphantom{ } ^K\WF(X),
\]
where $\leq_A$ is the partial order defined by Achar. In particular,
\[d(\OO^{\vee}) \le \hphantom{ }^{\bar{\sfk}}\WF(X).\]
\end{enumerate}
\end{theorem}

This result is a vast generalization of \cite[Theorem 1.2.1]{cmo} where a similar result was proved in the case when $\bfG$ is split and $X$ is Iwahori-spherical. When $\bfG$ is inner to $\mathsf{SO}(2n+1)$ and $\AZ(X)$ is tempered, a formula for $\hphantom{ }^{\bar{\sfk}}\WF(X)$ was previously obtained by Waldspurger in \cite{Wald18}. 

We note that the class of representations considered in Theorem \ref{t:main} is probably the most general setting for a result of this form. If $X$ is unipotent of \emph{non-real} infinitesimal character, then both formulas in (1) may not hold. 

\begin{example}
    Let $G^\vee = E_7$, let $T^\vee$ be a maximal torus, and le $s\in T^\vee$ be a semisimple element such that the centralizer $C_{G^{\vee}}(s)$ of $s$ is a subgroup of type $A_1+D_6$ (there are two such $s$ up to $W$-conjugacy).
    Let $u$ be a unipotent element in the orbit $(2)\times (5,3^2,1)$ of $C_{G^\vee}(s)$.
    Then $A(s,u)$ is a cyclic group of order 2 and $u$ belongs to the orbit $A_3+A_2+A_1$ of $G^\vee$.
    Let $\rho$ be the sign representation of $A(s,u)$.
    We have that $d(A_3+A_2+A_1) = A_4+A_2$.
    However if $\AZ(X)$ has parameters $(sq^{\frac12h^\vee},u,\rho)$ then $X$ has wavefront set $A_4+A_1$ which is the special orbit directly below $d(A_3+A_2+A_1)$.
    This is an example where both equalities in (1) fail (it suffices to show that second fails since the first implies the second).
\end{example}

We note that there is a strictly weaker result which appears to hold for all unipotent representations (including the example above). Let $d_S$ be the duality map (\ref{e:sommers-duality}) defined by Sommers \cite{Sommers2001}. The following conjecture is due to the third named author.

\begin{conj}
    \label{conj:one}
    Let $X$ be a unipotent representation of $\bfG(\sfk)$.
    Then $^K\WF(X)$ is a singleton and 
    $$d_S(\hphantom{ }^K\WF(X)) = \OO^\vee_{\AZ(X)}.$$
\end{conj}

This is a weakening of our main theorem in the following sense: for a representation $X$
$$^K\WF(X) = d_A(\OO^\vee_{\AZ(X)},1) \Leftrightarrow d_S(\hphantom{ }^K\WF(X)) = \OO^\vee_{\AZ(X)} \text{ and } \hphantom{ }^{\bark}\WF(X) = d(\OO^\vee_{\AZ(X)}).$$
%

%SPECULATION
%For some groups the conjecture is in fact equivalent to our main theorem for the following reason
%%
%$$d_S([\OO]) = \OO^\vee, \quad \mathcal N_o(\bark/K)(\OO) \text{ special} \implies [\OO] = d_A(\OO^\vee,1).$$
%%
%This implication is true for example for $PGL(n)$ and $SO(2n+1)$, but fails for $E_7$.
%END SPECULATION

If we leave the world of unipotent representations, further difficulties arise. For one, there is no complete classification of irreducible representations in this more general setting. 
%Moreover, the existence of regular depth-0 supercuspidal representations that are not generic implies the conjecture cannot hold for all depth-0 representations. 
Moreover, it was recently announced in \cite{Tsai2022} that $\hphantom{ }^{\bar{\sfk}}\WF(X)$ may not always be a singleton for positive depth representations.

We conclude with the following remark. 
The study of wavefront sets has traditionally been guided by the question `How can we determine the wavefront set from the L-parameter?'. However since $d_S$ is a many to one map, Conjecture \ref{conj:one} suggests it is more natural to ask `How can we determine the nilpotent part of the L-parameter from the wavefront set of the AZ dual?'. 
%Since there is no LLC for depth-0 representations one might even be tempted to take Conjecture \ref{conj:one} as a \emph{definition} of the nilpotent part of the L-parameter of a depth-0 representation.
Moreover, since $^K\WF(X)$ is a coarsening of $\WF(X)$ one might further wonder how much arithmetic information exactly is encoded in $\WF(X)$? 

\subsection{Structure of paper}

In Section \ref{sec:preliminaries} we collect some preliminaries regarding the structure theory of $p$-adic groups, nilpotent orbits, and wavefront sets. In Section \ref{sec:2}, we recall some facts regarding (graded) affine Hecke algebras, including Lusztig's reduction theorems, the classification of simple modules, analytic deformations, and branching to parahoric subalgebras. In Section \ref{sec:arithmeticgeometric}, we recall the details of Lusztig's arithmetic-geometric correspondence, and use it to obtain a formula for the structure of the parahoric restrictions of a unipotent representation. Our main results, stated and proved in Section \ref{sec:main}, depend on verifying a certain technical property of nilpotent orbits related to the generalized Springer correspondence. We call this property \emph{$x$-faithfulness}; it can be viewed as a generalization of the notion of an \emph{adapted pair} from \cite{lusztig2020families}. In Section \ref{sec:faithful} we check this property in all types. In the process, we establish a new combinatorial characterization of the generalized Springer correspondence for spin groups. In Section \ref{sec:main}, we use $x$-faithfulness to prove our main results. 

We mention that the motivation of \cite{lusztig2020families} is to give a new, purely algebraic, description of the Springer correspondence, via families of unipotent representations for (maximal) parahoric subgroups. We expect that our results in Section \ref{sec:faithful} will play a similar role for the generalized Springer correspondence. 

\subsection{Acknowledgments}

This research was partially supported by EP/V046713/1. The authors thank Ruben La for helpful discussions about the generalized Springer correspondence.

\section{Preliminaries}\label{sec:preliminaries}

\subsection{Notation}\label{subsec:notation}
Let $\mathsf{k}$ be a nonarchimedean local field of characteristic $0$ with residue field $\mathbb{F}_q$ of sufficiently large characteristic, ring of integers $\mathfrak{o} \subset \mathsf{k}$, and valuation $\mathsf{val}_{\mathsf{k}}$. Fix an algebraic closure $\bar{\mathsf{k}}$ of $\mathsf{k}$ with Galois group $\Gamma$, and let $K\subset \bar{\mathsf{k}}$ be the maximal unramified extension of $\mathsf{k}$ in $\bar{\mathsf{k}}$. 
Let $\mf O$ be the ring of integers of $K$.
Let $\mathrm{Frob}$ be the geometric Frobenius element of $\mathrm{Gal}(K/\mathsf{k})$, the topological generator which induces the inverse of the automorphism $x\to x^q$ of $\mathbb{F}_q$.

Let $\bfG$ be a connected reductive algebraic group defined over $\sfk$, inner to split, and let $\bfT \subset \mathbf{G}$ be a maximal torus. For any field $F$ containing $\sfk$, we write $\mathbf{G}(F)$, $\mathbf{T}(F)$, etc. for the groups of $F$-rational points. Let $\bfG_{\ad}=\bfG/Z(\bfG)$ denote the adjoint group of $\bfG$.

Write $X^*(\mathbf{T},\bark)$ (resp. $X_*(\mathbf{T},\bark)$) for the lattice of algebraic characters (resp. co-characters) of $\mathbf{T}(\bark)$, and write $\Phi(\mathbf{T},\bark)$ (resp. $\Phi^{\vee}(\mathbf{T},\bark)$) for the set of roots (resp. co-roots). Let
$$\mathcal R=(X^*(\mathbf{T},\bark), \ \Phi(\mathbf{T},\bark),X_*(\mathbf{T},\bark), \ \Phi^\vee(\mathbf{T},\bark), \ \langle \ , \ \rangle)$$
be the root datum corresponding to $\mathbf{G}$, and let $W$ the associated (finite) Weyl group.
Let $G$ be the complex reductive group with the same absolute root datum as $\bfG$ and let $\mathbf{G}^\vee$ be the Langlands dual group of $\bfG$, i.e. the connected reductive algebraic group defined and split over $\ZZ$ corresponding to the dual root datum 
$$\mathcal R^\vee=(X_*(\mathbf{T},\bark), \ \Phi^{\vee}(\mathbf{T},\bark),  X^*(\mathbf{T},\bark), \ \Phi(\mathbf{T},\bark), \ \langle \ , \ \rangle).$$
Set $\Omega=X_*(\mathbf{T},\bark)/\ZZ \Phi^\vee(\mathbf{T},\bark)$. The center $Z(\bfG^\vee)$ can be naturally identified with the irreducible characters $\mathsf{Irr} \Omega$, and dually, $\Omega\cong X^*(Z(\bfG^\vee))$. For $\omega\in\Omega$, let $\zeta_\omega$ denote the corresponding irreducible character of $Z(\bfG^\vee)$.

For details regarding the parametrization of inner twists of $\bfG$, see \cite[\S2]{Vogan1993}, \cite{Kottwitz1984}, \cite[\S2]{Kaletha2016}. %, or \cite[\S1.3]{ABPS2017} and \cite[\S1]{FengOpdamSolleveld2021}. 
For our purposes, it suffices to note that there is a natural bijection between equivalence classes of inner twists of $\mathbf G$ and elements of the Galois cohomology group
\[H^1(\Gamma, \mathbf G_{\ad})\cong H^1(F,\mathbf G_{\ad}(K))\cong\Omega_{\ad}\cong \Irr Z(\bfG^\vee_{\mathsf{sc}}),
\]
Here $\bfG^\vee_{\mathsf{sc}}$ denotes the Langlands dual group of $\bfG_{\ad}$ and $F$ denotes the action of $\mathrm{Frob}$ on $\bfG(K)$. The isomorphism above is defined as follows: for a cohomology class $h$ in $H^1(F, \mathbf G_{\ad}(K))$, let $z$ be a representative cocycle. Let $u\in \bfG_{\ad}(K)$ be the image of $F$ under $z$, and let $\omega$ denote the image of $u$ in $\Omega_{\ad}$. Set $F_\omega=\Ad(u)\circ F$. Then the rational structure of $\bfG$ corresponding to $h$ is given by $F_\omega$.
Write $\bfG^\omega$ for the connected reductive group defined over $\sfk$ such that $\bfG(K)^{F_\omega}=\bfG^\omega(\mathsf k)$.
%denote the corresponding group in the inner class of the split form. 
%Note that $\bfG^{1} = \bfG$ (where we view $\bfG$ as an algebraic group over $\sfk$ for this equality).

\

If $H$ is a complex reductive group and $x$ is an element of $H$ or $\fh$, we write $H(x)=C_H(x)$ for the centralizer of $x$ in $H$, and $A_H(x)$ for the group of connected components of $H(x)$. If $S$ is a subset of $H$ or $\fh$ (or indeed, of $H \cup \fh$), we can similarly define $H(S)$ and $A_H(S)$. We will sometimes write $A(x)$, $A(S)$ when the group $H$ is implicit. 
The subgroups of $H$ of the form $H(x)$ where $x$ is a semisimple element of $H$ are called \emph{pseudo-Levi} subgroups of $H$.

\medskip

Let $\mathcal C(\bfG(\mathsf k))$ be the category of finite-length smooth complex $\bfG(\mathsf k)$-representations and let $\Pi(\mathbf{G}(\mathsf k)) \subset \mathcal C(\bfG(\mathsf k))$ be the set of irreducible objects. Let $R(\bfG(\mathsf k))$ denote the Grothendieck group of $\mathcal C(\bfG(\mathsf k))$.

\subsection{The Bruhat-Tits Building}
\label{sec:btbuilding}

In this section we will recall some standard facts about the Bruhat-Tits building.

Fix a $\omega \in \Omega$ and let $\bfG^\omega$ be the inner twist of $\bfG$ corresponding $\omega$ as defined in the previous section.
Let $\mathcal B(\bfG^\omega,\sfk)$ denote the (enlarged) Bruhat-Tits building for $\bfG^\omega(\sfk)$. 
Let $\mathcal B(\bfG,K)$ denote the (enlarged) Bruhat-Tits building for $\bfG(K)$.
For an apartment $\mathcal A$ of $\mathcal B(\bfG,K)$ and $\Omega\subseteq \mathcal A$ we write $\mathcal A(\Omega,\mathcal A)$ for the smallest affine subspace of $\mathcal A$ containing $\Omega$.
The inner twist $\bfG^\omega$ of $\bfG$ gives rise to an action of the Galois group $\Gal(K/k)$ on $\mathcal B(\bfG,K)$ and we can (and will) identify $\mathcal B(\bfG^\omega,\sfk)$ with the fixed points of this action.
We use the notation  $c\subseteq \mathcal B(\bfG^\omega,\sfk)$ to indicate that $c$ is a face of $\mathcal B(\bfG^\omega,\sfk)$.
Given a maximal $\sfk$-split torus $\bfT$ of $\bfG^\omega$, write $\mathcal A(\bfT,\sfk)$ for the corresponding apartment in $\mathcal B(\bfG^\omega,\sfk)$.
For a face $c\subseteq \mathcal B(\bfG^\omega,\sfk)$ there is a group $\bfP_c^\dagger$ defined over $\mf o$ such that $\bfP_c^\dagger(\mf o)$ identifies with the stabiliser of $c$ in $\bfG(k)$. There is an exact sequence
\begin{equation}\label{eq:parahoricses}
    1 \to \bfU_c(\mf o) \to  \bfP_c^\dagger(\mf o) \to  \bfL_c^\dagger(\mathbb F_q) \to 1,
\end{equation}
where $\bfU_c(\mf o)$ is the pro-unipotent radical of $\bfP_c^\dagger(\mf o)$ and $\bfL_c^\dagger$ is the reductive quotient of the special fibre of $\bfP_c^\dagger$.
Let $\bfL_c$ denote the identity component of $\bfL_c^\dagger$, and let $\bfP_c$ be the subgroup of $\bf P_c^\dagger$ defined over $\mf o$ such that $\bfP_c(\mf o)$ is the inverse image of $\bfL_c(\mathbb F_q)$ in $\bfP_c^\dagger(\mf o)$.
We also write $\bfT$ for the well defined split torus scheme over $\mf o$ with generic fibre $\bfT$.
This scheme $\bfT$ defined over $\mf o$ is a subgroup of $\bfP_c$ and the special fibre of $\bfT$, denoted $\bar\bfT$, is a maximal torus of $\bfL_c$.
For $c$ viewed as a face of $\mathcal B(\bfG,K)$, the stabiliser of $c$ in $\bfG(K)$ identifies with $\bfP_c^\dagger(\mf O)$.
It has pro-unipotent radical $\bfU_c(\mf O)$ and $\bfP_c^\dagger(\mf O)/\bfU_c(\mf O) = \bfL_c^\dagger(\overline{\mathbb F}_q)$.
For $c$ a face lying in $\mathcal B(\bfG^\omega,\sfk)\subseteq \mathcal B(\bfG,K)$, $F_\omega$ stabilises $\bfP_c(\mf O)$ and induces a Frobenius on $\bfL_c(\overline{\mathbb F}_q)$.
The group $\bfL_c(\mathbb F_q)$ consists of the fixed points of this Frobenius.
The groups $\bfP_c(\mf o)$ obtained in this manner are called ($\sfk$-)\emph{parahoric subgroups} of $\bfG^\omega$.
When $c$ is a chamber, then we call $\bfP_c(\mf o)$ an \emph{Iwahori subgroup} of $\mathbf{G}$. 

For the purposes of this paper, it will be convenient to fix a maximal $\sfk$-split torus $\bfT$ of $\bfG^\omega$ and a maximal $K$-split torus $\bfT_1$ of $\bfG^\omega$ containing $\bfT$ and defined over $\sfk$.
We have that $\mathcal A(\bfT,\sfk) = \mathcal A(\bfT_1,K)^{\Gal(K/k)}$.
Write $\Phi(\bfT_1,K)$ (resp. $\Psi(\bfT_1,K)$) for the set of roots of $\bfG(K)$ (resp. affine roots) with respect to $\bfT_1(K)$. 
For $\psi\in \Psi(\bfT,\sfk)$ write $\dot\psi\in \Phi(\bfT,\sfk)$ for the gradient of $\psi$, and
$W=W(\bfT_1,K)$ for the Weyl group of $\bfG(K)$ with respect to $\bfT_1(K)$.
We will also fix a $\Gal(K/\sfk)$-stable chamber $c_0$ of $\mathcal A(\bfT_1,K)$ and a special point $x_0\in c_0$.
Let $\widetilde W=W\ltimes X_*(\mathbf{T}_1,K)$ be the (extended) affine Weyl group. 
The choice of special point $x_0$ of $\mathcal B(\bfG,K)$ fixes an inclusion $\Phi(\bfT_1,K)\to \Psi(\bfT_1,K)$ and an isomorphism between $\widetilde W$ and $N_{\bfG(K)}(\bfT_1(K))/\bfT_1(\mf O^\times)$.
Write 
\begin{equation}
    \widetilde W\to W, \qquad w\mapsto \dot w
\end{equation}
for the natural projection map.
For a face $c\subseteq \mathcal A(\bfT_1,K)$ let $W_c$ be the subgroup of $\widetilde{W}$ generated by reflections in the hyperplanes through $c$.
The special fibre of $\bfT_1$ (as a scheme over $\mf O$) which we denote by $\overline{\bfT}_1$, is a split maximal torus of $\bfL_c(\overline{\mathbb F}_q)$.
Write $\Phi_c(\bar\bfT_1,\overline{\mathbb F}_q)$ for the root system of $\bfL_c$ with respect to $\bar\bfT_1$.
Then $\Phi_c(\bar\bfT_1,\overline{\mathbb F}_q)$ naturally identifies with the set of $\psi\in\Psi(\bfT_1,K)$ that vanish on $c$, and the Weyl group of $\bar\bfT_1$ in $\bfL_c$ is isomorphic to $W_c$. 

Recall that a choice of $x_0$ fixes an embedding $\Phi(\bfT_1,K)\to \Psi(\bfT_1,L)$.
If we fix a set of simple roots $\Delta \subset \Phi(\bfT_1,K)$, this embedding determines a set of extended simple roots $\tilde\Delta\subseteq \Psi(\bfT_1,K)$.
When $\Phi(\bfT_1,K)$ is irreducible, $\tilde\Delta$ is just the set $\Delta\cup\{1-\alpha_0\}$ where $\alpha_0$ is the highest root of $\Phi(\bfT_1,K)$ with respect to $\Delta$.
When $\Phi(\bfT_1,K)$ is reducible, say $\Phi(\bfT_1,K) = \cup_i\Phi_i$ where each $\Phi_i$ is irreducible, then $\tilde\Delta = \cup_i\tilde\Delta_i$ where $\Delta_i = \Phi_i \cap \Delta$.
Fix $\Delta$ so that the chamber cut out by $\tilde\Delta$ is $c_0$.
Let
$$\bfP(\tilde\Delta) := \{J\subsetneq \tilde\Delta: J\cap \tilde\Delta_i\subsetneq \tilde\Delta_i,\forall i\}.$$
Each $J\in \bfP(\tilde\Delta)$ cuts out a face of $c_0$ which we denote by $c(J)$.
In particular $c(\Delta) = x_0$.
Note that since $\Omega \simeq \widetilde W/W\ltimes\ZZ\Phi(\bfT_1,K)$ (recall $\bfG$ is semisimple), and $W\ltimes \ZZ\Phi(\bfT_1,K)$ acts simply transitively on the chambers of $\mathcal A(\bfT_1,K)$, the action of $\widetilde W$ on $\mathcal A(\bfT_1,K)$ induces an action of $\Omega$ on the faces of $c_0$ and hence on $\tilde\Delta$ (and $\bfP(\tilde\Delta)$).
For $\omega\in \Omega$ let $\sigma_\omega$ denote the corresponding permutation of $\tilde\Delta$.
%The faces of $c_0$ that intersect $\mathcal B(\bfG^\omega,\sfk)$ are exactly those that correspond to $J\in \bfP(\tilde\Delta)$ with $\sigma_\omega(J) = J$.
Let
$$\bfP^\omega(\tilde\Delta) := \{J\in \bfP(\tilde\Delta) \mid \sigma_\omega(J) = J\}$$
and let $c_0^\omega$ be the chamber of $\mathcal B(\bfG^\omega,\sfk)$ lying in $c_0$.
The set $\bfP^\omega(\tilde\Delta)$ is an indexing set for the faces of $c_0^\omega$.
For $J\in \bfP^\omega(\tilde\Delta)$ write $c^\omega(J)$ for the face of $c_0^\omega$ corresponding to $J$.
The face $c^\omega(J)$ lies in $c(J)$.
Moreover for $J,J'\in \bfP^\omega(\tilde\Delta)$ (resp. $\bfP(\tilde\Delta)$) we have $J\subseteq J'$ if and only if $\overline{c^\omega(J)}\supseteq c^\omega(J')$ (resp. $\overline{c(J)}\supseteq c(J')$).
After this section we will write $I$ for $\tilde \Delta$ and $I_0$ for $\Delta$.

\subsection{Nilpotent orbits}
\label{sec:nil}
Let $\mathcal N$ be the functor which takes a field $F$ to the set of nilpotent elements in $\mf g(F)$, and let $\mathcal N_o$ be the functor which takes $F$ to the set of adjoint $\bfG(F)$-orbits on $\mathcal N(F)$. When $F$ is $\sfk$ or $K$, we view $\mathcal N_o(F)$ as a partially ordered set with respect to the closure ordering in the topology induced by the topology on $F$.
When $F$ is algebraically closed, we view $\mathcal N_o(F)$ as a partially ordered set with respect to the closure ordering in the Zariski topology.
To simplify the notation, we will write $\mathcal N(F'/F)$ (resp. $\mathcal N_o(F'/F)$) for $\mathcal N(F\to F')$ (resp. $\mathcal N_o(F\to F')$) where $F\to F'$ is a morphism of fields.
For $(F,F')=(\sfk,K)$ (resp. $(\sfk,\bark)$, $(K,\bark)$), the map $\mathcal N_o(F'/F)$ is strictly increasing (resp. strictly increasing, non-decreasing).
We will simply write $\mathcal N$ for $\mathcal N(\CC)$ and $\mathcal N_o$ for $\mathcal N_o(\CC)$.
In this case we also define $\mathcal N_{o,c}$ (resp. $\mathcal N_{o,\bar c}$) to be the set of all pairs $(\OO,C)$ such that $\OO\in \mathcal N_o$ and $C$ is a conjugacy class in the fundamental group $A(\OO)$ of $\OO$ (resp. Lusztig's canonical quotient $\bar A(\OO)$ of $A(\OO)$, see \cite[Section 5]{Sommers2001}). There is a natural map 
\begin{equation}
    \mf Q:\mathcal N_{o,c}\to\mathcal N_{o,\bar c}, \qquad (\OO,C)\mapsto (\OO,\bar C)
\end{equation}
where $\bar C$ is the image of $C$ in $\bar A(\OO)$ under the natural homomorphism $A(\OO)\twoheadrightarrow \bar A(\OO)$. There are also projection maps $\pr_1: \cN_{o,c} \to \cN_o$, $\pr_1: \cN_{o,\bar c} \to \cN_o$. We will typically write $\mathcal N^\vee$, $\mathcal N^\vee_o, \cN^{\vee}_{o,c}$, and $\cN^{\vee}_{o,\bar c}$ for the sets $\mathcal N$, $\mathcal N_o, \cN_{o,c}$, and $\cN_{o,\bar c}$ associated to the Langlands dual group $G^\vee$. When we wish to emphasise the group we are working with we include it as a superscript e.g. $\mathcal N_o^{\bfG^\omega}$.
Note that since $\bfG^\omega$ splits over $K$ we have that $\mathcal N_o^{\bfG}(F) = \mathcal N_o^{\bfG^\omega}(F)$ for field extensions $F$ of $K$. The following well-known result is due to Pommerening, Spaltenstein, Gerstenhaber, and Hesselink.

\begin{lemma}
    Let $F$ be algebraically closed of sufficiently large characteristic.
    Then there is canonical isomorphism of partially ordered sets $\Lambda^F:\mathcal N_o^\bfG(F)\xrightarrow{\sim}\mathcal N_o$.
\end{lemma}

\begin{proof}
Pommerening establishes a canonical bijection $\Lambda^F:\mathcal N_o^\bfG(F)\xrightarrow{\sim}\mathcal N_o$ in \cite[Corollary 3.5]{Pommerening} and \cite[Theorem 1.5]{Pommerening2}. In classical types, it is well-known that this bijection is order-preserving (the partial order is given by the dominance order on partitions on both sides, see \cite{Gerstenhaber} and \cite{Hesselink}). In exceptional types, one sees that $\Lambda^F$ is order-preserving by examining the diagrams in \cite[Chapter 4]{Spaltenstein}.
\end{proof}

There are several duality maps relating the sets $\cN_o$, $\cN_{o,c}$, $\cN_{o,\bar{c}}$ and their dual group counterparts.

\begin{enumerate}
    \item Write
    \begin{equation}\label{eq:dBV}
    d: \cN_o \to \cN_o^{\vee}, \qquad d: \cN_o^{\vee} \to \cN_o.
    \end{equation}
    %
    %for the \emph{Lusztig-Spaltenstein-Barbasch-Vogan duality maps} (see \cite[Appendix A]{BarbaschVogan1985}). 
    for the duality maps defined by Spaltenstein (\cite[Proposition 10.3]{Spaltenstein}), Lusztig (\cite[\S13.3]{Lusztig1984}), and Barbasch-Vogan (\cite[Appendix A]{BarbaschVogan1985}).
    \item Write 
    \begin{equation}\label{e:sommers-duality}
        d_S: \cN_{o,c} \twoheadrightarrow \cN^{\vee}_o, \qquad d_S: \cN^{\vee}_{o,c} \twoheadrightarrow \cN_o
    \end{equation}
    for the duality maps defined by Sommers (\cite[Section 6]{Sommers2001}).
    \item Write 
    \begin{equation}
        d_A: \cN_{o,\bar c} \to \cN^{\vee}_{o,\bar c}, \qquad d_A: \cN^{\vee}_{o,\bar c} \to \cN_{o,\bar c}
    \end{equation}
    for the duality maps defined by Achar (\cite[Section 1]{Acharduality}). 
\end{enumerate}

Write $\leq_A$ for Achar's preorder on $\mathcal N_{o,c}$ \cite[Introduction]{Acharduality} and $\sim_A$ for the equivalence relation on $\cN_{o,c}$ induced by this pre-order, i.e. 
$$(\OO_1,C_1) \sim_A (\OO_2,C_2) \iff (\OO_1,C_1) \leq_A (\OO_2,C_2) \text{ and } (\OO_2,C_2) \leq_A (\OO_1,C_1).$$
Write $[(\OO,C)]$ for the equivalence class of $(\OO,C) \in \cN_{o,c}$. The $\sim_A$-equivalence classes in $\cN_{o,c}$ coincide with the fibres of the projection map $\mf Q:\mathcal N_{o,c}\to\mathcal N_{o,\bar c}$ \cite[Theorem 1]{Acharduality}. So $\le_A$ descends to a partial order on $\mathcal N_{o,\bar c}$, which we simply denote by $\le_A$. The maps $d,d_S,d_A$ are all order reversing with respect to the relevant pre/partial orders. 

\subsubsection{Structure of $\mathcal N_o^\bfG(K)$}
Let $\bfT$ be a maximal $\sfk$-split torus of $\bfG^\omega$, $\bfT_1$ be a maximal $K$-split torus of $\bfG^\omega$ defined over $\sfk$ and containing $\bfT$, and let $x_0$ be a special point in $\mathcal A(\bfT_1,K)$.
In \cite[Section 2.1.5]{okada2021wavefront} the third-named author constructs a bijection
$$\theta_{x_0,\bfT_1}:\mathcal N_o^{\bfG^\omega}(K)\xrightarrow{\sim}\mathcal N_{o,c}.$$
The composition 
$$d^{un}:= d_S\circ \theta_{x_0,\bfT_1}:\cN_o(K) \to \cN_o^\vee$$
is independent of the choice of $x_0$ and natural in $\bfT_1$ \cite[Proposition 2.32]{okada2021wavefront}.

For $\OO_1,\OO_2\in \mathcal N_o(K)$ define $\OO_1\le_A\OO_2$ by
$$\OO_1\le_A \OO_2 \iff \mathcal N_o(\bar k/K)(\OO_1) \le \mathcal N_o(\bar k/K)(\OO_2),\text{ and } d^{un}(\OO_1)\ge d^{un}(\OO_2)$$
and let $\sim_A$ denote the equivalence classes of this pre-order.
\begin{theorem}
    \label{thm:unramclasses}
    \cite[Theorem 2.33]{okada2021wavefront}
    The composition $\mf Q\circ \theta_{x_0,\bfT_1}:\mathcal N_o(K)\to \mathcal N_{o,\bar c}$ descends to a (natural in $\bfT_1$) isomorphism of partial orders
    $$\bar\theta:\mathcal N_o(K)/\sim_A\to \mathcal N_{o,\bar c}$$
    which does not depend on $x_0$.
\end{theorem}

\subsubsection{Lifting nilpotent orbits}
Define 
\begin{equation}
    \mathcal I_o = \{(c,\OO) \mid c\subseteq \mathcal B(\bfG,K),\OO\in\cN_o^{\bfL_c}(\overline{\mathbb F}_q)\}.
\end{equation}
There is a partial order on $\mathcal{I}_o$, defined by
$$(c_1,\OO_1)\le(c_2,\OO_2) \iff c_1=c_2 \text{ and } \OO_1\le\OO_2,$$
and a strictly increasing surjective map \cite[Section 1.1.2]{okada2021wavefront}
\begin{equation}
    \label{eq:lift}
    \mathscr L:(\mathcal I_o,\le)\to(\mathcal N_o(K),\le).
\end{equation}
Recall the groups $G=\bfG(\CC),T=\bfT(\CC)$ from section \ref{subsec:notation}. We call a pseudo-Levi subgroup $L$ of $G$ \emph{standard} if it contains $T$, and we write $Z_L$ for its center. Let $\mathcal A = \mathcal A(\bfT_1,K)$.

Recall from the end of section \ref{sec:btbuilding} the definitions of $\Delta,\tilde\Delta,c_0,\bfP(\tilde\Delta)$ and $c(J)$.
Note that the definitions of $c_0$ and $c(J)$ depend on a choice of $x_0$ and $\Delta$.
Let $L_J$ denote the pseudo-Levi subgroup of $G$ generated by $T$ and the root groups corresponding to $\dot \alpha$ for $\alpha\in J$.
Then $L_J$ is the complex connected reductive group with the same absolute root datum as $\mathbf L_{c(J)}$.
Thus Lemma \ref{lem:Noalgclosed} gives rise to an isomorphism 
$$\Lambda_{J}^{\overline{\mathbb F}_q}:\mathcal N_o^{\bfL_{c(J)}(\mf o)}(\overline{\mathbb F}_q)\to \mathcal N_o^{L_J}(\CC).$$
%
%By \cite[Lemma 2.21]{okada2021wavefront}, the group $L_c$ and isomorphism $\Lambda_c^{\barF_q}$ do not depend on $x_0$.
Define 
\begin{align}
    \mathcal{I}_{x_0,\tilde\Delta}&=\{(J,\OO) \mid J\in \bfP(\tilde\Delta), \ \OO\in\mathcal{N}_o^{\bfL_{c(J)}}(\overline{\mathbb F}_q)\}, \\
    \mathcal{K}_{\tilde\Delta}&=\{(J,\OO) \mid J\in \bfP(\tilde\Delta), \ \OO\in\mathcal{N}_o^{L_{J}}(\CC)\}.
\end{align}
The map
$$\iota_{x_0}:\mathcal I_{x_0,\tilde\Delta}\to\mathcal K_{\tilde\Delta},\quad (J,\OO)\mapsto(J,\Lambda_{J}^{\overline{\mathbb F}_q}(\OO))$$
is a bijection.
Let
\begin{equation}
    \mathbb L:\mathcal K_{\tilde\Delta} \to \mathcal N_{c,o}
\end{equation}
be the map \cite[Corollary 2.19]{okada2021wavefront}.
By \cite[Theorem 2.1.7]{cmo} the diagram 
\begin{equation}
    \begin{tikzcd}[column sep = large]
        \mathcal I_{x_0,\tilde\Delta} \arrow[r,"\iota_{x_0}"',"\sim"] \arrow[d,two heads,"\mathscr L"] & \mathcal K_{\tilde\Delta}  \arrow[d,two heads,"\mathbb L"] \\
        \mathcal N_o(K) \arrow[r,"\theta_{x_0,\bfT_1}"',"\sim"] & \mathcal N_{o,c}
    \end{tikzcd}
\end{equation}
commutes.
Define 
$$\overline{\mathbb L}=\mf Q\circ \mathbb L.$$
This map can be computed using the algorithms in \cite[Section 3.4]{Acharduality}.

\subsection{Wavefront sets}\label{subsec:wavefrontsets}
Let $X$ be an admissible smooth representation of $\bfG^\omega(\sfk)$ and let $\Theta_X$ be the character of $X$.
Recall that for each nilpotent orbit $\OO\in \mathcal N_o^{\bfG^\omega}(\sfk)$ there is an associated distribution $\mu_\OO$ on $C_c^\infty(\mf g^\omega(\sfk))$ called the \emph{nilpotent orbital integral} of $\OO$ (\cite{rangarao}).
Write $\hat\mu_\OO$ for its Fourier transform.
Generalizing a result of Howe (\cite{Howe1974}), Harish-Chandra in \cite{HarishChandra1999} shows that there are complex numbers $c_{\OO}(X) \in \CC$ such that
\begin{equation}\label{eq:localcharacter}\Theta_{X}(\mathrm{exp}(\xi)) = \sum_{\OO} c_{\OO}(X) \hat{\mu}_{\OO}(\xi)\end{equation}
for $\xi \in \fg^\omega(\sfk)$ a regular element in a small neighborhood of $0$. The formula (\ref{eq:localcharacter}) is called the \emph{local character expansion} of $X$. 
The \textit{($p$-adic) wavefront set} of $X$ is
$$\WF(X) := \max\{\OO \mid  c_{\OO}(X)\ne 0\} \subseteq \mathcal N_o(\sfk).$$
The \emph{algebraic wavefront set} of $X$ is
$$^{\bar k}\WF(X) := \max \{\mathcal N_o(\bark/\sfk)(\OO) \mid c_{\OO}(X)\ne 0\} \subseteq \mathcal N_o(\bark),$$
see \cite[p. 1108]{Wald18} (warning: in \cite{Wald18}, the invariant $^{\bar k}\WF(X)$ is called simply the `wavefront set' of $X$). 

In \cite[Section 2.2.3]{okada2021wavefront} the third author has introduced a third type of wavefront set for depth-$0$ representations, called
the \emph{canonical unramified wavefront set}. This invariant is a natural refinement of $\hphantom{ }^{\bar{\sfk}}\WF(X)$. We will now define $^K\WF(X)$ and explain how to compute it. 

Recall from Equation \ref{eq:lift} the lifting map $\mathscr L$.
For every face $c \subseteq \mathcal{B}(\mathbf{G})$, the space of invariants $X^{\bfU_c(\mf o)}$ is a (finite-dimensional) $\bfL_c(\mathbb F_q)$-representation. Let $\WF(X^{\bfU_c(\mf o)}) \subseteq \cN^{\bfL_c}_o(\overline{\mathbb F}_q)$ denote the Kawanaka wavefront set \cite{kawanaka} and let 
\begin{equation}
    ^K\WF_c(X) := \{[\mathscr L (c,\OO)] \mid \OO\in \WF(X^{\bfU_{c}(\mf o)})\} \quad \subseteq \cN_o(K).
\end{equation}

%For simplicity we define $^K\WF$ in the spirit of \cite[Theorem 5.4]{okada2021wavefront} instead of the more complicated definition given in the paper.
%We also define it to live on $\mathcal N_o(K)/\sim_A$ instead of on $\mathcal N_o(K)$.

\begin{definition}\label{def:CUWF}
    Let $X$ be a depth-0 representation of $\bfG^\omega(\sfk)$.
    The \emph{canonical unramified wavefront set} of $X$ is
    \begin{equation}
        ^K\WF(X) := \max \{\hphantom{ }^K\WF_c(X) \mid  c\subseteq \mathcal B(\bfG^\omega,\sfk)\} \quad \subseteq \cN_o(K)/\sim_A.
    \end{equation}
\end{definition}

Fix $\bfT,\bfT_1,c_0,x_0,\Delta$ as at the end of section \ref{sec:btbuilding}.
By \cite[Lemma 2.36]{okada2021wavefront} we have that
\begin{equation}
    ^K\WF(X) = \max \{\hphantom{ }^K\WF_{c^\omega(J)}(X) \mid J \in \bfP^\omega(\tilde\Delta)\}.
\end{equation}
We will often want to view $^K\WF(X)$ and $^K\WF_c(X)$ as subsets of $\mathcal N_{o,\bar c}$ using the identification $\bar\theta_{\bfT_1}$ from Theorem \ref{thm:unramclasses}.
We will write 
$$^{K}\WF(X,\CC) := \bar \theta_{\bfT_1}(\hphantom{ }^K\WF(X)), \quad \hphantom{ }^{K}\WF_c(X,\CC) := \bar \theta_{\bfT_1}(\hphantom{ }^K\WF_c(X)).$$
We will also want to view $\hphantom{ }^{\bar{\sfk}}\WF(X)$, which is naturally contained in $\mathcal N_o(\bar k)$, as a subset of $\mathcal{N}_o$. Thus, we will write 
$$\hphantom{ }^{\bar{\sfk}}\WF(X,\CC) := \Lambda^{\bar \sfk}(\hphantom{ }^{\bar{\sfk}}\WF(X)).$$
By \cite[Theorem 2.37]{okada2021wavefront}, if $^{K}\WF(X)$ is a singleton, then $\hphantom{ }^{\bark}\WF(X)$ is also a singleton and 
\begin{equation}
    \label{eq:wfcompatibility}
    \hphantom{ }^{K}\WF(X,\CC) = (\hphantom{ }^{\bark}\WF(X,\CC),\bar C).
\end{equation}
for some conjugacy class $\bar C$ in $\bar A(\hphantom{ }^\bark\WF(X,\CC))$.

\section{Affine Hecke algebras and graded affine Hecke
  algebras}\label{sec:2} 
In this section we will collect some preliminaries regarding affine Hecke algebras and
graded affine Hecke algebras. We follow \cite{Lu-graded,lusztig-cuspidal-1,lusztig-cuspidal-2,lusztig-cuspidal-3} , \cite{BM2}, \cite{barbaschciubo-ajm}, and \cite{solleveld-2012}. 

\subsection{Generic affine Hecke algebras}\label{sec:2.1} Let $\Psi=(X,X^\vee,R,R^\vee)$ be a reduced root
datum. 
Let $W$ be the finite Weyl group associated to $\Psi$. Fix a choice of positive
roots $R^+ \subset R$, with basis $\Pi$ of
simple roots, and let $R^{\vee,+}$, $\Pi^\vee$ be the corresponding
subsets of $R^\vee$. Let
$\ell: W \to \ZZ$ be the length function corresponding to $\Pi$. Let $\widetilde W=W\ltimes X$ be the extended affine Weyl group.

We write $G(\Psi)$ (or
just $G$ if there is no ambiguity) for the complex connected reductive algebraic group associated to $\Psi$. Then
$T:=X^\vee\otimes_\bZ\bC^{\times}$ is a maximal torus in $G$. Let $B\supset
T$ be the Borel subgroup corresponding to $R^+.$ We may identify: 
\[X=\Hom(T,\bC^\times),\qquad X^\vee=\Hom(\bC^\times, T).
\] 

A parameter set for $\Psi$ is a pair of functions $(\lambda,\lambda^*)$,
$$
\begin{aligned}
 &\lambda:\Pi\to \bZ_{\ge 0},&\lambda^*:\{\al\in\Pi:\check\al\in 2X^\vee\}\to \bZ_{\ge 0}, 
\end{aligned}
$$
such that $\lambda(\al)=\lambda(\al')$ and
  $\la^*(\al)=\la^*(\al')$ whenever $\al,\al'$ are $W-$conjugate.

\begin{definition}[Bernstein presentation, {\cite[Section 3]{Lu-graded}}]\label{d:2.1} The affine Hecke algebra
  $\CH^{\lambda,\lambda^*}(\Psi,z)$, or just $\CH(\Psi,z)$, 
associated to the root datum
  $\Psi$ with parameter set $(\lambda,\lambda^*)$ is the
  associative algebra over $\bC[z,z^{-1}]$ with unit ($z$ is an indeterminate), generated by the elements $T_w$, $w\in W$, and $\theta_x$,
  $x\in X$ subject to the relations:
\begin{align}
&(T_{s_\al}+1)(T_{s_\al}-z^{2\lambda(\al)})=0,\text{ for all }\al\in\Pi,\\
&T_wT_{w'}=T_{ww'},\text{ for all }w,w'\in W\text{ such that }\ell(ww')=\ell(w)+\ell(w'),\notag\\
&\theta_x\theta_{x'}=\theta_{x+x'},\text{ for all }x,x'\in X,\\
&\theta_x T_{s_\al}-T_{s_\al}\theta_{s_\al(x)}=(\theta_x-\theta_{s_\al(x)}) 
  (\CG(\al)-1),\text{ where } x\in X, \al\in\Pi,\text{ and } \notag\\
&\CG(\al)=\begin{cases} 
\frac{\theta_\al z^{2\lambda(\al)}-1}{\theta_\al-1}, &\text{ if }
\check\al\notin 2X^\vee,\\
    \frac{(\theta_\al z^{\lambda(\al)+\lambda^*(\al)}-1)(\theta_\al
      z^{\lambda(\al)-\lambda^*(\al)}+1)}{\theta_{2\al}-1}, &\text{ if
    }\check\al\in 2X^\vee.
 \end{cases}
\end{align}
\end{definition}

%\begin{rmk}

%\noindent (1) {The case 
%$\Pi^\vee\cap 2X^\vee\neq\emptyset$ can occur only if $R$ has a factor of type $B$.}

%\noindent (2) If $\lambda(\al)=c$ for all $\al\in \Pi$, and
%$\lambda^*(\al)=c$, for all 
 % $\cha\in 2X^\vee,$ we say that
 % $\CH(\Psi)$ is a Hecke algebra with equal parameters.
%\end{rmk}

Let $\CA=\CA(\Psi,z)$ be the algebra of regular functions on $\bC^\times\times
  T.$  Note that $\cA$ can be identified with the abelian $\bC[z,z^{-1}]-$subalgebra of
$\CH(\Psi,z)$ generated by $\{\theta_x:x\in X\}$, where for
  every $x\in X,$  $\theta_x:T\to \bC^\times$ is defined by
\begin{equation}
\theta_x(y\otimes\zeta)=\zeta^{\langle x,y\rangle},\
y\in X^\vee,~\zeta\in\bC^\times.   
\end{equation}
If we denote by $\CH_W(z)$ the $\bC[z,z^{-1}]-$subalgebra generated by
$\{T_w:w\in W\}$, then 
\begin{equation}\label{eq:2.2.1}
\CH(\Psi,z)=\CH_W(z)\otimes_{\bC[z,z^{-1}]}\CA(\Psi,z),
\end{equation}
 as a $\bC[z,z^{-1}]-$module.

\begin{theorem}[{\cite[Proposition 3.11]{Lu-graded}}]\label{t:2.2} The center of
  $\CH(\Psi,z)$ is $\CZ=\CA(\Psi,z)^W,$ i.e. the
  $W-$invariants in $\CA(\Psi,z).$ 
\end{theorem}

Let $\CH(\Psi)\modd$ denote the category of
finite dimensional $\CH(\Psi,z)$-modules. {By Schur's lemma, every
irreducible module $(\pi,V)$ has a \textit{central character}, i.e. there
is a homomorphism $\chi:\CZ\longrightarrow\bC$ such that
$\pi(z)v=\chi(z)v$ for every $v\in V$ and $z\in \CZ.$} 
By Theorem \ref{t:2.2}, there is a one-to-one correspondence between central characters for $\CH(\Psi,z)$ and $W-$conjugacy  classes $(z_0,s)\in
\bC^\times\times T$. There is a block a decomposition:
\begin{equation}
\CH(\Psi,z)\modd=\prod_{(z_0,s)\in
  \bC^\times\times W\backslash T} 
\CH(\Psi,z)\modd_{(z_0,s)},
\end{equation}
where $\CH(\Psi,z)\modd_{(z_0,s)}$ is the
subcategory of modules with central character  $(z_0,s).$ 
Let $\text{Irr}_{(z_0,s)}\CH(\Psi,z)$ be the set of isomorphism classes of
simple objects in this category. For the remainder of this section, 
we will assume that $z_0$ is a fixed element of $\bR_{>1}.$

A second presentation of the affine Hecke algebra is due to Iwahori and Matsumoto \cite{IM} (and, in fact, predates the Bernstein presentation). Let $W^a=W\ltimes \mathbb Z R\subseteq \widetilde W$ be the affine Weyl group and let $S^a$ denote a set of reflections of $S^a$ containing in particular all $s_\alpha$, $\alpha\in \Pi$. The length function on $W^a$, still denoted by $\ell$, can be extended to a length function on $\widetilde W$. Let $L:S^a\to \mathbb Z_{\ge 0}$ be a $\widetilde W$-invariant function. 

\begin{definition}[Iwahori presentation]\label{d:Iwahori} The affine Hecke algebra
  $\mathcal H^{L}(\Psi,z)$ 
associated to the root datum
  $\Psi$ with parameter function $L$ is the
  associative algebra over $\mathbb C[z,z^{-1}]$ with unit ($z$ is an indeterminate) generated by the elements $T_w$, $w\in \widetilde W$ subject to the relations:
\begin{align}
&(T_{s}+1)(T_{s}-z^{L(s)})=0,\text{ for all }s\in S^a,\\
&T_wT_{w'}=T_{ww'},\text{ for all }w,w'\in \widetilde W\text{ such that }\ell(ww')=\ell(w)+\ell(w').\notag
\end{align}
For the dictionary between the parameter functions $L$ and $(\lambda,\lambda^*)$ as well as the relations between $\theta_x$ and $T_w$, we refer to \cite[Section 3]{Lu-graded}.
\end{definition}

If $J\subsetneq S^a$, denote by $\widetilde W_J$ the stabiliser in $\widetilde W$ of $S^a-J$. Let $\CH_J(\Psi,z)$ the subalgebra of $\CH^{L}(\Psi,z)$ generated by $T_w$, $w\in \widetilde W_J$. This is a finite Hecke algebra with parameters (possibly extended by diagram automorphims).

\subsection{Graded Hecke algebras}  

Fix a $W$-orbit $\CO \subset T$. Following \cite[Section 4]{Lu-graded} and
  \cite[Section 3]{BM2}, we will define a decreasing chain of ideals 
$\CI^1, \CI^2,...$ in $\CA$. Let

\begin{equation}
\CI:=\{f\text{ regular function on }\bC^\times\times T:
f(1,\sigma)=0,\text{ for all }\sigma\in \CO\},
\end{equation}
and let $\CI^k$ be the ideal of functions vanishing on $(1,\CO)$ to at least
order $k.$ Let $\widetilde {\CI}^k:={\CI}^k\CH=\CH\CI^k$, $k\ge 1,$ be
the chain of ideals in $\CH(\Psi)$, generated by the $\CI^k$'s.

\begin{definition}\label{d:2.3}
The graded affine Hecke algebra {$\bH_\CO(\Psi)$} is the
associated graded of the chain of ideals $\dots\supset \widetilde{\CI}^k\supset\dots$ in $\CH(\Psi).$ \end{definition}

Let $\ft=X^\vee\otimes_\bZ \bC$ be the Lie algebra of $T,$ and let
$\ft^*=X\otimes_\bZ\bC$ be the dual space. Extend the pairing
$\langle\ ,\ \rangle$ to $\ft^*\times \ft.$ Let $\bA$
be the algebra of regular functions on $\bC\oplus \ft.$ Note that
$\bA$ can be identified with $\bC[r]\otimes_\bC S(\ft^*),$ where
$S(~)$ denotes the symmetric algebra, and $r$ is an indeterminate. In
the following $\delta$ denotes the delta function.

\begin{theorem}[{{\cite[Proposition 4.4 and \S8.7]{Lu-graded}},\cite[Proposition 3.2]{BM2}}]\label{t:2.3} The graded Hecke algebra
$\bH_{\CO}(\Psi)$ is a $\bC[r]-$algebra generated by
%$\{t_\gamma\}_{\gamma\in\Gamma},$ 
$\{t_w:w\in W\}$,   $S(\ft^*),$
and a set of orthogonal   idempotents   $\{E_{\sigma}:\sigma\in\CO\}$  
subject to the relations 
\begin{equation}\label{eq:2.3.5}
\begin{aligned}
&t_w\cdot t_{w'}=t_{ww'},\ w,w'\in W,\quad \sum_{\sigma\in \CO}E_{\sigma}=1, \quad 
  E_{\sigma'}E_{\sigma''}=\delta_{\sigma',\sigma''} E_{\sigma'},\\  &E_{\sigma}t_{s_\al}=t_{s_\al}E_{s_\al\sigma},\  E_\sigma\omega=\omega E_\sigma,\\
&\omega\cdot t_{s_\al}-t_{s_\al}\cdot s_\al(\omega)=r 
  g(\al)\langle\omega,\check\al \rangle, \text{ where }\al\in \Pi,\ 
  \omega\in \ft^*,\\
& g(\al)=\sum_{\sigma\in\CO}
  E_{\sigma}\mu_{\sigma}(\al), \text{ and }\mu_{\sigma}(\al)=\left\{ \begin{matrix} 0, &\text{ if } s_\al\sigma\neq
    \sigma,\\
                                  2\lambda(\al), &\text{ if
                                  }s_\al\sigma=\sigma,\ \  \check\al\notin
                                  2X^\vee,\\
           \lambda(\al)+\lambda^*(\al)\theta_{-\al}(\sigma),   &\text{ if
                                  }s_\al\sigma=\sigma,\ \  \check\al\in
                                  2X^\vee.                   
\\ \end{matrix}  \right.
\end{aligned} 
\end{equation}
\end{theorem}

%Notice that if $s_\al\sigma=\sigma,$ then
%$\theta_\al(\sigma)=\theta_{-\al}(\sigma)$, or equivalently,
%$\theta_\al(\sigma)\in\{\pm 1\}.$ This implies that for the %parameters
%$\mu_\sigma(\al)$ we have $\mu_\sigma(\al)\in\{0,2\lambda(\al),
%\lambda(\al)-\lambda^*(\al)\}$, for every root $\al\in \Pi.$  In
%particular, in the case of equal parameters Hecke algebra, the only
%possibilities are $\mu_\sigma(\al)\in \{0,2\lambda(\al)\}.$

\begin{rmk}\label{r:2.3} An important special case is when
 $\CO$ is formed of a single ($W-$invariant) element $\sigma.$ 
Then there is only one
idempotent  generator $E_{\sigma}=1$, so it   is suppressed from the notation.
The algebra $\bH_\CO(\Psi)$ is generated by $\{t_w:w\in W\}$
and $S(\ft^*)$ subject to the commutation relation
\begin{equation}\label{e:basic-graded}
\begin{aligned}
&\omega\cdot t_{s_\al}-t_{s_\al}\cdot s_\al(\omega)=r
\mu_\sigma(\al)\langle\omega,\check\al\rangle, \quad \al\in
\Pi,\ \omega\in \ft^*,
\text{ where }\\
&\mu_{\sigma}(\al)=\left\{ \begin{matrix} 2\lambda(\al), &\text{ if
                                  }  \check\al\notin
                                  2X^\vee,\\
           \lambda(\al)+\lambda^*(\al)\theta_{-\al}(\sigma),   &\text{ if
                                  } \check\al\in
                                  2X^\vee.                   
\\ \end{matrix}  \right.
\end{aligned}
\end{equation}
Still assuming that $\sigma$ is $W-$invariant, we have  
$\theta_{\al}(\sigma)\in\{\pm 1\},$ for all $\al\in \Pi$. If 
$\sigma$ is in the center of group $G(\Psi)$, then  $\theta_{\al}(\sigma)=1,$
for all $\al\in\Pi.$ 
\end{rmk}

We will use the notation   $\bH_{\mu_\sigma}(r)$ for the
    graded affine Hecke
  algebra in the
  particular case 
  defined by equations (\ref{e:basic-graded}).
Define 
\begin{equation}\label{eq:2.3.16}
W(\sigma):=\{w\in W: w\sigma=\sigma\}.
\end{equation}  
Then $W\cdot\sigma=\{w_j\cdot \sigma: 1\le j\le n\}$, where
$\{w_1=1,w_2,\dots,w_n\}$ are coset representatives for
$W/{W}(\sigma).$ Then $\{E_{\tau}\}=\{E_{w_j\sigma}: 1\le j\le
n\},$ and from Proposition \ref{t:2.3}, 
$\bH_\CO=\bC[W]\otimes(\CE\otimes \bA)$, as a $\bC[r]-$vector
space, where $\CE$ is the algebra generated by the
$E_{w_j\cdot\sigma}$'s with $1\le j\le n.$

\begin{prop}[{\cite[Proposition 4.5]{Lu-graded}}]\label{p:2.3}
The center of $\bH_\CO$ is $Z=(\CE\otimes\bA)^{W}.$
%This can be identified with $\bA^{C_{W'}(\sigma)}.$
\end{prop}

It follows that the central characters of $\bH_\CO$ are
parameterized by ${W}(\sigma)-$orbits in $\bC\oplus \ft.$ Define the category
$\text{mod}_{(r_0,x)}\bH_\sigma,$ where $r$ acts by $r_0>0$ and
$x\in \ft.$

\medskip

We describe the structure of $\bH_{\CO}(\Psi)$
in more detail. Fix a $\sigma\in\CO\subset T.$  We define a  root datum
$\Psi_\sigma=(X,R_\sigma,X^\vee,R_\sigma^\vee)$ with positive roots
$R_\sigma^+$, defined as follows:
\begin{align}
&R_\sigma=\left\{ \al\in R: \theta_\al(\sigma)=\left\{\begin{matrix}1,
    &\text{ if }\check\al\notin 2X^\vee,\\ \pm 1, &\text{ if
    }\check\al\in 2X^\vee  \end{matrix}\right.\right\},\\
&R_\sigma^+=R_\sigma\cap R^+,\quad R_\sigma^\vee=\{\check\al\in R^\vee: \al\in R_\sigma\}\notag. 
\end{align}
{Note that $\Psi_\sigma$ is the root datum for $(G(\sigma)_0,T)$,
  with positive roots $R_\sigma^+$ with respect to the Borel subgroup $G(\sigma)_0\cap B.$} 

Define
\begin{equation}\label{2.4.4}
\Gamma_\sigma:=\{w\in {W}(\sigma): w(R_\sigma^+)=R_\sigma^+\}.
\end{equation}
Note that $\Gamma_{\sigma}$ acts on the root datum $\Psi_\sigma$ and preserves the set of positive roots. Thus, it defines an extended graded Hecke algebra
%There is a group homomorphism 
%$\Gamma_\sigma\longrightarrow \Aut(G(\sigma)_0,G(\sigma)_0\cap B,T)$
%such that $\mu_\sigma(\gamma\al)=\mu_\sigma(\al)$ for all $\gamma\in \Gamma_\sigma$, so the
%extended  graded Hecke algebra
\begin{equation}\label{eq:2.4.5}
\bH'_{\mu_\sigma}(\Psi_\sigma):=\bH_{\mu_\sigma}(\Psi_\sigma)\rtimes \Gamma_\sigma.
\end{equation}

\medskip
Recall the coset representatives  $\{w_1,\dots, w_n\}$  for
$W/{W}(\sigma)$ from the paragraph after (\ref{eq:2.3.16}). Set
\begin{equation}
E_{i,j}=t_{w_i^{-1}w_j}E_{w_j\sigma}=E_{w_i\sigma}t_{w_i^{-1}w_j},
\text{ for all } 1\le i,j\le n,
\end{equation}
and let $\CM_n$ be the matrix algebra with basis $\{E_{i,j}\}$.

\begin{theorem}[{\cite[Section 8]{Lu-graded}, \cite[Theorem 3.3]{BM2}}]\label{t:lusztig-graded}
There is an algebra isomorphism
$$
%\C H(\Psi)\cong\CM_n\otimes_\bC \C H'(\Psi_\sigma)\text{ %and }
\bH_\CO(\Psi)\cong\CM_n\otimes_\bC
\bH_{\mu_\sigma}'(\Psi_\sigma)=\CM_n\otimes_\bC
(\bH_{\mu_\sigma}(\Psi_\sigma)\rtimes \Gamma_\sigma).
$$ 
\end{theorem}

Since the only irreducible representation of $\CM_n$ is the
$n-$dimensional standard representation, one obtains immediately 
the equivalences of categories: 
\begin{equation}\label{eq:2.4.7}
%\text{mod}_{(z_0,s)} \C H(\Psi)\cong \text{mod}_{(z_0,s)} \C
% H'(\Psi_\sigma)\text{ and }
\text{mod}_{(r_0,x)} \bH_{\CO}(\Psi)\cong \text{mod}_{(r_0,x)}
\bH_{\mu_\sigma}'(\Psi_\sigma). 
\end{equation}

%\begin{rmk} 
%\noindent (1) When  $X^\vee$ is
%generated by $R^\vee$, that is, when $\Psi$ is of simply connected
%type, or more generally, if $X^\vee$ is generated by $R^\vee\cup \frac 12 R^\vee$ (which includes the case of factors of type $B$ as well), then $C_W(\sigma)\subset W_\sigma,$ and so $\Gamma_\sigma=\{1\},$ for every $\sigma\in T.$ In this case, there is no need to consider the extended graded Hecke algebras (\ref{eq:2.4.5}). 
%\noindent (2) When $\sigma$ is $W-$invariant, then $n=1,$ and so
%$\bH_\CO(\Psi)\cong \bH_{\mu_\sigma}'(\Psi_\sigma)=\bH_{\mu_\sigma}(\Psi_\sigma)\rtimes \Gamma_\sigma.$

%\end{rmk}

\subsection{Reduction theorems}\label{sec:2.5} In this section, we recall the relation
between $\CH(\Psi)$ and $\bH_\CO(\Psi)$. 
The torus $T=X^\vee\otimes_\bZ\bC^\times$ admits a polar decomposition
$T=T_c\times T_r,$ where $T_c=X^\vee\otimes_\bZ S^1,$ and
$T_r=X^\vee\otimes_\bZ \bR_{>0}.$ Consequently, every $s\in T$
decomposes uniquely as $s=s_c\cdot s_h,$ with $s_c\in T_c$ and
$s_h\in T_r.$ We call an element $s_c\in T_c$ elliptic, and an element
$s_h\in T_r$ hyperbolic. Similarly, $\ft=X^\vee\otimes_\bZ\bC$
admits a decomposition $\ft=\ft_{i\bR}\oplus\ft_\bR$, where $\ft_{i\bR}=X^\vee\otimes_\bZ i\bR$ and $\ft_\bR=X^\vee\otimes_\bZ \bR.$

There are certain completions of the Hecke algebras which appear in the reduction theorems. The
algebras $\bC[r],\ \mathscr{S},$ and $\bC[r]\otimes\mathscr{S}$ consist
of polynomial functions on $\bC,\ \ft$ and {$\CM:=\bC\oplus\ft$}, respectively. Let 
$\widehat{\bC[r]}$ and $\widehat{\mathscr{S}}$ 
be the  corresponding algebras of holomorphic functions. Let $\mathscr{K}$ and
$\widehat{\mathscr{K}}$ be the fields of rational and meromorphic functions
on $\CM.$ Finally set 
$\widehat\bA:=\bA\otimes_{\mathscr{S}}\widehat{\mathscr{S}}\subset\widehat{\mathscr K},$ 
and
\begin{align}
  \label{eq:compl}
&\widehat{\bH}_\CO:=\bC [W]\otimes(\widehat{\bC[r]}\otimes\widehat \bA),\quad \widehat\bH_\CO({\mathscr{K}}):=\bC [W]\otimes(\CE \otimes\mathscr{K})\supset \widehat\bH_\CO,\notag\\
&\widehat{\bH}_\CO(\widehat{\mathscr{K}}):=\bC [W]\otimes(\CE\otimes\widehat{\mathscr{K}})
\supset \bH_\CO(\mathscr{K}),\widehat\bH_\CO.\notag
\end{align}

\begin{theorem}[{\cite[Proposition 5.2]{Lu-graded} and \cite[Theorem 3.5]{BM2}}]
The map $\iota:\bC[W]\ltimes(\CE\otimes\widehat{\mathscr{K}})\rightarrow\widehat\bH_\CO(\widehat{\mathscr{K}})$, defined by
\[
\iota(E_\sigma)=E_\sigma,\quad \iota(f)=f,\ f\in\widehat{\mathscr{K}},\quad \iota(t_\al)=(t_\al+1)(\sum_{\sigma\in\CO}
{g_\sigma(\al)^{-1}}E_\sigma) -1,
\]
where
{\begin{equation}\label{eq:gal}
g_\sigma(\al)=1+\mu_\sigma(\al)\al^{-1}\in\widehat{\mathscr{K}},
\end{equation}}
is an algebra isomorphism. 
\end{theorem}

For every character $\chi$ of $\CZ$ (the center of $\CH=\CH_W\otimes\CA$),  there is a maximal
ideal $\CJ_\chi=\{z\in \CZ: \chi(z)=0\}$ of $\CZ.$ Consider the
quotients
\begin{equation*}
\CA_\chi=\CA/\CA\cdot\CJ_\chi,\qquad
\CH(\Psi)_\chi=\CH(\Psi)/\CH(\Psi)\cdot\CJ_\chi.
\end{equation*}
Similarly, consider the ideal $\CI_{\overline\chi}$ for every character
$\overline\chi$ of $Z$ in $\bH_\CO$, and define the analogous
quotients. 
Then
\begin{align*}
  \label{eq:2.5.4}
\bA_{\overline\chi}:=\bA/(\bA\cdot\CI_{\overline\chi})=
\widehat\bA/(\widehat{\bA}\cdot {\CI}_{\overline\chi})=\widehat\bA_{\overline\chi},\quad
\bH_{\overline\chi}:=\bH_\CO/(\bH_\CO\cdot\CI_{\overline\chi})=
\widehat\bH_\CO/(\widehat\bH_\CO\cdot\CI_{\overline\chi})=\widehat\bH_{\overline\chi}.
\end{align*}

The map 
\begin{equation*}
\tau:\bC\oplus\ft_\bR\to \bC^\times\times T,\quad (r_0,\nu)\mapsto
(z_0,s)=(e^{r_0},\sigma\cdot e^\nu)
\end{equation*}
is $W(\sigma)-$invariant. It matches the central characters 
\begin{equation}\label{eq:2.5.3} 
\tau:~\overline\chi=W(\sigma)\cdot (r_0,\nu)\longleftrightarrow
\chi=W\cdot (z_0,s).
\end{equation}
Moreover, $\tau$ is a bijection onto the central characters of $\CH$ with elliptic part in $\CO.$

\begin{theorem}[{\cite[Proposition 4.1, Theorem 4.3]{BM2}}]\label{t:graded-reduction}\ 
\begin{enumerate}
    \item The map $\phi:\CA[z,z^{-1}]\longrightarrow\widehat\bC[r]\otimes\widehat\bA$ defined by
\begin{equation}
\begin{aligned}
&\phi(z)=e^r,\quad \phi(\theta_x)=\sum_{\sigma\in\CO}\theta_x(\sigma)E_\sigma e^x,\qquad x\in
X,
\end{aligned}
\end{equation}
is a $\bC$-algebra homomorphism which maps $\CJ_\chi$ to $\CI_{\overline\chi}$ and defines by passage to the quotients an isomorphism $\CA_\chi\cong\bA_{\overline\chi}.$ 

\item The map $\Phi:\CH\longrightarrow\widehat\bH_\CO(\widehat{\mathscr{K}})$ 
defined by 
\begin{equation}
  \label{eq:2.5.5}
\begin{aligned}
&\Phi(a)=\phi(a),\quad \Phi(T_\al+1)=\sum_{\sigma\in\CO} E_\sigma(t_\al+1){\phi(\CG_\al) g_\sigma(\al)^{-1}}    
\end{aligned}
\end{equation}
 {with $\CG_\al$ as in Definition \ref{d:2.1}, and $g_\sigma(\al)$ as in
 (\ref{eq:gal}),} induces an isomorphism  $\CH(\Psi)_\chi\cong(\bH_\CO(\Psi))_{\overline\chi}.$ 
%The map $\Phi$ depends on {$(r_0,\nu)\in\bC\oplus\ft_\bR)$. }

\item Assume
  $\sigma\in T_c.$ Let   $\chi=W\cdot (e^{r_0},\sigma\cdot e^\nu)$,
  $\overline\chi=W(\sigma)\cdot   (r_0,\nu)$ be as in
  (\ref{eq:2.5.3}), with $(r_0,\nu)\in   \bC\oplus\ft_\bR.$ The
  map (\ref{eq:2.5.5})
\[
\Phi:\CH(\Psi)_\chi\xrightarrow{\cong}
  \bH_\CO(\Psi)_{\overline\chi}
\]
  is
  an analytic isomorphism in $(r_0,\nu)\in \bC\oplus\ft_\bR$ between two families of analytic $\CC$-algebras, in the sense of \cite[Definition 4.2]{BM2}.
  \end{enumerate}
\end{theorem} 

If $\sigma\in T_c$, denote by $r_\sigma$ the bijection defined by $\Phi$ between simple $\bH_{\mu_\sigma}(\Psi_\sigma)\rtimes \Gamma_\sigma$-modules with real central character and simple $\mathcal H(\Psi,z_0)$-modules with central character having compact part in $W\cdot \sigma$.

\subsection{Tempered modules}
Let $V$ be a
{$\CH(\Psi,z)-$module} on which $z$ acts by $z_0\in
\bR_{>1}$. For every $t\in T,$ define a $\CA$-generalized eigenspace
of $V$: 
\begin{equation}
V_t=\{v\in V: \text{ for all } x\in X,\ (x(t)-\theta_x)^kv=0, \text{
  for some } k\ge 0\}.
\end{equation}
We say that $t$ is a weight of $V$ if $V_t\neq\{0\}.$ Let $\Phi(V)$
denote the set of weights of $V.$ We have $V=\oplus_{t\in\Phi(V)}
V_t$ as $\CA$-modules.

\begin{definition}\label{d:2.7}
We say that $V$ is tempered if, for all $x\in X^+:=\{x\in X:\langle
x,\check\al\rangle\ge 0, \text{ for all }\al\in R^+\}$, and all
$t\in\Phi(V),$  $|x(t)|\le 1$ holds.
\end{definition}

Let $\overline V$ be a {
  $\bH_{\mu_\sigma}'(\Psi_\sigma)-$module} on which $r$ acts by 
$r_0>0.$ For every $\nu\in\ft,$ define a $\bA$-generalized eigenspace
of $\overline V$,
\begin{equation}
\overline V_\nu=\{v\in \overline V: \text{for all }\omega\in S(\ft^*),\ (\omega(\nu)-\omega)^kv=0, \text{ for some }k\ge 0\}.
\end{equation}
We say that $\nu$ is a weight of $\overline V$ if $\overline
V_\nu\neq\{0\}.$ Let $\Phi(\overline V)$ denote the set of weights
of $\overline V.$ We have 
$
\overline V=\oplus_{\nu\in\Phi(\overline V)} V_\nu,
$ as $\bA$-modules.

\begin{definition}
We say that $\overline V$ is tempered if, for all $\omega\in
\bA^+:=\{\omega\in S(\ft^*): \langle\omega,\check\al\rangle\ge
0,\text{ for all }\al\in R^+\}$, and all $\nu\in\Phi(\overline V)$,
we have $\langle\omega,\nu\rangle\le 0.$
\end{definition}

The two notions of tempered are naturally related, see for example \cite[Sections 6, 7]{BM2}: under the bijection $r_\sigma$ from 
  Theorem \ref{t:graded-reduction}, the tempered modules correspond.

\subsection{Deformations of standard modules}\label{sec:OS-deform}  We make use of certain continuous deformations of affine Hecke algebras and their Schwarz completions due to Solleveld \cite[\S4]{solleveld-2012}, based also on an idea of Opdam \cite{opdam-spectral}. Let $\mathcal H(\Psi,z_0)$ be the specialisation of the generic affine Hecke algebra $\mathcal H(\Psi,z_0)$
at $z=z_0>1$. There is a family of additive functors

\begin{equation}\label{e:OS-deform}
\widetilde\sigma_\epsilon:\text{mod} \mathcal H(\Psi,z_0)\to \text{mod}\mathcal H(\Psi,z_0^\epsilon),\qquad  \widetilde\sigma_\epsilon(\pi,V)=(\pi_\epsilon,V), \qquad \epsilon\in[-1,1]    
\end{equation}
such that for each $w\in \widetilde W$, the map
\begin{equation*}
  [-1,1]\to \operatorname{End}(V),\ \epsilon\mapsto \pi_\epsilon(T_w)  
\end{equation*}
is analytic in $\epsilon$. The functors $\widetilde\sigma_\epsilon$ are defined by
\[\pi_\epsilon=\pi\circ \rho_\epsilon,
\]
see \cite[Proposition 4.1.2]{solleveld-2012}, where $\rho_\epsilon$ is a scaling map (which is an isomorphism of topological algebras when $\epsilon\neq 0$) between certain analytic completions of $\mathcal H(\Psi,z_0^\epsilon)$ and $\mathcal H(\Psi,z_0)$, such that the maps $\epsilon\mapsto \rho_\epsilon(T_w)$ are analytic for all $w\in \widetilde W$.
We summarize the properties of $\widetilde\sigma_\epsilon$. 

\begin{theorem}[{\cite[Corollary 4.2.2  ]{solleveld-2012} }] The functors $\widetilde\sigma_\epsilon$:
  \begin{enumerate}
      \item map modules with central character $s\in W\backslash T$ to modules with the same central character;
      \item are equivalences of categories if $\epsilon\neq 0$;
      \item map tempered modules to tempered modules when $\epsilon\ge 0$.
  \end{enumerate}
\end{theorem}

We are particularly interested in the functor $\widetilde\sigma_0$. For finite-dimensional tempered modules, a precise description is available. Let 
\begin{equation}
    \zeta^*: \{\text{simple tempered $\mathcal H(\Psi,z_0)$-modules}\}\longrightarrow \{\text{tempered $\mathbb C[X\rtimes W]$-modules}\}
\end{equation}
denote the restriction of $\widetilde\sigma_0$ to the set of tempered modules. Fix a compact semisimple element $\sigma\in T_c$, and let $r_{\sigma}$ denote the corresponding bijection given after Theorem \ref{t:graded-reduction}. 

\begin{theorem}[{\cite[Theorem 2.3.1, Corollary 4.4.3]{solleveld-2012}}]\label{t:sol-zeta}
  If $\widetilde \pi$ is any simple tempered $\bH_{\mu_\sigma}(\Psi_\sigma)\rtimes \Gamma_\sigma$-module with real central character and $\widetilde \pi\circ r_\sigma$ the corresponding simple tempered $\mathcal H(\Psi,z_0)$-module, then
  \[\zeta^*(\widetilde\pi\circ r_\sigma)=\operatorname{Ind}_{X\rtimes W(\sigma)}^{X\rtimes W}(\mathbb C_\sigma\otimes \widetilde\pi|_{W(\sigma)}).
  \]
  
\end{theorem}

Let $J\subsetneq S^a$ be given. The finite Hecke algebra $\mathcal H_J(\Psi,z_0^\epsilon)$ is semisimple and finite dimensional with the underlying vector space $\mathbb C[\widetilde W_J]$. Since the multiplication varies continuously with $\epsilon$, it follows by Tits' deformation theorem, that the algebras $\mathcal H_J(\Psi,z_0^\epsilon)$ are isomorphic for all $\epsilon$. 
Moreover since the maps $\epsilon\to \pi_\epsilon(T_w)$ are continuous in $\epsilon$, it follows that the multiplicities of the restrictions of $\pi_\epsilon$ to $\mathcal H_J(\Psi,z_0^\epsilon)$ are constant with respect to $\epsilon$. We record these facts in the following corollary.

\begin{cor}\label{c:restr-temp}
Let $\tau$ be a simple $\mathcal H_J(\Psi,z_0)$-module and let $\tau|_{z_0\to 1}$ be the corresponding simple $\mathbb C[\widetilde W_J]$-module. Then, for every finite-dimensional $\mathcal H(\Psi,z_0)$-module $\pi$,
\[\Hom_{\mathcal H_J(\Psi,z_0)}[\tau,\pi]=\Hom_{\widetilde W_J}[\tau|_{z_0\to 1}, \widetilde\sigma_0(\pi)].
\]
In particular, for every simple tempered $\pi=\widetilde\pi\circ r_\sigma$-module as in Theorem \ref{t:sol-zeta},
\begin{equation}
   \Hom_{\mathcal H_J(\Psi,z_0)}[\tau,\pi]=\Hom_{\widetilde W_J}[\tau|_{z_0\to 1}, \zeta^*(\pi)]=\Hom_{\widetilde W_J}[\tau|_{z_0\to 1},\operatorname{Ind}_{X\rtimes W(\sigma)}^{X\rtimes W}(\mathbb C_\sigma\otimes \widetilde\pi|_{W(\sigma)})]. 
\end{equation}
\end{cor}

\subsection{Parabolic subalgebras} For $P\subseteq \Pi$, define
\[R_P=\Pi\cap \mathbb Q P,\quad R_P^\vee=R^\vee\cap \mathbb Q P^\vee,\quad X_P=X/X\cap (P^\vee)^\perp,\quad X_P^\vee=X^\vee\cap \mathbb QP^\vee.
\]
Let $\Psi^P=(X,X^\vee,R_P,R_P^\vee)$ and $\Psi_P=(X_P,X_P^\vee,R_P,R_P^\vee)$ be the corresponding root data. The parameters $\lambda,\lambda^*$ of $\Psi$ give rise by restriction to certain parameters for $\Psi^P$, $\Psi_P$, see \cite[\S1.4, (1.3)]{solleveld-2012}. Denote the resulting affine Hecke algebras by $\mathcal H_P=\mathcal H(\Psi_P,z_0)$ and $\mathcal H^P=\mathcal H(\Psi^P,z_0)$. The algebra $\mathcal H^P$ is a subalgebra of $\mathcal H(\Psi,z_0)$ and there is a natural surjective algebra homomorphism 
\[\mathcal H^P\to \mathcal H_P,\quad \theta_x\mapsto \theta_{x_P},\ T_w\mapsto T_w,\ x\in X,\ w\in W,
\]
arising from the natural projection $X\to X_P$, $x\mapsto x_P$.

For every $t\in T^P:=\Hom[X^P,\mathbb C^\times]$, we may define an algebra automorphism
\[\chi_t: \mathcal H^P\to\mathcal H^P,\quad \chi_t(\theta_x)=t(x)\theta_x,\ \chi_t(T_w)=T_w.
\]
Every finite-dimensional $\mathcal H_P$-module $\pi_P$ can be regarded as an $\mathcal H^P$-module by pullback and we may define the parabolically induced representations
\begin{equation}
    I(P,\pi_P,t)=\operatorname{Ind}_{\mathcal H^P}^{\mathcal H}(\pi_P\circ\chi_t),\quad t\in T^P.
\end{equation}
Since the action varies continuously in $t$, the restriction of $I(P,\pi_P,t)$ to any finite Hecke algebra $\mathcal H_J(\Psi,z_0)$ is independent of $t$. It is clear from the definitions that if $\pi_P$ is a simple tempered module for $\mathcal H_P$, then $I(P,\pi_P,1)=\operatorname{Ind}_{\mathcal H^P}^{\mathcal H}(\pi_P)$ is a (not necessarily simple) tempered module for $\mathcal H$.  Hence Corollary \ref{c:restr-temp} implies at once:

\begin{prop}\label{p:restr-standard}
Let $\pi_P$ be a simple tempered $\mathcal H_P$-module, $t\in T^P$, and $J\subsetneq S^a$.  Then for every simple $\mathcal H_J(\Psi,z_0)$-module $\tau$,
\begin{equation}
   \Hom_{\mathcal H_J(\Psi,z_0)}[\tau,I(P,\pi_P,t)]=\Hom_{\widetilde W_J}[\tau|_{z_0\to 1}, \zeta^*(I(P,\pi_P,1))],
\end{equation}
which can be computed using Corollary \ref{c:restr-temp}.
\end{prop}

\subsection{Geometric graded Hecke algebras}\label{sec:graded} 
We analyze the graded Hecke algebra $\bH=\bH_{\mu_\sigma}(\Psi_\sigma)$ defined by
the relations in (\ref{e:basic-graded}) for the case $\sigma=1$.  %The lattices $X$ and $X^\vee$ are not explicitly needed
%for the relations, just $\ft:=X^\vee\otimes_\bZ\bC$ and $\ft^*=X\otimes_\bZ\bC.$ Then $\bH=\bH^\mu(\ft^*,R)$ is generated over $\bC[r]$ by $\{t_w:w\in W\}$ and $\omega \in\ft^*$, subject to
%\begin{align}\label{eq:3.3.1}
%&t_w\cdot t_{w'}=t_{ww'},\quad w,w'\in W,\\\notag
%&\omega\cdot\omega'=\omega'\cdot\omega,\quad \omega,\omega'\in\ft^*,\\\notag
%&\omega\cdot t_{s_{\al_i}}-t_{s_{\al_i}}\cdot
%s_{\al}(\omega)=2r\mu_\al\langle\omega,\check\al \rangle,\quad
%\omega\in\ft^*,~ \al\in\Pi.
%\end{align} 
We will consider only the Hecke algebras 
arising by the constructions of \cite{lusztig-cuspidal-1,lusztig-cuspidal-2}. 
 Let $G$ be a complex connected
reductive group, with a fixed Borel subgroup $B$,
and maximal torus $A\subset B$. The Lie algebras will be denoted by
the corresponding Gothic letters.

\begin{definition}
A cuspidal triple for $G$ is a triple $(L,\OO,\mathcal E),$ where $L$ is a
Levi subgroup of $G,$ $\OO$ is a nilpotent adjoint $L$-orbit on the Lie
algebra $\fl$, and $\mathcal E$ is an irreducible $G$-equivariant local
system on $\OO$ which is cuspidal in the sense of \cite[Definition 2.4]{lusztig-intersection}. 
\end{definition}
Let $\fL(G)$ denote the set of $G$-conjugacy
classes of cuspidal triples for $G$. %For example, $(A,0,\mathsf{triv})\in \fL(G).$ 
Let us fix $(L,\OO,\mathcal E)\in \fL(G)$, such that $A\subset L,$
and $P=LU\supset B$ is a parabolic subgroup.  Let
$T$ denote the identity component of the center of $L$. Set
$W=N_G(L)/L$, a (finite) Coxeter group.

Let $\bH=\bH(L,\OO,\mathcal E)$ denote the graded affine Hecke algebra attached to $(L,\OO,\mathcal E)$ in \cite[section 2]{lusztig-cuspidal-1}. The
  explicit classification of cuspidal triples when $G$ is simple,
  along with the corresponding values for the parameters of the graded Hecke algebra can
  be found in the tables of \cite[\S2.13]{lusztig-cuspidal-1}.

%\begin{definition}[{\cite[section 2]{lusztig-cuspidal-1}}] Let $\bH(L,\OO,\mathcal E)$ define a
%  graded Hecke   algebra as in (\ref{e:basic-graded}) where:
%\begin{enumerate}
%\item[(i)] $\ft$ is the Lie algebra of $T$;
%\item[(ii)] $W=N_G(L)/L$;
%\item[(iii)] $R$ is the reduced part of the root system given by
 % the nonzero weights of  $\text{ad}(\ft)$ on $\mathfrak g$; it can be
 % identified with the root system of the reductive part of $G(x),$ where $x\in \OO$; 
%\item[(iv)] $R^+$ is the subset of $R$ for which the
%  corresponding weight space lives in $\fu$;
%\item[(v)] the simple roots $\Pi=\{\al_i:i\in I\}$ correspond to the Levi
 % subgroups $L_i$ containing $L$ maximally: $\al_i$ is the unique
 % element in $R^+$ which is trivial on the center of $\fl_i$;
%\item[(vi)] for every simple $\al_i$, $\mu_{\al_i}\ge 2$ is defined to be
 % the smallest integer  such that
%\begin{equation}
%\ad(x)^{\mu_{\al_i}-1}:\fl_i\cap\fu\to \fl_i\cap\fu
%\text{ is zero.} 
%\end{equation} 
%\end{enumerate}
%\end{definition}
%The
 % explicit classification of cuspidal triples when $G$ is simple,
 % along with the corresponding values for the parameters $\mu_\al$ can  be found in the tables of \cite[\S2.13]{lusztig-cuspidal-1}.

Consider  the algebraic varieties
\begin{align}
\widetilde\fg=\{(x,g)\in \fg\times G:  \Ad(g^{-1})x\in \OO+\ft+\fu\},\quad\dot{\fg}=\{(x,gP)\in \fg\times G/P: (x,g)\in \widetilde\fg\},
%\\\notag
%&\ddot{\fg}=\{(x,gP,g'P)\in \fg\times G/P\times G/P:
%  ~(x,gP)\in\dot{\fg},~(x,g'P)\in \dot{\fg}\},
\end{align} 
on which $G\times \bC^\times$ acts via $(g_1,\lambda)$: $x\mapsto
\lambda^{-2}\text{Ad}(g_1)x,$ $x\in \fg,$ and $gP\mapsto g_1gP,$ $g\in
G.$ Let $\pr_P:\widetilde\fg\to\dot\fg$ denote the obvious $G\times\bC^\times$-equivariant map.

If $\CV$ is any $G\times \bC^\times$-stable subvariety of $\fg,$ denote by $\dot{\CV}$ the preimage under the first
projection $\pr_1:\dot\fg\to\fg$. Let $\mathcal N$ denote the variety of nilpotent elements in
$\fg.$ Define \[\CP_n=\pr_2(\dot{\{n\}})=\{gP:~ \Ad(g^{-1})n\in \OO+\fu\},\]
for any $n\in\mathcal N,$ where
$\pr_2:\dot\fg\to G/P$ is the second projection. Define also $\CP_n^s=\{gP\in \CP_n: \Ad(g^{-1})s\in \fp\},$
for any semisimple $s\in \fg$.

Consider the $G\times\bC^\times$-equivariant projection $\pr_{\OO}:\widetilde\fg\to \OO$, $\pr_{\OO}(x,g)=\pr_{\OO}(\Ad(g^{-1})x)$.  Let $\dot{\mathcal E}$ be the
$G\times\bC^\times$-equivariant local system on $\dot\fg$ defined by
the condition $\pr_{\OO}^*(\mathcal E)=\pr_P^*(\dot{\mathcal E}).$ %and let $\dot{\CL}^*$ be its dual local system. 

\medskip

The classification of simple modules
for $\bH$ is in \cite{lusztig-cuspidal-1,lusztig-cuspidal-2,lusztig-cuspidal-3}. 
 Fix a semisimple element $s\in \fg$ and $r_0\in
\bC^\times$, and let $\CT=\CT_{s,r_0}$ be the smallest torus in
$G\times\bC^\times$ whose Lie algebra contains $(s,r_0).$ Let
$\fg_{2r_0}$ be the set of $\CT-$fixed vectors in $\fg,$ namely
\begin{equation}
\fg_{2r_0}=\{x\in\fg:~ [s,x]=2r_0x\}.
\end{equation} 
Let $G(s)\times \bC^\times$ be the centralizer of $(s,r_0)$ in
$G\times \bC^\times.$ The group $G(s)$ is connected (in fact, a Levi subgroup of
$G$).
Denote
\[M(n)=C_{G\times\bC^*}(n)=\{(x,\lambda): \Ad(x)n=\lambda^2 n\},
\]
and let $M(n)^\circ$ be its identity component. 

If $Q$ is an algebraic group, $X$ a $Q$-variety, and $\CL$ a $Q$-equivariant local system on $X$, let $H_\bullet^Q(X,\CL)$ denote the equivariant homology groups as in \cite[\S1]{lusztig-cuspidal-1}. In particular, if $Q=\{1\}$, then
\begin{equation}\label{e:hom-equiv}
H_j(X,\CL):=H^{\{1\}}_j(X,\CL)=H^{2\dim X-j}_c(X,\CL^*)^*.
\end{equation}
%which will provide the link with the generalized Springer correspondence from \cite{lusztig-intersection}.

The construction of standard modules using equivariant homology $H_\bullet^{M(n)^\circ}$ is in \S8 in {\it loc. cit.}, see also \cite[10.7, 10.12]{lusztig-cuspidal-2}). It is shown that for $n\in \fg_{2r_0}$, the space
\[Y_\bH(s,n)=\bC_{(s,r_0)}\otimes_{H^\bullet_{M(n)^\circ}}H_\bullet^{M(n)^\circ}(\CP_n,\dot{\mathcal E})\cong H_\bullet^{\{1\}}(\CP_n^s,\dot{\mathcal E})
\]
carries an action of $A(s,n)\times \bH$. In particular, the groups $H_\bullet(\CP_n,\dot{\mathcal E})$ carry an $A(n)\times W$-action which is compatible via the identification (\ref{e:hom-equiv}), and with the appropriate normalization, with the action defined in \cite{lusztig-intersection}.
 We normalize the $W$-action so that $H_0(\CP_n,\dot{\mathcal E})=\mathsf{triv}$ when $n$ is the {\it smallest} nilpotent orbit that occurs in this setting.

Let $\widehat {A(n)}_\mathcal E$ denote the irreducible representations that occur in the action of $A(n)$ on $H^{\{1\}}(\CP_n,\dot{\mathcal E})$. We have a natural injection $A(s,n)\to A(n)$ and let $\widehat {A(s,n)}_\mathcal E$ denote the irreducible $A(s,n)$-representations that occur in the restrictions of $\widehat {A(n)}_\mathcal E$. 
For every $\rho\in \widehat {A(s,n)}$, define the geometric {\it standard module}
\begin{equation}
Y_\bH(s,n,\rho)=\Hom_{A(s,n)}[\rho,Y_\bH(s,n)].
\end{equation}
By
\cite[Proposition 8.10]{lusztig-cuspidal-1}, $Y_\bH(s,n,\rho)\neq 0$ if and
only if  $\rho\in \widehat {A(s,n)}_\mathcal E.$ 
One can phrase the
classification as follows.

\begin{theorem}[{\cite[\S8.10, 8.14, 8.17]{lusztig-cuspidal-2},\cite[Theorem 1.15]{lusztig-cuspidal-3}}]\label{t:graded-classif}
The standard module $Y_\bH(s,n,\rho),$ $s\in \fg$ semisimple, $n\in\fg_{2r_0}$, $\rho\in \widehat {A(s,n)}_\mathcal E$ 
has a unique irreducible quotient $X_\bH(s,n,\rho).$ 
This induces a natural one-to-one  correspondence
\begin{equation}
  \begin{aligned}
&\textup{Irr}_{r_0}\bH(L,\OO,\mathcal E)\leftrightarrow\{(s,n,\rho): [s,e]=2r_0e,\ s\in \fg \text{  semisimple},~n\in \mathcal N,~ \widehat {A(s,n)}_\mathcal E\}/G.    
  \end{aligned}
\end{equation}
\end{theorem}

 We are interested in the $W-$structure of
standard modules. 

\begin{lemma}[\cite{lusztig-cuspidal-1},10.13]\label{l:graded-W-action} There is an isomorphism of $W$-representations
\begin{equation}
  \begin{aligned}
Y_\bH{(s,n,\rho)}=&\operatorname{Hom}_{A(s,n)}[\rho,H_\bullet(\CP_n^s,\dot{\mathcal E})]=\operatorname{Hom}_{A(s,n)}[\rho,H_\bullet(\CP_n,\dot{\mathcal E})].    
  \end{aligned}
\end{equation}
\end{lemma}

Denote
\begin{equation}\label{eq:zero-hom}
\mu(n,\tilde\rho):=\operatorname{Hom}_{A(n)}[\tilde\rho, H_0(\CP_n,\dot{\mathcal E})],\quad \tilde\rho\in\widehat{A(n)}_{\mathcal E},
\end{equation}
an irreducible $W$-representation. The
correspondence $\widehat {A(n)}_\mathcal E\to \widehat W,$ 
$(n,\tilde\rho)\to \mu(n,\tilde\rho)$ is a bijection. 

The results from \cite[\S6.5]{lusztig-intersection},\cite[\S24.4]{lusztig-sheaves-V}, see also the summary in \cite[Theorem 3.4.1]{barbaschciubo-ajm}, imply the following $W$-character formula:
\begin{equation}
  \operatorname{Hom}_{A(n)}[\tilde\rho, H_\bullet(\CP_n,\dot{\mathcal E})]=\mu(n,\tilde\rho)+\sum_{(n',\tilde\rho')} m_{(n,\tilde\rho),(n',\tilde\rho')} \mu(n',\tilde\rho'),
\end{equation}
where $m_{(n,\tilde\rho),(n',\tilde\rho')}\in \mathbb Z_{\ge 0}$, $\tilde\rho'\in \widehat{A(n')}_{\mathcal E}$, and $n'$ ranges over a set of representatives of $G$-orbits in $\mathcal N$ such that $n\in\overline{G\cdot n'}\setminus {G\cdot n'}$. Combining this with Lemma \ref{l:graded-W-action}, we arrive at a character formula for the standard modules.

\begin{prop}\label{p:stand-graded-struct}
  Let $(s,n,\rho)$ be as in Theorem \ref{t:graded-classif}.  As $W$-representations,
    \[
    Y_\bH(s,n,\rho)=\sum_{\tilde\rho\in \widehat{A(n)}_{\mathcal E}} m(\rho,\tilde\rho) \mu(n,\tilde\rho) + \sum_{(n',\tilde\rho')} m_{(n,\tilde\rho),(n',\tilde\rho')} \mu(n',\tilde\rho'),
    \]
    where $m(\rho,\tilde\rho)=\dim \Hom_{A(s,n)}[\rho,\tilde\rho|_{A(s,n)}]$, $m_{(n,\tilde\rho),(n',\tilde\rho')}\in \mathbb Z_{\ge 0}$, $\tilde\rho'\in \widehat{A(n')}_{\mathcal E}$, and $n'$ ranges over a set of representatives of $G$-orbits in $\mathcal N$ such that $n\in\overline{G\cdot n'}\setminus {G\cdot n'}$. 
\end{prop}

Finally, we record the classification of tempered modules with real central character, see \cite[\S3.5]{barbaschciubo-ajm} for a summary of Lusztig's results \cite{lusztig-cuspidal-3}. If $n\in \mathcal N_G$, let $\{\bar n,h,n\}$ be a Lie triple for $n$. Then $A(h,n)=A(n)$, so we may identify $\widehat{A(n)}_\mathcal E=\widehat{A(h,n)}_\mathcal E.$

\begin{theorem}\label{t:graded-tempered} The map
\begin{equation}
\{(n,\rho): e\in\mathcal N,\rho\in \widehat {A(n)}_\mathcal E\}/G\longleftrightarrow X_\bH{(r_0h,n,\rho)}=Y_\bH(r_0h,n,\rho),
\end{equation}
 is a one-to-one correspondence onto the set of isomorphism classes of 
  tempered irreducible $\bH(L,\OO,\mathcal E)$- modules with real central
  character, on which $r$ acts by $r_0$.
\end{theorem}

\section{Lusztig's arithmetic-geometric correspondence}\label{sec:arithmeticgeometric}
We assume in this section that $\bfG$ is adjoint simple.

\subsection{Representations with unipotent cuspidal support}\label{subsec:LLC} We briefly recall the classification of irreducible representations with unipotent cuspidal support defined in \cite[\S1.6,\S1.21]{Lu-unip1}. With the notation as in section \ref{sec:btbuilding}, let $(\pi,V)$ be a $\mathbf G^\omega(\mathsf k)$-representation. We say $V$ has \emph{unipotent cuspidal support} if there exists a parahoric subgroup $\mathbf P_c(\mathfrak o)\subset \mathbf G^\omega(\mathsf k)$ and an irreducible unipotent cuspidal representation $(\sigma_{\mathbf E},\mathbf E)$ of $\mathbf L_c(\mathbb F_q)$ such that the $\mathbf E$-isotypic component of $V|_{P_c(\mathfrak o)}$ generates $V$. (By abuse of notation, we denote by $\mathbf E$ also the pullback of $\mathbf E$ to $\mathbf P_c(\mathfrak o)$). Let $\mathcal C(\mathbf G^\omega(\mathsf k),\bf E)$ denote the corresponding full subcategory of smooth $\mathbf G^\omega(\mathsf k)$-representations. The assignment
\[\mathsf{m}_{\mathbf E}:V\mapsto \Hom_{\mathbf P_c(\mathfrak o)}(\mathbf E,V)
\]
defines an equivalence of categories between $\mathcal C(\mathbf G^\omega(\mathsf k),\mathbf E)$ and the category of modules for the convolution algebra \[\mathcal H(\mathbf G^\omega(\mathsf k),\mathbf E)=\{f:\mathbf G^\omega(\mathsf k)\to \End_{\mathbb C}(\mathbf E)\mid f(p_1 g p_2)=\sigma_{\mathbf E}(p_1)\circ f(g)\circ \sigma_{\mathbf E}(p_2),\ g\in \mathbf G^\omega(\mathsf k), p_1,p_2\in \mathbf P_c(\mathfrak o)\}.\]

Let $G^\vee=\bfG^\vee(\CC)$ and let $T^\vee=X^*(\mathbf{T},\bark)\otimes_\ZZ \CC^\times$. The polar decomposition of $T^\vee$ is $T^\vee=T^\vee_c T^\vee_r$, where $T^\vee_c=X^*(\mathbf{T},\bark)\otimes_\ZZ S^1$ and $T^\vee_r=X^*(\mathbf{T},\bark)\otimes_\ZZ \RR_{>0}$. 
%If $\mathcal S$ is a subset of $G^\vee$, $A(\mathcal S)$ denote the component group of the centralizer $G^\vee(\mathcal S)$.
%\[Z^1_{G^\vee_{\mathsf{sc}}}(\mathcal S)=\text{ preimage of }G^\vee(\mathcal S)/Z(G^\vee)\text{ under the projection } G^\vee_{\mathsf{sc}}\to G^\vee_{\ad},
%and let $A^1(\mathcal S)$ denote the component group of $Z^1_{G^\vee_{\mathsf{sc}}}(\mathcal S)$. 

Write $\Phi(G^\vee)$ for the set of $G^\vee$-orbits (under conjugation) of triples $(s,n,\rho)$ where
\begin{itemize}
    \item $s\in G^\vee$ is semisimple,
    \item $n\in \mathfrak g^\vee$ such that $\operatorname{Ad}(s) n=q n$,
    \item $\rho\in \mathrm{Irr}(A_{G^{\vee}}(s,n))$.
    %such that $\rho|_{Z(G^\vee)}$ is a multiple of the identity.
\end{itemize}
Without loss of generality, we may assume that $s\in T^\vee$. Note that $n\in\mathfrak g^\vee$ is necessarily nilpotent. The group $G^\vee(s)$ acts with finitely many orbits on the $q$-eigenspace of $\Ad(s)$
$$\mathfrak g_q^\vee=\{x\in\mathfrak g^\vee\mid \operatorname{Ad}(s) x=qx\}.$$
In particular, there is a unique open $G^\vee(s)$-orbit in $\mathfrak g_q^\vee$.

Fix an $\mathfrak{sl}(2)$-triple $\{n^-,h^\vee,n\} \subset \fg^{\vee}$ with $h^\vee\in \mathfrak t^\vee_{\mathbb R}$ and set
$$s_0:=sq^{-\frac{h^\vee}{2}}.$$
Then $\operatorname{Ad}(s_0)n=n$.

A parameter $(s,n,\rho)\in \Phi(G^\vee)$ is called \emph{discrete} if ${G^\vee}(s,n)$ does not contain a nontrivial torus.
A discrete parameter is called \emph{cuspidal} if $(n,\rho)$ is a cuspidal pair in ${G^\vee}(s)$ in the sense of section \ref{sec:graded}. 
Let $\Pi^{\mathsf{Lus}}(\bfG^\omega(\mathsf k))$ denote the set of equivalence classes of irreducible $\bfG^\omega(\mathsf k)$-representations with unipotent cuspidal support and let
\[\Pi^{\mathsf{Lus}}(\bfG)=\bigsqcup_{\omega\in\Omega_{\mathsf{ad}}} \Pi^{\mathsf{Lus}}(\bfG^\omega(\mathsf k)).
\]
(In this subsection, $\bfG^1(\mathsf k)$ is the split form.)
The following theorem is a combination of several results, namely \cite[Theorems 7.12, 8.2, 8.3]{KL} for Iwahori-spherical representations, and \cite[Corollary 6.5]{Lu-unip1} and \cite[Theorem 10.5]{Lu-unip2} for representations with unipotent cuspidal support. %\cite[Theorem 3.5.4]{Re-isogeny} for $\bfG$ arbitrary and Iwahori-spherical representations, and \cite{FengOpdamSolleveld2021,FengOpdam2020,Sol-LLC} for $\bfG$ arbitrary and representations with unipotent cuspidal support. 
See \cite[\S2.3]{AMSol} for a discussion of the compatibility between these  classifications. 
%Define the subset $\mathrm{Irr}(A(s,n))_0 \subset \mathrm{Irr}(A(s,n))$ as in (\ref{eq:defofIrr0}) (taking $S=\{s,n\}$).

\begin{theorem}[{Deligne-Langlands-Lusztig correspondence}]\label{thm:Langlands} There is a bijection
$$\Phi(G^\vee)\xrightarrow{\sim} \Pi^{\mathsf{Lus}}(\bfG),
%\bigsqcup_{\omega\in\Omega_{\mathrm{ad}}}\Pi^{\mathsf{Lus}}(\bfG^\omega(\mathsf k)),
\qquad (s,n,\rho)\mapsto X(s,n,\rho),$$
with the following properties:
\begin{enumerate}
    \item $X(s,n,\rho)$ is tempered if and only if $s_0\in T_c^\vee$ (in particular, $\overline {G^\vee(s)n}=\mathfrak g_q^\vee$),
    \item $X(s,n,\rho)$ is square integrable (modulo center) if and only if it is tempered and $(s,n,\rho)$ is discrete.
    \item $X(s,n,\rho)$ is supercuspidal if and only if $(s,n,\rho)$ is a cuspidal parameter.
\item $X(s,n,\rho)\in \Pi^{\mathsf{Lus}}(\bfG^\omega(\mathsf k))$ if and only if $\rho|_{Z(G^\vee)}$ is a multiple of $\zeta_\omega$. 
\end{enumerate}
%This bijection satisfies several natural desiderata (including formal degrees, equivariance with respect to tensoring by weakly unramified characters), see \cite[Theorem 1]{Sol-LLC} and \cite[Theorem 2]{FengOpdamSolleveld2021}.
\end{theorem}
For each parameter $(s,n,\rho)$, there is an associated admissible \emph{standard module} $Y(s,n,\rho)$ with the following properties,  see \cite[Section 10]{Lu-unip2}, which relies on the graded affine Hecke algebra results in \cite[Theorems 1.15, 1.21, 1.22]{lusztig-cuspidal-3}):
\begin{enumerate}
\item $X(s,n,\rho)$ is the unique simple quotient of $Y(s,n,\rho)$;
\item If $\overline {G^\vee(s)n}=\mathfrak g_q^\vee$, then $X(s,n,\rho)=Y(s,n,\rho)$; in particular, this is the case when $X(s,n,\rho)$ is tempered. More generally:
\item There is an identity of characters (\cite[Corollary 10.7]{lusztig-cuspidal-2})
\begin{equation}\label{e:std-irred}
Y(s,n,\rho)=X(s,n,\rho)+\sum_{(n',\rho')} m((n,\rho),(n',\rho')) X(s,n',\rho'),
\end{equation}
where $m((n,\rho),(n',\rho'))\in \mathbb Z_{\ge 0}$, and the sum is over all $G^\vee(s)$-orbits $(n',\rho')$, such that $n\in \overline {G^\vee(s)n'}$.
\end{enumerate}

We denote by $\Pi^{\mathsf{Lus}}_s(\bfG(\mathsf k))$ the set of irreducible $\bfG(\mathsf k)$-representations $X(s,n,\rho)$ for a fixed $s\in T^\vee$.
We say that $X(s,n,\rho)$ has \emph{real infinitesimal character} if $s\in T^\vee_r.$

\begin{rmk}\label{r:real-indep}
    When $\bfG$ is not adjoint, the parameters of the unipotent representations take a similar form $(s,n,\rho)$, see \cite{Sol-LLC}, except $\rho$ is an irreducible representation of the group $A^1_{G^\vee}(s,n)$, rather than $A_{G^\vee}(s,n)$. To define $A^1_{G^\vee}(s,n)$, let $G^\vee_{\mathsf{sc}}\twoheadrightarrow G^\vee\twoheadrightarrow G^\vee_{\mathsf{ad}}$ be the chain of isogenies, and let $Z^1_{G^\vee}(s)$ be the preimage in $G^\vee$ of the image of $G^\vee(s)$ in $G^\vee_{\mathsf{ad}}.$ Then $A^1_{G^\vee}(s,n)$ is the group of components of the centralizer of $n$ in $Z^1_{G^\vee}(s)$. 

    If $s\in T^\vee_r$, let $\bar s$ be the unique element of the Lie algebra $\mathfrak t^\vee_r\subset \mathfrak g^\vee$ that exponentiates to $s$. Then $G^\vee(s)=C_{G^\vee}(\bar s)$, and $Z^1_{G^\vee}(s)=C_{G^\vee_{\mathsf{sc}}}(\bar s)$, and therefore $A^1(s,n)=A_{G^\vee_{\mathsf{sc}}}(\bar s,n)$. In other words, the classification on unipotent representations with real infinitesimal character is independent of isogeny. We will use this implicitly in Section \ref{sec:main}.
\end{rmk}

\subsection{The arithmetic side}\label{subsec:arithmetic} We follow \cite{lusztigunip}. Let $\mathcal H(\mathbf G^\omega(\mathsf k),\mathbf E)$ be a Hecke algebra as in Section \ref{subsec:LLC}. Let $J\in \bfP^\omega(I)$ be the subset of the affine simple roots corresponding to $\mathbf P_c(\mathfrak o)$, and let $\widetilde W_J$ be the finite subgroup of $\widetilde W$ generated by the simple reflections in $J$. Write $\mathcal W=N_{\widetilde W}(\widetilde W_J)/\widetilde W_J$. As a set, $\mathcal{W}$ can be identified with the set of minimal-length $(\widetilde W_J,\widetilde W_J)$-double cosets in $\widetilde W$ which are contained in $N_{\widetilde W}(\widetilde W_J)$. Let $\mathcal W'$ (resp. $\overline \Omega$) denote the subgroup of $\mathcal W$ given by the double cosets contained in $W^a$ (resp. $\Omega$). 

The element $\omega\in \Omega$ acts on $I$ and preserves $J$. Let $\mathcal W^\omega$, ${\mathcal W'}^{\omega}$, $\overline\Omega^\omega$ denote the fixed points. Let $(I-J)/\omega$ denote the orbits of the action of $\omega$ on $I-J$. The group $\overline \Omega^\omega$ acts on  $(I-J)/\omega$ by permutations. Let $\overline\Omega^\omega_1$ (resp. $\overline\Omega^\omega_2$) denote the kernel (resp. the image) of the resulting homomorphism $\overline\Omega^\omega\to \mathrm{Perm}(I-J/\omega)$, see \cite[\S1.20]{lusztigunip}. 

The affine Coxeter group ${\mathcal W'}^\omega$ has generators $s_{\underline k}$, where ${\underline k}$ ranges over $(I-J)/\omega$ if this set has cardinality at least two. Otherwise ${\mathcal W'}^\omega=1$. Let $\mathcal H'(I,J,\omega,\mathbf E)$ be the affine Hecke algebra with parameters $L(\underline k)$ constructed in \cite[\S1.18]{lusztigunip} for the extended affine Weyl group 
\begin{equation}\label{e:affineWeyl-arith}
\mathcal W(J,\omega,\mathbf E):={\mathcal W'}^\omega\rtimes \overline\Omega^\omega_2.
\end{equation}Then
\begin{equation}\label{e:arithm-HA}
    \mathcal H(\mathbf G^\omega(\mathsf k),\mathbf E)=\mathcal H'(I,J,\omega,\mathbf E)\otimes \mathcal C[\overline\Omega^\omega_1].
\end{equation}
This can be rewritten as $ \mathcal H(\mathbf G^\omega(\mathsf k),\mathbf E)=\bigoplus_{\psi\in \widehat{\overline\Omega^\omega_1}}\mathcal H'(I,J,\omega,\mathbf E)$. Each (conjugacy class of) triple $(J,\mathbf E,\psi)$ as above is called an \emph{arithmetic diagram} for $\mathbf G^\omega$ with associated affine Hecke algebra $\mathcal H'(I,J,\omega,\mathbf E)$ (independent of $\psi$). Let $\bar{\mathfrak U}_\omega$ denote the set of arithmetic diagrams $(J,\mathbf E,\psi)$ of $\mathbf G^\omega$ and set $\bar{\mathfrak U}=\sqcup_{\omega\in\Omega}\bar{\mathfrak U}_\omega$. 

There is a natural one-to-one correspondence (coming from an equivalence of categories)
\begin{equation}
\Pi^{\mathsf{Lus}}(\mathbf G^\omega(\mathsf k))\longleftrightarrow \bigsqcup_{(J,\mathbf E,\psi)\in\bar{\mathfrak U}_\omega} \mathrm{Irr}~ \mathcal H'(I,J,\omega,\mathbf E).
\end{equation}
Denote by $\mathsf{Acusp}(V)=(J,\omega,\mathbf E)$, the supercuspidal support of a simple $\mathbf G^\omega(\mathsf k)$-representation $V$.

\subsection{The geometric side}\label{subsec:geometric} Let $I^\vee$ denote the affine Dynkin diagram of the simply-connected simple group $G^\vee$. If $J^\vee\subsetneq I^\vee$, denote by $G^\vee_{J^\vee}$ the corresponding pseudo-Levi subgroup of $G^\vee$. Let $\bar{\mathfrak S}$ denote the set of triples ({\it geometric diagrams}) $(J^\vee,\mathcal C^\vee, \mathcal F^\vee)$, where $J^\vee\subsetneq I^\vee$, $\mathcal C^\vee$ is a nilpotent $G^\vee_{J^\vee}$-orbit on the Lie algebra $\mathfrak g^\vee_{J^\vee}$ and $\mathcal E^\vee$ is an isomorphism class of $G^\vee_{J^\vee}$-equivariant cuspidal local systems on $\mathcal C^\vee$. 
To such a triple, in \cite[\S5.12]{lusztigunip}, Lusztig associates a generic affine Hecke algebra $\mathcal H(G^\vee,G^\vee_{J^\vee},\mathcal C^\vee,\mathcal E^\vee)=\mathcal H^{\lambda,\lambda^*}(\Psi,z)$, for some root datum $\Psi$ and parameters $\lambda,\lambda^*$ as in {\it loc.cit.}. Denote by $\mathcal W^\vee(J^\vee,\mathcal C^\vee,\mathcal E^\vee)=\bar W\ltimes \mathcal X$ the extended affine Weyl group underlying this affine Hecke algebra. The classification of the simple modules of $\mathcal H(G^\vee,G^\vee_{J^\vee},\mathcal C^\vee,\mathcal E^\vee)$ goes via the reduction to graded affine Hecke algebras. More precisely, if $\mathcal T=\mathcal T_c\cdot \mathcal T_r=\Hom_{\mathbb Z}(\mathcal X,\mathbb C^\times)$, then fixing a semisimple conjugacy class $\bar W\cdot t\subset \mathcal T$, $t=t_c t_h$ we have
\[
\mathrm{Irr}_{\bar Wt,z_0}\mathcal H(G^\vee,G^\vee_{J^\vee},\mathcal C^\vee,\mathcal E^\vee)\longleftrightarrow \mathrm{Irr}_{\log t_h,r_0} \mathbf H({G^\vee}(s_c),G^\vee_{J^\vee},\mathcal C^\vee,\mathcal E^\vee),
\]
where $s_c$ (also denoted by $t_c^\psi$ in {\it loc.cit.}) is the element of $T^\vee_c$ associated to $t_c$ in \cite[\S5.13]{lusztigunip}. The group $G^\vee_{J^\vee}$ is a Levi subgroup of $Z_{G^\vee}(s_c)=G^\vee_{I^\vee-S^\vee}$, for a nonempty subset $S^\vee\subset I^\vee-J^\vee$. Therefore, the graded affine Hecke algebra in the right hand side is of the type defined in section \ref{sec:graded}, with $G=G^\vee_{I^\vee-S^\vee}$, $L=G^\vee_{J^\vee}$, and $(\mathcal C,\mathcal L)=(\mathcal C^\vee,\mathcal E^\vee).$

For every $\zeta\in \mathrm{Irr} Z(G^\vee)$, let $\bar{\mathfrak S}_\zeta$ denote subset of the triples $(J^\vee,\mathcal C^\vee,\mathcal E^\vee)$ such that $Z(G^\vee)\subset G^\vee_{J^\vee}$ acts on each stalk of $\mathcal E^\vee$ by $\zeta$.

For every $z_0>0$, denote by $\mathfrak S(z_0)$ the set of $G^\vee$-conjugacy classes of triples $(s,n,\rho)$, where $s\in T^\vee$, $n\in \mathfrak g^\vee$ nilpotent such that $\Ad(s)n=z_0^2 n$, and $\rho\in\Irr A(s,n).$  Denote by $\mathfrak S(z_0)_\zeta$ the subset of $\mathfrak S(z_0)$ with the property that $Z(G^\vee)$ acts in $\rho$ by the character $\zeta$. 

Via the classification of simple modules for graded Hecke algebras recalled in Theorem \ref{t:graded-classif}, we have (recall $s_c$ is determined by $t_c$):
\begin{equation}\label{e:irr-geom-diagram}
\mathrm{Irr}_{\bar Wt,z_0}\mathcal H(G^\vee,G^\vee_{J^\vee},\mathcal C^\vee,\mathcal E^\vee)\leftrightarrow \{(s,n,\rho)\in \mathfrak S(z_0)\mid s=s_c s_h,\ s_h\in T^\vee_h,\  \rho\in \widehat{A(s,n)}_{\mathcal E^\vee}\}.
\end{equation}

Then, by combining all such correspondences for all geometric diagrams, \cite[Theorem 5.21]{lusztigunip}] says that there is a natural bijection 
\begin{equation}
    \mathfrak S(z_0)_\zeta\longleftrightarrow \bigsqcup_{(J^\vee,\mathcal C^\vee,\mathcal E^\vee)\in \bar{\mathfrak S}_\zeta} \mathrm{Irr}~ \mathcal H_{z_0}(G^\vee,G^\vee_{J^\vee},\mathcal C^\vee,\mathcal E^\vee).
\end{equation}
    If $(s,n,\rho)\in \mathfrak S(z_0)$ parametrizes a module for $\mathcal H_{z_0}(G^\vee,G^\vee_{J^\vee},\mathcal C^\vee,\mathcal E^\vee)$, denote by $\mathsf{Gcusp}(s,n,\rho)=(J^\vee,\mathcal C^\vee,\mathcal E^\vee)$ and call it the geometric cuspidal support of $(s,n,\rho)$ (or $X(s,n,\rho), Y(s,n,\rho)$).
\begin{rmk}\label{r:real-geom-diag}
    If $X$ has real infinitesimal character and $\mathsf{Gcusp}(X)=(J^\vee,\mathcal C^\vee,\mathcal E^\vee)$, then $J^\vee \subseteq I^\vee_0$. 
    \end{rmk}

To complete the correspondence and the proof of Theorem \ref{thm:Langlands}, Lusztig matched the arithmetic and geometric diagrams.  More precisely, \cite[Theorem 6.3]{lusztigunip} gives an explicit bijection, for each $\omega\in\Omega$,
\begin{equation}
    \theta:\bar{\mathfrak S}_{\zeta_{\omega}} \longleftrightarrow\bar{\mathfrak U}_\omega,
\end{equation}
such that if $\theta((J^\vee,\mathcal C^\vee,\mathcal E^\vee))=(J,\mathbf E,\psi)$, then
\[
\mathcal H_{\sqrt q}(G^\vee,G^\vee_{J^\vee},\mathcal C^\vee,\mathcal E^\vee)\cong \mathcal H' (I,J,\omega,\mathbf E).
\]
In particular, the underlying (extended) affine Weyl groups are isomorphic
\[
\mathcal W^\vee(J^\vee,\mathcal C^\vee,\mathcal E^\vee)\cong\mathcal W(J,\omega,\mathbf E)
\]
Denote this isomorphism of Hecke algebras also by $\theta$ and let $\theta^*$ the dual module correspondence.
We also write $\theta$ for the bijection between $\overline{\mathfrak S}$ and $\overline{\mathfrak U}$. In this case write $\theta(J^\vee,\mathcal C^\vee,\mathcal E^\vee) = (J,\bf E,\psi,\omega)$.

\subsection{Multiplicities}\label{sec:multi}
%Retain the notation from section \ref{sec:btbuilding}. 
Let $(\pi,V)\in \mathcal C(\mathbf G^\omega(\mathsf k),\bf E)$ be an irreducible or standard representation with unipotent cuspidal support $(\mathbf P_c(\mathfrak o),\mathbf E)$, and $\mathbf P_{c'}(\mathfrak o)\supseteq \mathbf P_c(\mathfrak o)$  another parahoric subgroup. Define the convolution algebra
\[
\mathcal H(\mathbf P_{c'}(\mathfrak o),\mathbf E)=\{f\in \mathcal H(\mathbf G^\omega(\mathsf k),\mathbf E)\mid \mathsf{supp} f\subseteq \mathbf P_{c'}(\mathfrak o)\}.
\]
If $\tau$ is irreducible representation of $\mathbf L_{c'}(\mathbb F_q)$, denote by $\mathsf{m}_{\mathbf E}(\tau)$ the representation of $\mathcal H(\mathbf P_{c'}(\mathfrak o),\bf E)$ on $\Hom_{\mathbf L_c(\mathbb F_q)}(\mathbf E, \tau|_{\mathbf L_c(\mathbb F_q)}).$ Then we have
\begin{equation}
    \Hom_{\mathbf P_{c'}(\mathfrak o)}(\tau, V)=\Hom_{\mathcal H(\mathbf P_{c'}(\mathfrak o),\mathbf E)}(\mathsf m_{\mathbf E}(\tau),\mathsf m_{\mathbf E}(V)|_{\mathcal H(\mathbf P_{c'}(\mathfrak o),\mathbf E)} ).
\end{equation}
Next, let $J\subseteq J'\in \bfP^\omega(I)$ be the subsets defining $\mathbf P_c(\mathfrak o),\mathbf P_{c'}(\mathfrak o)$, respectively. Then using (\ref{e:arithm-HA}), we find
\begin{equation}
   \Hom_{\mathbf P_{c'}(\mathfrak o)}(\tau, V)=\Hom_{\mathcal H(J',J,\omega,\mathbf E)}( \mathsf m_{\mathbf E}(\tau), \mathsf m_{\mathbf E}(V)).
\end{equation}
Here, $\mathsf m_{\mathbf E}(V)$ is a module an affine Hecke algebras $\mathcal H'(I,J,\omega,\mathbf E)$ (for one $\psi$), and $\mathcal H(J',J,\omega,\mathbf E)$ is the parahoric (finite) Hecke algebra with support in $J'$.

Let $(J^\vee, \mathcal C^\vee,\mathcal E^\vee)=\theta^{-1}(J,\mathbf E,\psi)$ be the corresponding geometric diagram in $\bar{\mathfrak S}_{\zeta_\omega}.$ Denote by ${J'}^\vee=\theta^{-1}(J')$ the corresponding subset of $I^\vee$ such that there is an isomorphism of finite Hecke algebras
\[
\mathcal H_{\sqrt q}(G^\vee_{{J'}^\vee},G^\vee_{J^\vee},\mathcal C^\vee,\mathcal E^\vee)\cong \mathcal H' (J',J,\omega,\mathbf E).
\]
Let $\mathcal W^\vee_{{J'}^\vee}(J^\vee,\mathcal C^\vee,\mathcal E^\vee)\cong\mathcal W_{J'}(J,\omega,\mathbf E)$ denote the finite parahoric subgroups defined by ${J'}^\vee\supseteq J^\vee$ and $J'\supseteq J$, respectively.
It follows that
\[
 \Hom_{\mathbf P_{c'}(\mathfrak o)}(\tau, V)=\Hom_{\mathcal H_{\sqrt q}(G^\vee_{{J'}^\vee},G^\vee_{J^\vee},\mathcal C^\vee,\mathcal E^\vee)}( \theta^*(\mathsf m_{\mathbf E}(\tau)), \theta^*(\mathsf m_{\mathbf E}(V))),
\]
where 
$\theta^*(\mathsf m_{\mathbf E}(V))$ is an $\mathcal H_{\sqrt q}(G^\vee,G^\vee_{J^\vee},\mathcal C^\vee,\mathcal E^\vee)$-module. We are now in position to apply the deformation argument from section \ref{sec:OS-deform}. From the first part of Corollary \ref{c:restr-temp}, it follows at once that, when $V$ is irreducible,
\begin{equation}\label{e:mult-1}
     \Hom_{\mathbf P_{c'}(\mathfrak o)}(\tau, V)=\Hom_{\mathcal W^\vee_{{J'}^\vee}(J^\vee,\mathcal C^\vee,\mathcal E^\vee)}(\theta^*(\mathsf m_{\mathbf E}(\tau))_{q\to 1}, \widetilde \sigma_0(\theta^*(\mathsf m_{\mathbf E}(V))),
\end{equation}
where $\widetilde\sigma_0$ is the deformation functor from (\ref{e:OS-deform}). Here we use the notation $q\to 1$ in place of $z_0\to 1$ as in section \ref{sec:OS-deform}. Recall that $\widetilde \sigma_0(\theta^*(\mathsf m_{\mathbf E}(V))$ is a module for the extended affine Weyl group $\mathcal W^\vee(J^\vee,\mathcal C^\vee,\mathcal E^\vee)$.

 \
 
 For standard modules, we can explicitate this formula further using Proposition \ref{p:restr-standard}. Let $Y(s,n,\rho)$ be a standard $\mathbf G^\omega(\mathsf k)$-module with unipotent cuspidal support. Let $(J,\mathbf E)$ be the supercuspidal support of $X(s,n,\rho)$, and $(J^\vee,\mathcal C^\vee, \mathcal E^\vee)=\mathsf{Gcusp}(s,n,\rho).$ Let $G^\vee_{I^\vee-S^\vee}=Z_{G^\vee}(s_c)$. Let $\mathcal P_{s_c,n}$ denote the generalized Springer fiber $\mathcal P_n$ defined as in section \ref{sec:graded} for the pair $(\mathcal C^\vee, \mathcal E^\vee)$ in $\mathfrak g^\vee_{I^\vee-S^\vee}.$ Let $s_c$ be the compact part of $s$ and $t_c$ the corresponding element in $\mathcal T_c$ for the geometric Hecke algebra $\mathcal H(G^\vee_{{J'}^\vee},G^\vee_{J^\vee},\mathcal C^\vee,\mathcal E^\vee)$. 

\begin{theorem}\label{t:mult-general} Let $\mathbf P_{c}(\mathfrak o)$ be a parahoric subgroup for $\mathbf G^\omega(\mathsf k)$ corresponding to $J$, $J\subset J'\in \bfP^\omega(I)$, and $\tau$ an irreducible $\mathbf L_{c'}(\mathbf F_q)$-representation. Write $\mathcal W^\vee:=\mathcal W^\vee(J^\vee,\mathcal C^\vee,\mathcal E^\vee)$ and $\mathcal W^\vee_{{J'}^\vee}:=\mathcal W^\vee_{{J'}^\vee}(J^\vee,\mathcal C^\vee,\mathcal E^\vee)$ for simplicity, and let $\mathcal W^\vee(t_c)$ be the centralizer of $t_c$ in $\mathcal W^\vee$. Then
   \begin{equation}
   \begin{aligned}
   \Hom_{\mathbf P_{c'}(\mathfrak o)}(\tau,Y(s,n,\rho))&=\Hom_{\mathcal W^\vee_{{J'}^\vee}}(\theta^*(m_{\mathbf E}(\tau))_{q\to 1},\Ind_{\mathcal W^\vee(t_c)}^{\mathcal W^\vee}(\mathbb C_{t_c}\otimes Y_{\mathbf H}(s,n,\rho)))\\
   &=\Hom_{\mathcal W^\vee_{{J'}^\vee}}(\theta^*(m_{\mathbf E}(\tau))_{q\to 1},\Ind_{\mathcal W^\vee(t_c)}^{\mathcal W^\vee}(\mathbb C_{t_c}\otimes H_\bullet (\mathcal P_{s_c,n},\dot{\mathcal E}^\vee)^\rho)   ,
   \end{aligned}
   \end{equation}
   where $Y_{\mathbf H}(s,n,\rho))$ is the standard module for the graded affine Hecke algebra $\mathbf H({G^\vee}(s_c),G^\vee_{J^\vee},\mathcal C^\vee,\mathcal E^\vee)$, and $ H_\bullet (\mathcal P_{s_c,n},\dot{\mathcal E}^\vee)^\rho$ is a (reducible) generalized Springer representation of the finite Weyl group $N_{{G^\vee}(s_c)}(G^\vee_{J^\vee})/G^\vee_{J^\vee}$, as exposited in section \ref{sec:graded}.
\end{theorem}

\begin{proof}
    The first equality follows immediately from (\ref{e:mult-1}), Proposition \ref{p:restr-standard} and the second part of Corollary \ref{c:restr-temp}. The second equality is implied directly by Lemma \ref{l:graded-W-action}.
\end{proof}

In the case of real infinitesimal character in this paper, this result takes the following simpler form.

\begin{cor}\label{c:mult-real}
    Retain the notation from Theorem \ref{t:mult-general}. Suppose $Y(s,n,\rho)$ has real infinitesimal character, i.e. $s_c=1$. Then
    \begin{equation}
   \begin{aligned}
   \Hom_{\mathbf P_{c'}(\mathfrak o)}(\tau,Y(s,n,\rho))&=\Hom_{\mathcal W^\vee_{{J'}^\vee}}(\theta^*(m_{\mathbf E}(\tau))_{q\to 1}, Y_{\mathbf H}(s,n,\rho)))\\&=\Hom_{\mathcal W^\vee_{{J'}^\vee}}(\theta^*(m_{\mathbf E}(\tau))_{q\to 1}, H_\bullet (\mathcal P_{n},\dot{\mathcal E}^\vee)^\rho)   ,
   \end{aligned}
   \end{equation}
   where $Y_{\mathbf H}(s,n,\rho))$ is the standard module for the graded affine Hecke algebra $\mathbf H(G^\vee,G^\vee_{J^\vee},\mathcal C^\vee,\mathcal E^\vee)$, $\mathcal W^\vee_{J^\vee}=N_{G^\vee}(G^\vee_{J^\vee})/G^\vee_{J^\vee}$, and $H_\bullet(\mathcal P_n,\dot{\mathcal E}^\vee)$ are the homology groups carrying a $\mathcal W^\vee_{J^\vee}\times A(n)$-representation that occur in the generalized Springer correspondence for $G^\vee_{J^\vee}\subset G^\vee$ (see \ref{sec:graded}).
\end{cor}

\section{Faithfulness for the generalised Springer correspondence}\label{sec:faithful}

\subsection{Preliminaries on partitions}\label{subsec:partitions}

In this section, we will recall some of the basic notation for partitions, and some related objects. A \emph{partition of $n$} is a non-increasing sequence of positive integers $(\lambda_1 \geq \lambda_2 \geq ... \geq \lambda_m)$ such that $\lambda_1+\lambda_2+...+\lambda_m=n$ (occasionally, it will be more convenient to denote a partition as a non-decreasing sequence, but we will clearly indicate when this is the case). Let $\mathcal{P}(n)$ denote the set of partitions of $n$. We write $\#\lambda :=m$ and $|\lambda|:=n$. If $x \in \ZZ_{\geq 0}$, we write $m_{\lambda}(x)$ for the multiplicity of $x$ in $\lambda$ and $\mathrm{ht}_{\lambda}(x)$ for its \emph{height}, i.e. $\mathrm{ht}_{\lambda}(x) = \sum_{y \geq x} m_{\lambda}(y)$. A partition is \emph{very even} if all parts are even, occurring with even multiplicity. A partition is \emph{rather odd} if every even part occurs with even multiplicity and every odd part occurs with multiplicity $\leq 1$. We write $\mathcal{P}_{ve}(n)$ (resp. $\mathcal{P}_{ro}(n)$) for the set of very even (resp. rather odd) partitions of $n$.

There is a partial order on integer sequences (and in particular, on $\mathcal{P}(n)$), defined by the formula
\begin{equation}\label{eq:dominanceordering}\lambda \leq \mu \iff \sum_{i \leq j} \lambda_i \leq \sum_{i \leq j} \mu_i.\end{equation}
Given an arbitrary partition $\lambda = (\lambda_1,...,\lambda_m)$, we define two new partitions (of $|\lambda|+1$ and $|\lambda|-1$)
$$\lambda^+ := (\lambda_1+1,\lambda_2,...,\lambda_m), \qquad \lambda^- := (\lambda_1,\lambda_2,...,\lambda_m-1)$$
and we write $\lambda^t$ for the transpose of $\lambda$.

Given $\lambda \in \mathcal{P}(n)$ and $\mu \in \mathcal{P}(m)$, we define $\lambda \cup \mu \in \mathcal{P}(m+n)$ by `adding multiplicities'
$$m_{\lambda \cup \mu}(x) = m_{\lambda}(x) + m_{\mu}(x), \qquad \forall x \in \ZZ_{\geq 0}.$$
We will need the following elementary fact.

\begin{lemma}\label{lem:uniondominance}
Let $\alpha_1,\alpha_2,\beta_1,\beta_2$ be partitions such that $\alpha_1 \leq \alpha_2$ and $\beta_1 \leq \beta_2$. Then $\alpha_1\cup\beta_1 \leq \alpha_2\cup\beta_2$.
\end{lemma}

\begin{proof}
We begin with the following observation. If $\lambda$ is an integer sequence and $\lambda'$ is the partition obtained by permuting the elements of $\lambda$, then $\lambda \leq \lambda'$. Indeed, $\lambda'$ is obtained from $\lambda$ through the following algorithm: if $\lambda_i > \lambda_j$ for $i > j$, swap $\lambda_i$ and $\lambda_j$, and swaps of this type are strictly increasing with respect to $\leq$. 

Add $0$'s to $\alpha_i$, $\beta_i$ so that $\#\alpha_1+\#\beta_1=\#\alpha_2 + \#\beta_2$, and choose a function $f: \{1,2,...,\#\alpha_1 + \#\beta_1\} \to \{\alpha,\beta\}$ such that $\alpha_1 = ((\alpha_1 \cup \beta_1)_i \mid f(i) = \alpha)$ and $\beta_1 = ((\alpha_1 \cup \beta_1)_i \mid f(i) = \beta)$. Define a sequence $\lambda$ by replacing each element of $\alpha_1 \cup \beta_2$ with the corresponding element in $\alpha_2$ or $\beta_2$ according to $f$, i.e.

$$\lambda_i = \begin{cases*}
                    (\alpha_2)_{\#\{j \leq i \mid f(j) = \alpha\}} & if  $f(i)=\alpha$  \\
                     (\beta_2)_{\#\{j \leq i \mid f(j) = \beta\}} & if  $f(i)=\beta$
                 \end{cases*} $$
Since $\alpha_1 \leq \alpha_2$ and $\beta_1 \leq \beta_2$, we have $\alpha_1 \cup \beta_1 \leq \lambda$. On the other hand, $\alpha_2 \cup \beta_2$ is the partition obtained by permuting the entries of $\lambda$. So by the discussion of the previous paragraph, $\lambda \leq \alpha_2 \cup \beta_2$. So $\alpha_1 \cup \beta_1 \leq \alpha_2 \cup \beta_2$, as asserted.
\end{proof}

If $r\geq \lambda_1$, we write $(r,\lambda) = \lambda \cup \{r\}$. We write $\alpha = \lambda \cup_{\geq} \mu$ if $\alpha = \lambda \cup \mu$ and the smallest part of $\lambda$ is at least as large as the largest part of $\mu$. We write $\mu \subseteq \lambda$ if $m_{\mu}(x) \leq m_{\lambda}(x)$ for every $x \in \ZZ_{\geq 0}$. In this case, there is a unique subpartition partition $\lambda \setminus \mu \subseteq \lambda$ such that $\lambda = \mu \cup (\lambda \setminus \mu)$.

%A \emph{decorated} partition of $n$ is a partition $\lambda \in \mathcal{P}(n)$ together with a decoration $\kappa \in \{0,1\}$. We write $(\lambda,\kappa)$ as $\lambda^{\kappa}$. If $\lambda$ is \emph{not} very even, we declare $\lambda^0 = \lambda^1$. Otherwise we regard $\lambda^0$ and $\lambda^1$ as distinct. Write $\mathcal{P}^d(n)$ for the set of decorated partitions of $n$, with the equivalence relation just defined.

A \emph{bipartition} of $n$ is a pair of partitions $(\lambda,\mu)$ such that $|\lambda|+|\mu|=n$. Write $\mathcal{B}(n)$ for the set of bipartitions of $n$. $\ZZ_2$ acts by flipping on $\mathcal{B}(n)$, and we write $\mathcal{D}(n)$ for the set of $\ZZ_2$-orbits on $\mathcal{B}(n)$.

\subsection{Preliminaries on nilpotent orbits in classical types}\label{subsec:orbits}

For $X \in \{B,C,D\}$ and $n$ a positive integer, let $\fg_X(n)$ denote the simple complex Lie algebra of type $X$ and rank $n$. Define the sets
\begin{align*}
    \mathcal{P}_B(2n+1) &= \{\lambda \in \mathcal{P}(2n+1) \mid x \text{ even} \implies m_{\lambda}(x) \text{ even}\}\\
    \mathcal{P}_C(2n) &= \{\lambda \in \mathcal{P}(2n) \mid x \text{ odd} \implies m_{\lambda}(x) \text{ odd (or 0)}\}\\
    \mathcal{P}_D(2n) &= \{\lambda \in \mathcal{P}(2n) \mid x \text{ even} \implies m_{\lambda}(x) \text{ even}\}
\end{align*}
The elements of $\mathcal{P}_B(2n+1)$ (resp. $ \mathcal{P}_C(2n)$, resp. $\mathcal{P}_D(2n)$) are called \emph{$B$-partitions} (resp. \emph{$C$-partitions}, resp. \emph{$D$-partitions}). There are well-known bijections
$$\cN_o^{\fg_B} \xrightarrow{\sim} \mathcal{P}_B(2n+1), \qquad  \cN_o^{\fg_C} \xrightarrow{\sim} \mathcal{P}_C(2n)$$
For $X=D$, there is a surjection
$$\cN_o^{\fg_D} \twoheadrightarrow \mathcal{P}_D(2n)$$
which is a bijection over partitions in $\mathcal{P}_D(2n)$, except the very even partitions, over which is two-to-one, see \cite[Section 5.1]{CM}. 

For $X \in \{B,C,D\}$, the $X$-collapse of $\lambda$ is the unique largest $X$-partition $\lambda_X$ such that $\lambda_X \leq \lambda$. A formula for $\lambda_X$ is provided in \cite[Lem 6.3.8]{CM}. 

If $\lambda$ is an $X$-partition, define the \emph{dual} partition $d_{BV}(\lambda)$ of $\lambda$ as follows:
\begin{itemize}
    \item If $\lambda \in \mathcal{P}_B(2n+1)$, then $d_{BV}(\lambda) = ((\lambda^t)^-)_C \in \mathcal{P}_C(2n)$.
    \item If $\lambda \in \mathcal{P}_C(2n)$, then $d_{BV}(\lambda) = ((\lambda^t)^+)_B \in \mathcal{P}_B(2n+1)$.
    \item If $\lambda \in \mathcal{P}_D(2n)$, then $d_{BV}(\lambda)=(\lambda^t)_D \in \mathcal{P}_D(2n)$.
\end{itemize}

\subsection{Preliminaries on symbols}\label{subsec:symbols}

In this section, we will recall some of the basic notation for symbols, following \cite[Section 4.4]{geckmalle}. 

For a sequence $X=(x_1,x_2,\dots,x_n)$ and integer $a\in \ZZ$ write $X+a$ for the sequence $(x_1+a,x_2+a,\dots,x_n+a)$. We write $z$ for the sequence $(0,1,\dots,n-1)$ ($n$ will always be clear from the context and therefore omitted from the notation).

A $\beta$-\emph{set} is a strictly increasing sequence $x=(x_1 < x_2 < ... <x_m)$. The partition $\mathcal{P}(x)$ associated to $x$ is $(x_1-0,x_2-1,\dots,x_m-(m-1))$. We say that $x$ is associated to a partition $\lambda$ if $\mathcal{P}(x)=\lambda$. For example, the $\beta$-sets associated to  $\lambda=(1,2,4)$ are those of the form $(1,3,6)$ or $(0,1,2,...,k-1,k+1,k+3,k+6)$ for some $k$.

For integers $a,b \geq 0$ and $d$, let $\widetilde{\mathcal{X}}_d^{a,b}$ denote the set of pairs $(X,Y)$ of finite sequences $X=(x_1,...,x_r)$ and $Y=(y_1,...,y_s)$ of nonnegative integers, such that
\begin{align*}
    x_j &\geq x_{j-1} + a+b \qquad && 2 \leq j \leq r,\\
    y_j &\geq y_{j-1} + a+b &&2 \leq j \leq s,\\
    y_1 &\geq b, && \\
    r-s &=d. &&
\end{align*}
The elements of $\widetilde{\mathcal{X}}_d^{a,b}$ are called \emph{symbols}, and are often denoted as follows
$$(X,Y) = \begin{pmatrix} x_1 & &x_2&\dots & &x_r\\&y_1&&\dots&y_s
\end{pmatrix}$$
The \emph{defect} of $S=(X,Y) \in \widetilde{\mathcal{X}}_d^{a,b}$ is the integer $d(S)=d=|X|-|Y|$. The \emph{content} is the integer $|S|:=|X|+|Y|$. Define
\begin{align*}
    r(S) &:= \sum_{x \in X \cup Y} x -(a+b) \lfloor \frac{(|S|-1)^2}{4}\rfloor - b \lfloor \frac{|S|}{2}\rfloor \\
    m(S) &:= \sum_{x \in X \cup Y} x -(a+b) \left(\binom{|X|}{2}+\binom{|Y|}{2}\right)- b |Y|.
\end{align*}
where the sums run over the multiset $X \cup Y$. We call $r(S)$ (resp. $m(S)$) the \emph{rank} (resp. \emph{relative rank}) of $S$. It is easy to check that both $r(S)$ and $m(S)$ are nonnegative integers.
Define 
$$\Delta^{a,b}_d := (a+b)\left\lfloor \frac {d^2}{4}\right\rfloor - b \left\lfloor\frac {d}2\right\rfloor = \begin{cases} (a+b)\left(\frac {d}{2}\right)^2 - b \left(\frac {d}2\right) & d \text{ even}\\ (a+b)\left(\frac {d-1}{2}\right)\left(\frac {d-1}{2}+1\right) - b \left(\frac {d-1}2\right) & d \text{ odd}.\end{cases}$$
Then for $S\in \tilde {\mathcal X}^{a,b}_d$
$$m(S) = r(S) - \Delta^{a,b}_{d}.$$
The \emph{shift} of a symbol $S =(X,Y) \in \widetilde{\mathcal{X}}_d^{a,b}$ is the symbol
$$S' = (\{0\} \cup (X+a+b), \{b\} \cup (Y+a+b)) \in \widetilde{\mathcal{X}}_d^{a,b}$$
Note that both defect and rank are invariant under shift. Write $\mathcal{X}_d^{a,b}$ for the set of equivalence classes under the equivalence relation generated by the shift operation. Write $\mathcal{X}_{d,n}^{a,b}$ for the set of classes in $\mathcal{X}_d^{a,b}$ of rank $n$ and write
$$\mathcal{X}^{a,b}_{\bullet,n} = \bigsqcup_{d \in \ZZ} \mathcal{X}_{d,n}^{a,b},\quad \mathcal{X}^{a,b}_{even,n} = \bigsqcup_{d \in \ZZ,d\ even} \mathcal{X}_{d,n}^{a,b},\quad \mathcal{X}^{a,b}_{odd,n} = \bigsqcup_{d \in \ZZ,d\ odd} \mathcal{X}_{d,n}^{a,b}.$$

There is an equivalence relation $\sim$ on $\mathcal{X}_{\bullet,n}^{a,b}$ defined as follows: $S_1 \sim S_2$ if and only if there are representatives $(X_1,Y_1)$ and $(X_2,Y_2)$ of $S_1$ and $S_2$ such that the multisets $X_1 \cup Y_1$ and $X_2 \cup Y_2$ coincide. 

There is a bijection $\mathcal {B}:\mathcal X^{a,b}_{d,n}\to \mathcal B(n-\Delta_d^{a,b})$ given by 
\begin{equation}\label{eq:P}(X,Y)\mapsto (X-(0,(a+b),2(a+b),\dots),Y-(b,b+(a+b),b+2(a+b),\dots)).\end{equation}
For $a=b=0$ this specialises to a bijection $\mathcal B:\mathcal X^{0,0}_{d,n}\to \mathcal B(n)$.

If $b=0$, then $\ZZ_2$ acts on $\mathcal{X}_d^{a,0} \cup \mathcal{X}^{a,0}_{-d}$ by interchanging the rows. Let $\mathcal{Y}_d^a$ denote the set of $\ZZ_2$-orbits on $\mathcal{X}_d^{a,0} \cup \mathcal{X}^{a,0}_{-d}$. Note that the rank of a symbol is invariant under flipping. We call $S\in \mathcal Y_d^a$ degenerate if its $\ZZ_2$-orbit is a singleton. We write $\mathcal{Y}_{d,n}^a$ for the set of orbits of rank $n$ and
$$\mathcal{Y}^{a}_{\bullet,n} = \bigsqcup_{d \ge 0} \mathcal{Y}_{d,n}^{a},\quad \mathcal{Y}^{a}_{even,n} = \bigsqcup_{d \ge 0,d\ even} \mathcal{Y}_{d,n}^{a},\quad  \mathcal{Y}^{a}_{odd,n} = \bigsqcup_{d \ge 0,d\ odd} \mathcal{Y}_{d,n}^{a}.$$
It is clear that $\sim$ induces an equivalence relation on $\mathcal{Y}_{\bullet, n}^a$.

Given a symbol $S = (X,Y)\in \mathcal Y^a_{d,n}$ define $\underline S\in \mathcal X^{a,0}_{|d|,n}$ as follows: if $d\ne 0$
\begin{equation}
    \underline S := \begin{cases}
        (X,Y) & \mbox{if } |X|>|Y| \\
        (Y,X) & \mbox{if } |Y|>|X|
    \end{cases}
\end{equation}
and if $d=0$
\begin{equation}
    \underline S := \begin{cases}
        (X,Y) & \mbox{if } \sum_{x\in X}x\ge \sum_{y\in Y}y \\
        (Y,X) & \mbox{if } \sum_{x\in X}x< \sum_{y\in Y}y 
    \end{cases}
\end{equation}

There is a bijection 
\begin{equation}
    \label{eq:mathcalD}
    \mathcal {D}:\mathcal Y^{a}_{d,n}\to \begin{cases}\mathcal D(n) & d=0 \\ \mathcal B_{n-\Delta^{a,0}_d} & d\ne 0.\end{cases}
\end{equation}
given by 
\begin{equation}\label{eq:P}S\mapsto (X-(0,(a+b),2(a+b),\dots),Y-(0,(a+b),2(a+b),\dots))\end{equation}
where $(X,Y) = \underline S$.

For $d\in \mathbb Z$ there is a bijection $\mathcal X_{d,n}^{1,0}\to \mathcal X_{d+1,n}^{0,1}$ given by $S=(X,Y)\mapsto S^!$ where 
$$S^! = (\{0\}\cup (X+1),Y+1).$$
We have that $\mathcal B(S) = \mathcal B(S^!)$.
%This induces a map $\mathcal Y^1_{d,n} \mapsto \mathcal X^{0,1}_{|d|+1,n}$ which we also denote $S\mapsto S^!$ given by $S\mapsto \underline S^!$.

A symbol $S =(X,Y) \in \mathcal{X}^{a,b}_{d,n}$ is called \emph{special} if $d \in \{0,1\}$ and $x_1 \leq y_1 \leq x_2 \leq ... $. Note that this definition is independent of the choice of representative of $S$. Every $\sim$-equivalence class in $\mathcal{X}^{a,b}(n)$ contains a unique special element. A symbol $S \in \mathcal{Y}^{a,b}_n$ is called \emph{special} if the corresponding $\ZZ_2$-orbit in $\mathcal{X}^{a,b}_n$ contains a special symbol. Given a symbol $S$ in $\mathcal{X}^{a,b}_n$ or $\mathcal{Y}^{a}_n$, we write $S^{sp}$ for the (unique) special symbol in the $\sim$-equivalence class of $S$. 
%Write $S^{sp,l}$ for the sequence consisting of the multiset of entries of $S$ in non-decreasing order.

%For a special symbol $S\in \mathcal X^{a,b}_n$ define $\lambda(S)$ to be $\mathcal P(2X\cup2Y+1)$
%For a special symbol $S\in \mathcal Y^a_n$ define $\lambda(S)$ to be $\lambda(S')$ where $S'$ is the unique special symbol in the $\ZZ_2$-orbit of $S$.

For $n \geq 0$, write $W_n$ (resp. $W_n'$) for the Weyl group of type $B_n$ (resp. $D_n$). 
Recall that $\mathrm{Irr}(W_n)\leftrightarrow \mathcal B_n$ and $\mathrm{Irr}(W_n') \leftrightarrow \mathcal D_n$ except to each degenerate orbit in $\mathcal D_n$ there are two irreducible representations of $W_n'$.

To make our terminology consistent with that of \cite{cmo} we say that $S$ is an \emph{a-symbol} if $a+b = 1$ and an \emph{s-symbol} if $a+b = 2$.

We conclude this subsection with two elementary lemmas. A \emph{linear presentation} of a symbol $S =(X,Y) \in \widetilde{\mathcal{X}}_{\bullet}^{a,0}$ is a pair $(Z,\epsilon)$ consisting of sequences $Z \in \ZZ_{\geq 0}^{|S|}$ and $\epsilon \in \{0,1\}^{|S|}$ such that $X=(Z_i \mid \epsilon_i = 0)$ and $Y= (Z_i \mid \epsilon_i = 1)$. We say that $(Z,\epsilon)$ is \emph{canonical} if $Z$ is non-decreasing and $\epsilon_i=0$ whenever $Z_i=Z_{i+1}$. Note that if $a>0$, then $S$ has a unique canonical linear presentation, which we will denote by $(Z^*(S),\epsilon^*(S))$ or simply $(Z^*,\epsilon^*)$ when there is no risk for confusion. 

\begin{lemma}
\label{lem:linearpresentation}
Suppose $S =(X,Y) \in \widetilde{\mathcal{X}}_{\bullet}^{1,0}$. Let $(Z^*,\epsilon^*)$ be the canonical linear presentation of $S$ and let $(Z,\epsilon)$ be any linear presentation of $S$. Then 
$$Z^* + \epsilon^* \geq Z + \epsilon$$
\end{lemma}

\begin{proof}
Note that $Z^*+\epsilon=Z+\epsilon$ as multisets. So by the argument in the proof of Lemma \ref{lem:uniondominance}, it suffices to show that the sequence $Z^*+\epsilon^*$ is non-decreasing. If $i<j$, then either
\begin{itemize}
    \item[(1)] $Z^*_i < Z^*_j$, or
    \item[(2)] $Z^*_i=Z^*_j$, $\epsilon^*_i=0$, and $\epsilon^*_j=1$. 
\end{itemize}
In the first case, we have $(Z^*+\epsilon^*)_i \leq Z^*_i+1 \leq Z^*_j \leq (Z^*+\epsilon^*)_j$. In the second case, $(Z^*+\epsilon^*)_i = Z^*_i < Z^*_i+1 = (Z^*+\epsilon^*)_j$. Hence $Z^*+\epsilon^*$ is non-decreasing.
\end{proof}

For a symbol $S=(X,Y)$ let $\tilde S = (2X,2Y+1)^{sp}$.
\begin{lemma}
    \label{lem:unwind}
Let $S=(X,Y)$ be an a-symbol and let $(Z,\epsilon)$ be the canonical linear presentation for $S$. Then $Z^*(\tilde S) = 2Z+\epsilon$.
\end{lemma}
\begin{proof}
    Clearly $\tilde S$ and $2Z+\epsilon$ are the same as multisets.
    Thus it suffices to prove that $2Z+\epsilon$ is non-decreasing.
    Let $i<j$. Then either
    \begin{itemize}
        \item[(1)] $Z_i < Z_j$, or
        \item[(2)] $Z_i=Z_j$, $\epsilon_i=0$, and $\epsilon_j=1$. 
    \end{itemize}
    In the first case $2Z_i +2 \le 2Z_j$ and so $2Z_i+\epsilon_i < 2Z_i + 2 \le 2Z_j \le 2Z_j +\epsilon_j$.
    In the second case $2Z_i+\epsilon_i = 2Z_i < 2Z_i+1 = 2Z_j+\epsilon_j$. Hence $2Z+\epsilon$ is non-decreasing (in fact increasing).
\end{proof}

\subsection{Combinatorics for Spin groups}
\label{subsec:spincombo}

In this section, we will recall, and reinterpret, some of the combinatorics from \cite[Section 4]{LusztigSpaltenstein1985} required to formulate the generalized Springer correspondence for $Spin(N)$.

If $m \in \ZZ$, define $d(m) \in \{0,\pm 1\}$ by
$$d(m) = \begin{cases*}
                    \phantom{-}0 & if  $m$ is even\\
                    \phantom{-}1 & if $m \equiv 1 \mod{4}$\\
                    -1 & if $m \equiv 3 \mod{4}$
                 \end{cases*}$$
If $\lambda \in \mathcal{P}(n)$, define 
\begin{equation}\label{eq:dlambda}
d(\lambda) = \sum_i d(\lambda_i).
\end{equation}

Now suppose $\lambda = (\lambda_1 \leq ... \leq \lambda_m)\in \mathcal{P}_{ro}(N)$. Following \cite[Section 4.11]{LusztigSpaltenstein1985}, we will define a bipartition $\rho(\lambda)$. If $m=0$, we put $\rho(\lambda) = ((0),(0))$. Suppose $m\neq 0$. Let
\begin{equation}\label{eq:mu}
\mu = \begin{cases*}
                    (\lambda_1 \leq ... \leq \lambda_{m-1}) & if  $\lambda_m$ is odd\\
                    (\lambda_1 \leq ... \leq \lambda_{m-2}) & if $\lambda_m$ is even
                 \end{cases*}\end{equation}
We will now define $\rho(\lambda)=(\alpha,\beta)$ in terms of $\rho(\mu)=(\gamma,\delta)$. There are eight cases:
\vspace{2mm}
\begin{itemize}
    \item[$(a)$] If $d(\lambda_m)=0$, set $r=\lfloor \frac{1}{4}(\lambda_m+2)\rfloor - d(\mu)$, $d=\lfloor \frac{1}{4}\lambda_m\rfloor + d(\mu)$. 
    \begin{itemize}
        \item[$(a_1)$] If $d(\mu)<0$, then $\alpha=(\gamma,r)$ and $\beta=(\delta,s)$.
        \item[$(a_2)$] If $d(\mu) \leq 0$, then $\alpha=(\gamma,s)$ and $\beta=(\gamma,r)$.
    \end{itemize}
    \item[$(b)$] If $d(\lambda_m)=1$, set $r=\frac{1}{4}(\lambda_m-1)-d(\mu)$.
    \begin{itemize}
        \item[$(b_1)$] If $d(\mu)>0$, then $\alpha=(\gamma,r)$ and $\beta=\delta$.
        \item[$(b_2)$] If $d(\mu)=0$, then $\alpha=(\delta,r)$ and $\beta=\gamma$.
        \item[$(b_3)$] If $d(\mu)<0$, then $\alpha=\gamma$ and $\beta=(\delta,r)$.
    \end{itemize}
    \item[$(c)$] If $d(\lambda_m)=-1$, set $r=\frac{1}{4}(\lambda_m-3)+d(\mu)$.
    \begin{itemize}
        \item[$(c_1)$] If $d(\mu)>1$, then $\alpha=\gamma$ and $\beta=(\delta,r)$.
        \item[$(c_2)$] If $d(\mu)=1$, then $\alpha=(\delta,r)$ and $\beta=\gamma$.
        \item[$(c_3)$] If $d(\mu)<1$, then $\alpha=(\gamma,r)$ and $\beta=\delta$.
    \end{itemize}
\end{itemize}
It is shown in \cite[Lemma 4.12]{LusztigSpaltenstein1985} that $\rho(\lambda) \in \mathcal{B}(\frac{1}{4}(N-d(\lambda)(2d(\lambda)-1)))$. 

Below, it will be convenient to have a non-recursive formula for $\rho(\lambda)$. This will require some additional notation. For $n \in \ZZ$, let $r(n) \in \{0,1,2,3\}$ denote the residue of $n$ mod $4$ and let $q(n) \in \ZZ$ be such that $n=4q(n)+r(n)$. For $\lambda \in \mathcal{P}_{ro}(n)$ and $x \in \{0,1,2,3\}$, define the partition
$$\lambda^x = (q(\lambda_i) \mid r(\lambda_i)=x)$$
For any sequence $\mu=(\mu_1,...,\mu_m)$ of the same length as $\lambda$ define
$$\mu^x = (\mu_i \mid r(\lambda_i)=x)$$
Also let
$$D(\lambda) = (\sum_{j<i} d(\lambda_j))_i$$
Since every even part of $\lambda$ has even multiplicity, every part of $\lambda^x$, for $x\in\{0,2\}$, has even multiplicity. So, for $x \in \{0,2\}$, there is a partition $\underline{\lambda}^x$ such that 
$$\lambda^x = \underline{\lambda}^x \cup \underline{\lambda}^x$$
Similarly, for $x \in \{0,2\}$, there is a partition $\underline{D}(\lambda)^x$ such that
$$D(\lambda)^x = \underline{D}(\lambda)^x \cup \underline{D}(\lambda)^x$$

\begin{prop}\label{prop:spinnonrecursive}
Let $\lambda \in \mathcal{P}_{ro}(n)$ be rather odd and put $D=D(\lambda)$. Then
$$\rho(\lambda) = \begin{cases*}
                    (\underline{\lambda}^0 - \underline{D}^0) \cup (\lambda^1-D^1) \cup(\underline{\lambda}^2 + 1 - \underline{D}^2) \times (\underline{\lambda}^0+\underline{D}^0) \cup (\lambda^3+D^3) \cup (\underline{\lambda}^2+\underline{D}^2) & if  $d(\lambda)>0$  \\
                     (\underline{\lambda}^0+\underline{D}^0) \cup (\lambda^3+D^3) \cup (\underline{\lambda}^2+\underline{D}^2) \times (\underline{\lambda}^0 - \underline{D}^0) \cup (\lambda^1-D^1) \cup(\underline{\lambda}^2 + 1 - \underline{D}^2) & if $d(\lambda) \leq 0$
                 \end{cases*} $$
\end{prop}

\begin{proof}
Let $\rho^*(\lambda)$ denote the right hand side of the equation in the statement of the proposition. We will show by induction on the number of parts in $\lambda$ that $\rho^*(\lambda)=\rho(\lambda)$. Fix $\mu$ as in (\ref{eq:mu}) and let $(\gamma,\delta) = \rho(\mu) = \rho^*(\mu)$. There are eight cases to consider:
\begin{itemize}
    \item[$(a_1)$] \underline{$d(\lambda_m)=0$ and $d(\mu)>0$}.
    Let $r=\lfloor\frac{1}{4}(\lambda_m+2)\rfloor - d(\mu)$ and $s=\lfloor\frac{1}{4}\lambda_m\rfloor + d(\mu)$. Then by the definition of $\rho$
    $$\rho(\lambda) = (\gamma,r) \times (\delta,s)$$
    On the other hand $D(\lambda) = (D(\mu),d(\mu),d(\mu))$, $D(\lambda)^1= D(\mu)^1$, $D(\lambda)^3 = D(\mu)^3$, $\lambda^1 = \mu^1$, and $\lambda^3 = \mu^3$. To compute $\rho^*(\lambda)$, there are two subcases to consider. First assume that $\lambda_m \equiv 0 \mod{4}$, then $\underline{D}(\lambda)^0= (\underline{D}(\mu)^0, d(\mu))$, $\underline{D}(\lambda)^2 = \underline{D}(\mu)^2$, $\underline{\lambda}^0 = (\underline{\mu}^0, \frac{1}{4}\lambda_m)$, and $\underline{\lambda}^2 = \underline{\mu}^2$.
    So
    $$\rho^*(\lambda) = (\gamma, \frac{1}{4}\lambda_m - d(\mu)) \times (\delta,\frac{1}{4}\lambda_m + d(\mu)) = (\gamma,r) \times (\gamma,s)$$
    If instead $\lambda_m \equiv 2 \mod{4}$, then $\underline{D}(\lambda)^0 = \underline{D}(\mu)^0$, $\underline{D}(\lambda)^2 = (\underline{D}(\mu), d(\mu))$, $\underline{\lambda}^0 = \underline{\mu}^0$, and $\underline{\lambda}^2 = (\underline{\mu}^2, \frac{1}{4}(\lambda_m-2))$. So
    $$\rho^*(\lambda) = (\gamma, \frac{1}{4}(\lambda_m-2)+1-d(\mu)) \times (\delta,\frac{1}{4}(\lambda_m-2)+d(\mu)) = (\gamma,r) \times (\gamma,s)$$
    \item[$(a_2)$]  \underline{$d(\lambda_m)=0$ and $d(\mu)\leq 0$}. This is identical to case $(a_1)$ except with $\rho(\lambda)$ and $\rho^*(\lambda)$ flipped.
    
    \item[$(b_1)$] \underline{$d(\lambda_m)=1$ and $d(\mu)>0$}. Let $r=\frac{1}{4}(\lambda_m-1) - d(\mu)$. Then by the definition of $\rho$
    $$\rho(\lambda) = (\gamma,r) \times \delta.$$
    On the other hand, $D(\lambda)=(D(\mu), d(\mu))$, $\underline{D}(\lambda)^0=\underline{D}(\mu)^0$, $D(\lambda)^1 = (D(\mu)^1,d(\mu))$,  $\underline{D}(\lambda)^2=\underline{D}(\mu)^2$, $D(\lambda)^3=D(\mu)^3$, $\underline{\lambda}^0=\underline{\mu}^0$, $\lambda^1 = (\mu^1,\frac{1}{4}(\lambda_m-1))$, $\underline{\lambda}^2 = \underline{\mu}^2$, and $\lambda^3=\mu^3$. So
    $$\rho^*(\lambda) = (\gamma,\frac{1}{4}(\lambda_m-1) - d(\mu)) \times \delta = (\gamma,r) \times \delta$$
    \item[$(b_2)$] \underline{$d(\lambda_m)=1$ and $d(\mu)=0$}.Let $r=\frac{1}{4}(\lambda_m-1) - d(\mu)$. Then by the definition of $\rho$
    $$\rho(\lambda) = (\delta,r) \times \gamma.$$
    On the other hand, $D(\lambda)=(D(\mu), d(\mu))$, $\underline{D}(\lambda)^0=\underline{D}(\mu)^0$, $D(\lambda)^1 = (D(\mu)^1,d(\mu))$,  $\underline{D}(\lambda)^2=\underline{D}(\mu)^2$, $D(\lambda)^3=D(\mu)^3$, $\underline{\lambda}^0=\underline{\mu}^0$, $\lambda^1 = (\mu^1,\frac{1}{4}(\lambda_m-1))$, $\underline{\lambda}^2 = \underline{\mu}^2$, and $\lambda^3=\mu^3$. So
    $$\rho^*(\lambda) = (\delta,\frac{1}{4}(\lambda_m-1) - d(\mu)) \times \gamma = (\delta,r) \times \gamma$$
    \item[$(b_3)$] \underline{$d(\lambda_m)=1$ and $d(\mu)<0$}. This is identical to case $(b_1)$ except with $\rho(\lambda)$ and $\rho^*(\lambda)$ flipped.
    
    \item[$(c_1)$] \underline{$d(\lambda_m)=-1$ and $d(\mu)>1$}. Let $r=\frac{1}{4}(\lambda_m-3) + d(\mu)$. Then by the definition of $\rho$
    $$\rho(\lambda) = \gamma \times (\delta,r).$$
    On the other hand, $D(\lambda) = (D(\mu),d(\mu)-1)$, $\underline{D}(\lambda)^0=\underline{D}(\mu)^0$, $D(\lambda)^1=D(\mu)^1$, $\underline{D}(\lambda)^2=\underline{D}(\mu)^2$, $D(\lambda)^3 = (D(\mu)^3,d(\mu))$, $\underline{\lambda}^0 = \underline{\mu}^0$,$\lambda^1 = \mu^1$, $\underline{\lambda}^2=\underline{\mu}^2$, and $\lambda^3 = (\mu^3,\frac{1}{4}(\lambda_m-3))$. So
    $$\rho^*(\lambda) = \gamma \times (\delta, \frac{1}{4}(\lambda_m-3)+d(\mu)) = \gamma \times (\delta,r)$$

    \item[$(c_2)$] \underline{$d(\lambda_m)=1$ and $d(\mu)=1$}. This is identifical to case $(c_1)$ except with $\rho(\lambda)$ and $\rho^*(\lambda)$ flipped.

    \item[$(c_3)$] \underline{$d(\lambda_m)=-1$ and $d(\mu)<1$}.  Let $r=\frac{1}{4}(\lambda_m-3) + d(\mu)$. Then by the definition of $\rho$
    $$\rho(\lambda) = (\gamma,r) \times \delta.$$
    On the other hand, $D(\lambda) = (D(\mu),d(\mu))$, $\underline{D}(\lambda)^0=\underline{D}(\mu)^0$, $D(\lambda)^1=D(\mu)^1$, $\underline{D}(\lambda)^2=\underline{D}(\mu)^2$, $D(\lambda)^3 = (D(\mu)^3,d(\mu)-1)$, $\underline{\lambda}^0 = \underline{\mu}^0$,$\lambda^1 = \mu^1$, $\underline{\lambda}^2=\underline{\mu}^2$, and $\lambda^3 = (\mu^3,\frac{1}{4}(\lambda_m-3))$. So
    $$\rho^*(\lambda) = (\gamma,\frac{1}{4}(\lambda_m-3)+d(\mu)) \times \delta = (\gamma,r) \times \delta$$
\end{itemize}
\end{proof}

Proposition \ref{prop:spinnonrecursive} motivates the following much simpler recursive definition for $\rho(\lambda)$ which is obvious from Proposition \ref{prop:spinnonrecursive}.

Let $\lambda = (\lambda_1 \leq ... \leq \lambda_m) \in \mathcal{P}_{ro}(n)$.
If $m = 0$ set $\tilde\rho(\lambda) = (0)\times (0)$.
Otherwise, define $\mu$ as in (\ref{eq:mu}). Suppose $\tilde\rho(\mu) = \gamma\times\delta$.
We define $\tilde\rho(\lambda) = \alpha\times\beta$ where $\alpha,\beta$ are defined as follows: set $q = q(\lambda_m), d = d(\mu)$
\begin{enumerate}
    \item if $\lambda_m \equiv 0 \pmod 4$ then set $\alpha = (\gamma,q-d)$, $\beta = (\delta,q+d)$;
    \item if $\lambda_m \equiv 1 \pmod 4$ then set $\alpha = (\gamma,q-d)$, $\beta = \delta$;
    \item if $\lambda_m \equiv 2 \pmod 4$ then set $\alpha = (\gamma,q-d+1)$, $\beta = (\delta,q+d)$;
    \item if $\lambda_m \equiv 3 \pmod 4$ then set $\alpha = \gamma$, $\beta = (\delta,q+d)$.
\end{enumerate}
Then 
$$\rho(\lambda) := \begin{cases}
\alpha\times \beta & \text{if } d(\lambda)>0 \\
\beta\times\alpha & \text{if } d(\lambda) \le 0.\end{cases}$$

Finally we remark that from the above recursive description of $\tilde\rho(\lambda)$ one can deduce the following explicit description of $\tilde\rho(\lambda)$.

For $r\in \{0,1,2,3\}$ and $a\ge 0$, define $\epsilon(r,a)$ to be the sequence of length $a$ of the form
\begin{equation}
    \begin{cases}
        (0,1,0,1,\dots) & \mbox{if } r=0 \\
        (0,0,0,0,\dots) & \mbox{if } r=1 \\
        (1,0,1,0,\dots) & \mbox{if } r=2 \\
        (1,1,1,1,\dots) & \mbox{if } r=3.
    \end{cases}
\end{equation}
Similarly for $r\in \{0,1,2,3\}$ and $a\ge 0$, define $\delta(r,a)$ to be the sequence of length $a$ of the form
\begin{equation}
    \begin{cases}
        (0,1,0,1,\dots) & \mbox{if } r=0 \\
        (0,0,0,0,\dots) & \mbox{if } r=1,2,3.
    \end{cases}
\end{equation}
Now for $\lambda\in \cN_{ro}$ write $\lambda$ as $(\lambda_1^{a_1},\lambda_2^{a_2},\dots,\lambda_k^{a_k})$.
Define 
\begin{align*}
    q(\lambda) &= (q(\lambda_1)^{a_1},q(\lambda_2)^{a_2},\dots,q(\lambda_k)^{a_k}), \\
    r(\lambda) &= (r(\lambda_1)^{a_1},r(\lambda_2)^{a_2},\dots,r(\lambda_k)^{a_k}), \\
    \epsilon(\lambda) &= (\epsilon(r_1,a_1),\epsilon(r_2,a_2),\dots,\epsilon(r_k,a_k)), \\
    \delta(\lambda) &= (\delta(r_1,a_1),\delta(r_2,a_2),\dots,\delta(r_k,a_k)).
\end{align*}
When $\lambda$ is clear from the context, we will write $q_i$ (resp. $r_i$) for $q(\lambda_i)$ (resp. $r(\lambda_i)$).

Recall from Section \ref{subsec:symbols} the sequence $z=(0,1,...)$.
\begin{prop}
    \label{prop:ssymbspin}
    Let $\lambda \in \cN_{ro}$, $q=q(\lambda),\epsilon=\epsilon(\lambda),\delta=\delta(\lambda)$.
    Let $\Lambda\in \mathcal X^{2,0}_{d(\lambda)}$ be the symbol with linear presentation $(q+z-\delta,\epsilon)$.
    Then $\tilde\rho(\lambda) = \mathcal B(\Lambda)$ and $(q+z-\delta,\epsilon)$ is the canonical linear presentation for $\Lambda$.
\end{prop}
\begin{proof}
    Let $d_i = \sum_{j<i}d(\lambda_j)$, $A_i = \sum_{j<i}a_j$, and
    $$x_i = \sum_{j<i}[r_j=1]+[2\mid\lambda_j]a_j/2, \quad y_i = \sum_{j<i}[r_j=3]+[2\mid\lambda_j]a_j/2.$$
    Clearly $d_i = x_i-y_i$ and $x_i+y_i = A_i$.
    Write $\tilde\rho(\lambda)=(B_1,B_2)$.
    By the recursive formula for $\tilde\rho(\lambda)$ we have that $B_1$ consists of blocks of sequences of the form
    \begin{equation}
        \begin{cases}
            q_i-d_i+(0,0,\dots,0) & \mbox{if } r_i=0 \\
            q_i-d_i+(0) & \mbox{if } r_i=1 \\
            q_i-d_i+1+(0,0,\dots,0) & \mbox{if } r_i=2 \\
            \emptyset & \mbox{if } r_i=3 
        \end{cases}
    \end{equation}
    and $B_2$ consists of blocks of sequences of the form
    \begin{equation}
        \begin{cases}
            q_i+d_i+(0,0,\dots,0) & \mbox{if } r_i=0 \\
            \emptyset & \mbox{if } r_i=1 \\
            q_i+d_i+(0,0,\dots,0) & \mbox{if } r_i=2 \\
            q_i+d_i+(0) & \mbox{if } r_i=3.
        \end{cases}
    \end{equation}
    Here the sequences of the form $(0,0,\dots,0)$ have length $a_i/2$.
    We wish to compute $\Lambda = (B_1,B_2)+2((0,1,2,\dots),(0,1,2,\dots))$.
    To compute $B_1+2(0,1,2,\dots)$ we compute along blocks.
    The sequence $(0,1,2,\dots)$ can be written as a concatenation of sequences of the form
    \begin{equation}
        \begin{cases}
            x_i+(0,1,\dots,a_i/2-1) & \mbox{if } r_i=0 \\
            x_i+(0) & \mbox{if } r_i=1 \\
            x_i+(0,1,\dots,a_i/2-1) & \mbox{if } r_i=2 \\
            \emptyset & \mbox{if } r_i=3.
        \end{cases}
    \end{equation}
    These blocks have the same form as the blocks for $B_1$ so we can add block-wise.
    Noting that $2x_i-d_i = x_i+y_i = A_i$, we get that $B_1+2(0,1,2,\dots)$ consists of blocks of the form
    \begin{equation}
        \begin{cases}
            q_i+A_i+(0,2,\dots,a_i-2) & \mbox{if } r_i=0 \\
            q_i+A_i+(0) & \mbox{if } r_i=1 \\
            q_i+A_i+(1,3,\dots,a_i-1) & \mbox{if } r_i=2 \\
            \emptyset & \mbox{if } r_i=3.
        \end{cases}
    \end{equation}
    Similarly, using that $d_i+2y_i = x_i+y_i = A_i$, we get that $B_2+2(0,1,2,\dots)$ consists of blocks of the form
    \begin{equation}
        \begin{cases}
            q_i+A_i+(0,2,\dots,a_i-2) & \mbox{if } r_i=0 \\
            \emptyset & \mbox{if } r_i=1 \\
            q_i+A_i+(0,2,\dots,a_i-2) & \mbox{if } r_i=2 \\
            q_i+A_i+(0) & \mbox{if } r_i=3.
        \end{cases}
    \end{equation}
    Thus we see that $(q+z-\delta,\epsilon)$ is a linear presentation of $\Lambda$.
    To see that it is the canonical presentation for $\Lambda$ note that $q$ is non-decreasing and $z$ is strictly increasing. 
    Thus $q+z$ is strictly increasing.
    Thus $q+z-\delta$ is strictly increasing except for in blocks corresponding to $r_i = 0$ where the block takes the form $(x,x,x+2,x+2,x+4,x+4,\dots)$.
    But for such blocks $\epsilon$ is of the form $(0,1,0,1,0,1,\dots)$ and so $(q+z-\delta,\epsilon)$ is indeed the canonical linear presentation for $\Lambda$.
\end{proof}

\begin{example}
    Let $\lambda = (2^4,5,7,12^6,15)$.
    Then 
    \begin{align*}
        q(\lambda) &= (0,0,0,0,1,1,3,3,3,3,3,3,3)\\
        r(\lambda) &= (2,2,2,2,1,3,0,0,0,0,0,0,3)\\
        \epsilon(\lambda) &= (1,0,1,0,0,1,0,1,0,1,0,1,1)\\
        \delta(\lambda) &= (0,0,0,0,0,0,0,1,0,1,0,1,0)\\
        q(\lambda) + z - \delta(\lambda) &= (0,1,2,3,5,6,9,9,11,11,13,13,15) \\
        \Lambda &= \begin{pmatrix}
            1 & 3 & 5 & 9 & 11 & 13 \\
            0 & 2 & 6 & 9 & 11 & 13 & 15
        \end{pmatrix} \\
        \tilde\rho(\lambda) &= \begin{pmatrix}
            1 & 1 & 1 & 3 & 3 & 3 \\
            0 & 0 & 2 & 3 & 3 & 3 & 3
        \end{pmatrix} \\
        \rho(\lambda) &= \begin{pmatrix}
            0 & 0 & 2 & 3 & 3 & 3 & 3\\
            1 & 1 & 1 & 3 & 3 & 3 
        \end{pmatrix}.
    \end{align*}
    
\end{example}

We now give an explicit description of $d_{BV}(\lambda)$ for $\lambda\in \mathcal{P}_{ro}(N)$.

Let $\lambda \in \mathcal P_{ro}(N)$. Write $\lambda = (\lambda_1^{a_1},\lambda_2^{a_2},\dots,\lambda_k^{a_k})$ for $\lambda_1<\cdots<\lambda_k$.
Let 
$$f:\{1,\dots,\tilde l\}\to \{1,\dots,k\}$$
be the increasing function which enumerates the odd parts of $\lambda$. For example, if $\lambda=(1,2^2,5,8^4,9)$, then $\tilde{l}=3$ and $f$ is given by $f(1)=1$, $f(2)=3$, $f(3)=5$. Note that $N$ and $\tilde l$ have the same parity. Extend the domain of $f$ to a set of even size, as follows. If $N$ is odd, write $\tilde l$ as $\tilde l = 2l-1$ and define $f(2l) = k$.
In this manner we obtain a function 
$$f:\{1,\dots,2l\}\to \{1,\dots,k\}.$$
Define $\delta^i$ to be the sequence of length $k$ of the form
\begin{equation}
    \delta^i = \begin{cases}
        (0,\dots,0,1,-1,0,\dots,0) & \text{ if $i$ is even}  \\
        (0,\dots,0,-1,1,0,\dots,0) & \text{ if $i$ is odd}.
    \end{cases}
\end{equation}
where the first non-zero entry arises at position $f(i)$.
If $f(i)=k$ then we interpret $\delta^i$ as the sequence
\begin{equation}
    \delta^i = \begin{cases}
        (0,\dots,0,1) & \text{ if $i$ is even}  \\
        (0,\dots,0,-1) & \text{ if $i$ is odd}.
    \end{cases}
\end{equation}
\begin{lemma}\label{lem:dualro}
    Let $\lambda\in \mathcal P_{ro}(N)$.
    Let $\delta = \sum_{i=1}^{\tilde l}\delta^i$, $A_i = \sum_{j\ge i} a_j$, and $\bar \lambda_i =  \lambda_i-\lambda_{i-1}$ (with the convention $\lambda_0 = 0$).
    Then 
    \begin{equation}\label{eq:dro}d_{BV}(\lambda) = (A_i^{\bar\lambda_i+\delta_i})_i.\end{equation}
\end{lemma}
\begin{proof}
If $a$ is even, then $d_{BV}(\lambda \cup (a,a)) = d_{BV}(\lambda)+(2^a)$. So if (\ref{eq:dro}) holds for $\lambda$, it also holds for $\lambda \cup (a,a)$. Thus, we can reduce to the case when $\lambda$ consists of only odd parts, i.e. $\lambda = (\lambda_1,\lambda_2, ...,\lambda_k)$ with all $\lambda_i$ odd. In this case, 
$$\lambda^t = (k^{\overline{\lambda}_1},(k-1)^{\overline{\lambda}_2},...,2^{\overline{\lambda}_{k-1}},1^{\overline{\lambda}_k}).$$
By the standard algorithm for the collapse of a partition (cf. \cite[Lemma 6.3.3]{CM})
$$d_{BV}(\lambda) = (k^{\overline{\lambda}_1+\delta_1},(k-1)^{\overline{\lambda}_2+\delta_2},...,2^{\overline{\lambda}_{k-1}+\delta_{k-1}},1^{\overline{\lambda}_k+\delta_k})$$
where
$$\delta= \begin{cases} (-1,2,-2,2,...,-2,2)& \text{ if $N$ is even}\\
(-1,2,-2,2,...,2,-2)  & \text{ if $N$ is odd}.
\end{cases}.$$
But in both cases, $\delta=\sum_{i=1}^{\tilde{l}}\delta^i$. This completes the proof.
\end{proof}
\begin{cor}
    \label{cor:dlambda}
    There exists a partition $\underline{d_{BV}(\lambda)}$ such that $d_{BV}(\lambda) = \underline{d_{BV}(\lambda)}\cup \underline{d_{BV}(\lambda)}$.
\end{cor}
\begin{proof}
    It suffices to show that $\bar\lambda_i+\delta_i$ is even for all $i$.
    Suppose $\lambda_i,\lambda_{i-1}$ are both even. Then $\delta_i=0$ and $\bar\lambda_i +\delta_i = \lambda_i - \lambda_{i-1}$ is even.
    If $\lambda_{i-1}$ is odd and $\lambda_i$ is even, or $\lambda_{i-1}$ is even and $\lambda_i$ is odd, then $\delta_i \equiv 1 \pmod 2$ so $\bar\lambda_i+\delta_i$ is even.
    Finally, if both $\lambda_{i-1},\lambda_i$ are odd then $\delta_i \equiv 1+1 \pmod 2$ and so $\bar\lambda_i + \delta_i \equiv 0 \pmod 2$.
\end{proof}

\subsection{Unipotent representations of finite groups of Lie type}\label{subsec:unipotent}

In this section, we will collect various facts about unipotent representations of finite classical groups of Lie type (e.g. classification, decomposition into Harish-Chandra series, decomposition into families, computation of Kawanaka wavefront sets). Essentially all of this information appears in \cite{Lusztig1984}. However, the presentation below is from \cite{geckmalle}. 

Let $\mathbf G$ be a connected reductive group defined over $\mathbb F_q$ and write $\unip(\mathbf G(\mathbb F_q))$ for the set of unipotent representations of $\mathbf G(\mathbb F_q)$.
The set $\unip(\mathbf G(\mathbb F_q))$ decomposes into \emph{Harish-Chandra series}
$$\unip(\mathbf G(\mathbb F_q)) = \coprod_{(\mathbf L,\bf E)}\unip_{\mathbf L,\bf E}(\mathbf G(\mathbb F_q))$$
where $(\mathbf L,\bf E)$ ranges over $\mathbf G(\mathbb F_q)$-conjugacy classes of pairs consisting of a $\bfG$-split Levi subgroup $\mathbf L$ of $\mathbf G$ and a cuspidal unipotent representation $\bf E$ of $\mathbf L(\mathbb F_q)$, and $\unip_{\mathbf L,\bf E}(\mathbf G(\mathbb F_q))$ denotes the subset of $\unip(\mathbf G(\mathbb F_q))$ consisting of unipotent representations with cuspidal support $(\mathbf L,\bf E)$.
Moreover $\unip_{\mathbf L,\bf E}(\mathbf G(\mathbb F_q))$ is parameterised by $\mathrm {Irr}(W_{\mathbf L})$ where $W_{\mathbf L}:=N_{\mathbf G(\mathbb F_q)}(\mathbf L(\mathbb F_q))/\mathbf L(\mathbb F_q)$.
Thus there is a bijection
\begin{equation}
    \label{eq:hcseries}
    \unip(\mathbf G(\mathbb F_q)) \leftrightarrow \coprod_{(\mathbf L,\bf E)}\mathrm {Irr}(W_{\mathbf L}).
\end{equation}
where, again, the union runs over all $\bfG(\mathbb{F}_q)$-conjugacy classes of pairs  $(\mathbf L,\bf E)$ consisting of a $\bfG$-split Levi subgroup $\mathbf L$ of $\mathbf G$ and a cuspidal unipotent representation $\bf E$ of $\mathbf L(\mathbb F_q)$. 

There is also a decomposition of $\unip(\mathbf G(\mathbb F_q))$ into families. It is clear from \cite[Section 13.4]{Lusztig1984} that $\WF$ is constant on families, so for a family $\phi$, the notation $\WF(\phi)$ makes sense.
In fact $\WF$ induces a bijection between families and special nilpotent orbits of $\bfG$ defined over $\mathbb F_q$.
It follows, and will be useful to note, that a unipotent representation $X$ lies in a family $\phi$ if and only if $\WF(X) = \WF(\phi)$.
%Each family $\phi \subset \unip(\mathbf G(\mathbb F_q))$ contains a unique special representation, denoted

Below, we list all simple finite groups of Lie type which arise in reductive quotients of classical $p$-adic groups, inner to split. For each group, we give
\begin{enumerate}
    \item A parameterization of the set $\mathrm{Unip}(\mathbf{G}(\mathbb{F}_q))$. See \cite[Theorem 4.4.13]{geckmalle}.
    \item A description of the families in $\mathrm{Unip}(\mathbf{G}(\mathbb{F}_q))$. See \cite[Proposition 4.4.20]{geckmalle}.
    \item A formula for the Kawanaka wavefront set of the elements of $\mathrm{Unip}(\mathbf{G}(\mathbb{F}_q))$, in terms of their parameters.
    \item A list of all Levi subgroups (up to conjugation) which admit cuspidal representations. In all the cases we consider, there is at most one cuspidal per Levi. See \cite[Theorem 4.4.28]{geckmalle}.
    \item A description of the relative Weyl groups corresponding to the Levi subgroups in (4). See \cite[Proposition 4.4.29]{geckmalle}.
    \item A description of the matching (\ref{eq:hcseries}). See \cite[Proposition 4.4.29]{geckmalle}.
\end{enumerate}

\subsubsection{$A_{n-1}(q)$.} 
\label{subsubsec:unipA}
Let $\mathbf G(\mathbb F_q) = PGL(n,\mathbb F_q)$. 

\begin{enumerate} 
    \item $\unip(\mathbf G(\mathbb F_q)) \leftrightarrow \mathcal P(n)$ 
    \item Every family is a singleton. 
    \item $\WF(\lambda) = \lambda^t$.
    \item The only Levi supporting a cuspidal representation is the maximal torus. 
    \item The relative Weyl is $S_n$.
    \item The matching is $\lambda \mapsto \lambda$.
\end{enumerate}

\subsubsection{$^2\!A_{n-1}(q^2)$.} Let $\mathbf G(\mathbb F_q) = PU(n,\mathbb F_q)$. 
\label{subsubsec:unip2A}

\begin{enumerate} 
    \item $\unip(\mathbf G(\mathbb F_q)) \leftrightarrow \mathcal P(n)$ 
    \item Every family is a singleton. 
    \item $WF(\lambda) = \lambda^t$.
    \item $\mathbf{L}$ is of type $A_{r(r+1)/2-1}$ for $r \geq 0$ such that $r(r+1)/2\le n$ and $2\mid n-r(r+1)/2$. 
    \item $W_{\mathbf{L}} = W_{\frac12(n-\frac12 r(r+1))}$, where $W_k$ denotes the Weyl group of type $B_k$. 
    \item Let $\lambda \in \mathcal P(n)$.
    The number of boxes in the young diagram of $\lambda$ with odd hook length minus the number of boxes with even hook length is of the form $\frac 12 s(s+1)$ for some $s\ge 0$.
    The partition $\lambda$ belongs to the Harish-Chandra series corresponding to $r=s$.
    Now let $Z$ be a $\beta$-set for $\lambda$ of odd length.
    Enumerate the even and odd parts of $Z$ as
    $$2x_1<2x_2<\cdots < 2x_k, 2y_1+1<2y_2+1<\cdots < 2y_l+1$$
    and let $X=\{x_1,x_2,\dots,x_k\}, Y=\{y_1,y_2,\dots,y_l\}$.
    If $s$ is even then $\lambda$ corresponds to the bipartition $(\mathcal P(X),\mathcal P(Y))$, and if $s$ is odd it corresponds to $(\mathcal P(Y),\mathcal P(X))$.
\end{enumerate}

\subsubsection{$B_n(q)$.} Let $\mathbf G(\mathbb F_q) = PSO(2n+1,\mathbb F_q)$. 
\label{subsubsec:unipB}

\begin{enumerate} 
    \item $\unip(\mathbf G(\mathbb F_q)) \leftrightarrow \mathcal Y_{odd,n}^1$. 
    \item The families are the $\sim$-classes in $\mathcal Y_{odd,n}^1$.
    \item $\WF(S) = \mathcal P(2X+1\cup 2Y)^t$, where $(X,Y) = S^{sp}$.
    \item $\mathbf{L}$ is of type $B_{r^2+r}$ for $r \geq 0$ such that $r^2+r\le n$. 
    \item $W_{\mathbf{L}} = W_{n-r^2-r}$. 
    \item Let $S\in \mathcal Y_{d,n}^1$ for $d\geq 1$ odd.
    Then $S$ belongs to the Harish-Chandra series corresponding to $r = (d-1)/2$.
    The corresponding bipartition is $\mathcal D(S)$.
\end{enumerate}

\subsubsection{$C_n(q)$.} Let $\mathbf G(\mathbb F_q) = PSp(2n,\mathbb F_q)$.
\label{subsubsec:unipC}

\begin{enumerate} 
    \item $\unip(\mathbf G(\mathbb F_q)) \leftrightarrow \mathcal Y_{odd,n}^1$. 
    \item The families are the $\sim$-classes in $\mathcal Y_{odd,n}^1$.
    \item $\WF(S) = \mathcal P(2X\cup 2Y+1)^t$, where $(X,Y) = S^{sp}$.
    \item $\mathbf{L}$ is of type $C_{r^2+r}$ for $r \geq 0$ such that $r^2+r\le n$. 
    \item $W_{\mathbf{L}} = W_{n-r^2-r}$. 
    \item Let $S\in \mathcal Y_{d,n}^1$ for $d \geq 1$ odd.
    Then $S$ belongs to the Harish-Chandra series corresponding to $r = (d-1)/2$.
    The corresponding bipartition is $\mathcal D(S)$.
\end{enumerate}

\subsubsection{$D_n(q)$.} Let $\mathbf G(\mathbb F_q) = PCO^+(2n,\mathbb F_q)$.
\label{subsubsec:unipD}

\begin{enumerate} 
    \item $\unip(\mathbf G(\mathbb F_q)) \leftrightarrow \sqcup_{d\geq 0, 4\mid d}\mathcal Y_{d,n}^1$ except to each degenerate symbol in $\mathcal Y_{0,n}^1$ there are two unipotent representations.
    \item The families are the $\sim$-classes in $\sqcup_{d\geq 0, 4\mid d}\mathcal Y_{d,n}^1$.
    \item $\WF(S) = (\mathcal P(2X+1\cup 2Y)^t)_D$, where $(X,Y) = S^{sp}$.
    \item $\mathbf{L}$ is of type $D_{r^2}$ for $r \geq 0$ even such that $r^2\le n$. 
    \item $W_{\mathbf{L}} = W_{n-r^2}$ if $r \geq 2$ and $W_n'$ if $r=0$. 
    \item Let $S\in \mathcal Y_{d,n}^1$ for $d\in \ZZ, 4\mid d$. 
    Then $S$ belongs to the Harish-Chandra series corresponding to $r = d/2$.
    The corresponding bipartition is $\mathcal D(S)$. 
    In the case where $S$ is degenerate there are two unipotent representations parameterised by $S$, but also two irreducible representations of $W_n'$ parameterised by $\mathcal D(S)$ and these are matched in some manner (the precise nature of which is not important to us).
\end{enumerate}

\subsubsection{$^2\!D_n(q^2)$.} Let $\mathbf G(\mathbb F_q) = PCO^-(2n,\FF_q)$. 
\label{subsubsec:unip2D}

\begin{enumerate} 
    \item $\unip(\mathbf G(\mathbb F_q)) \leftrightarrow \sqcup_{d\geq 0, 4\mid d-2}\mathcal Y_{d,n}^1.$ 
    \item The families are the $\sim$-classes in $\sqcup_{d\geq 0, 4\mid d-2}\mathcal Y_{d,n}^1$.
    \item $\WF(S) = (\mathcal P(2X+1\cup 2Y)^t)_D$, where $(X,Y) = S^{sp}$.
    \item $\mathbf{L}$ is of type $^2D_{r^2}$ for $r \geq 1$ odd such that $r^2\le n$. Here $^2D_1$ is regarded as the empty graph.
    \item $W_{\mathbf{L}} = W_{n-r^2}$. 
    \item Let $S\in \mathcal Y_{d,n}^1$ for $d\geq 0, 4\mid d-2$. Then $S$ belongs to the Harish-Chandra series corresponding to $r = d/2$.
    The corresponding bipartition is $\mathcal D(S)$. 
\end{enumerate}

%\subsubsection{Exceptional groups}
%For the exceptional groups we refer to the tables in \cite[Section 4.5]{geckmalle}.

\subsection{The generalised Springer correspondence}\label{subsec:Springer}
In this section we will recall the statement of the generalised Springer correspondence and give a description of this correspondence in classical types. The main ideas are from \cite{Lusztig1984intersection} and \cite{LusztigSpaltenstein1985}.

Let $G$ be a complex reductive group, $T$ a maximal torus of $G$, $B$ a Borel subgroup of $G$ containing $T$, and $I_0$ the set of simple roots determined by $T$ and $B$. 
For a subset $J\subseteq I_0$ write $G_J$ for the Levi subgroup of $G$ determined by $T$ and $J$.
For a Levi subgroup $L$ of $G$ write $W_L = N_G(L)/L$.
Let $\mathcal N_{o,l}$ be the set of pairs $(\OO,\mathscr L)$ consisting of a nilpotent orbit $\OO\in \mathcal N_o$ and a $G$-equivariant irreducible local system $\mathscr L$ on $\OO$. 
The generalised Springer correspondence is a bijection
\begin{equation}\label{eq:generalizespringer}
    \mathcal N_{o,l} \leftrightarrow \coprod_{(L,\mathcal C,\mathscr E)\in \mathfrak C(G)} \mathrm{Irr}(W_L), \quad 
(\OO, \rho)\mapsto E(\OO,\rho)
\end{equation}
where $\mathfrak C(G)$ is the set of $G$-conjugacy classes of triples $(L,\mathcal C,\mathcal E)$ where $L$ is a Levi subgroup of $G$, $\mathcal C$ is a nilpotent orbit of the Lie algebras $\mathfrak l$ of $L$, and $\mathcal E$ is a $L$-equivariant cuspidal local system on $\mathcal C$.
Let $\bar {\mathfrak C}(G)$ be the set of all triples $(J,\mathcal C,\mathcal{E})$ where $J\subseteq I_0$, $\mathcal C$ is a nilpotent orbit of the Lie algebra $\mathfrak g_J$ of $G_J$, and $\mathcal E$ is a $G_J$-equivariant cuspidal lcoal system on $\mathcal C$. 
There is the obvious map $\bar{\mathfrak C}(G)\to \mathfrak C(G)$, given by $(J,\mathcal C,\mathcal E)\mapsto (G_J,\mathcal C,\mathcal E)$.
It is clear from form of the Levis that admit cuspidal local systems that this map is a bijection.
For $x = (J,\mathcal C,\mathcal E)\in \bar{\mathfrak C}(G)$ we will write $L_x$ for $G_J$ and $W_x$ for $W_{L_x}$ .
We write $(\mathcal N_{o,l})_x$ for the pairs lying in the block parameterised by $x\in \bar{\mathfrak C}(G)$ and $(\cN_o)_x$ for $p_1((\cN_{o,l})_x)$.
So if $x=(J,\mathcal C,\mathcal E)$ then 
$$(\mathcal N_{o,l})_x \leftrightarrow \mathrm{Irr}(W_{x}).$$
Note that for a Levi subgroup $L \subset G$ and a character $\chi$ of $Z(L)$, there is at most one pair $(\mathcal C,\mathcal{E})$ such that $(L,\mathcal C,\mathcal{E}) \in \mathfrak{C}(G)$ and $Z(L)$ acts via (a multiple of) $\chi$ on $\mathcal{E}$. 

Below, we describe this bijection for all simple simply connected groups of classical type. For each group, we give

\begin{enumerate}
    \item A parameterization of $\mathcal{N}_{o,l}$. 
    \item A description of the projection $p_1:\mathcal{N}_{o,l} \to \cN_o$ in terms of their parameter.
    \item A list of all pairs $(L,\chi)$ consisting of a Levi subgroup $L \subset G$ and a character $\chi$ of $Z(L)$ such that there is a triple $(L,\OO,\mathcal{E}) \in \mathcal{C}$ and $Z(L)$ acts via (a multiple of) $\chi$ on $\mathcal{E}$. In most cases, $\chi$ is uniquely determined by $L$; in such cases we omit it.
    \item A description of the relative Weyl groups corresponding to the Levi subgroups in (3).
    \item A description of the matching (\ref{eq:generalizespringer}).
\end{enumerate}

The references are: \cite[Section 5.1]{LusztigSpaltenstein1985} for $SL(N)$, \cite[Section 10.6]{Lusztig1984intersection} for $SO(N)$, \cite[Section 10.4]{Lusztig1984intersection}and $Sp(2n)$, and \cite[Section 0.10, Section 4]{LusztigSpaltenstein1985} for $Spin(N)$. 

For $G=Spin(N)$, we present (1)-(5) in a slightly different manner. 
There are natural inclusions 
$$\cN_{o,l}^{SO(N)}\to \cN_{o,l}^{Spin(N)}, \quad \overline{\mathfrak C}(SO(N))\to \overline{\mathfrak C}(Spin(N)).$$
With respect to these inclusions there are decompositions
$$\cN_{o,l}^{Spin(N)} =   \mathcal N_{o,l}^{SO(N)}\sqcup \mathcal N_{o,l}^*$$
where $\mathcal{N}_{o,l}^*$ is the set of pairs $(\OO,\mathcal{L}) \in \mathcal{N}_{o,l}^{Spin(n)}$ consisting of a nilpotent orbit $\OO \in \cN_o$ and a $Spin(N)$-equivariant local system $\mathcal{L}$ on $\OO$ \emph{which is not} $SO(N)$\emph{-equivariant}, and
$$\overline{\mathfrak C}(Spin(N)) = \overline{\mathfrak C}(SO(N)) \sqcup \overline{\mathfrak C}^*(Spin(N))$$
where $\overline{\mathfrak{C}}^*$ is the set of triples $(J,\mathcal C,\mathcal{E}) \in \overline{\mathfrak C}(Spin(N))$ where $\mathcal{E}$ \emph{is not equivariant} for the associated Levi subgroup of $SO(N)$. The matching (\ref{eq:generalizespringer}) takes $\cN^{SO(N)}_{o,l}$ to $\overline{\mathfrak{C}}(SO(N))$ and $\cN_{o,l}^*$ to $\overline{\mathfrak C}^*(Spin(N))$. In (\ref{subsubsec:sospringerodd}) and (\ref{subsubsec:sospringereven}), we describe the correspondence for $\cN_{o,l}^{SO(N)}$, and in (\ref{subsubsec:spinspringerodd}) and (\ref{subsubsec:spinspringereven}), we describe the correspondence for $\cN_{o,l}^* $. 

\subsubsection{$SL(N)$}
\begin{enumerate}
    \item $\cN_{o,l}$ is parameterized by pairs $(\lambda,\rho)$, where $\lambda = (\lambda_1,...,\lambda_k) \in \mathcal{P}(N)$ and $\rho$ is a character of $\ZZ_d$, where $d=\mathrm{gcd}(\lambda_1,...,\lambda_k)$. 
    \item $p_1(\lambda,\rho)=\lambda$.
    \item $L$ is of type $(N/r)A_{r-1}$ for $r \mid N$, and $\chi$ is a character (trivial on connected center) of order $r$. 
    \item $W_{L}=S_{N/r}$.
    \item Let $\lambda=(\lambda_1,...,\lambda_k)$, $d=\mathrm{gcd}(\lambda_1,...,\lambda_k)$, and let $\rho$ be a character of $\ZZ_d$ of degree $m$. Then $(\lambda,\rho)$ belongs to a block with to $r=m$. The specific $\chi$ it corresponds to depends on $\rho$, but is not important for our purposes. The corresponding $S_{N/r}$-representation is given by the partition $(\lambda_1/r,...,\lambda_k/r)$ of $N/r$.
\end{enumerate}

\subsubsection{$SO(N)$, $N$ odd}\label{subsubsec:sospringerodd}
Let $n=\lfloor N/2 \rfloor$.
\begin{enumerate}
    \item $\mathcal{N}_{o,l}\leftrightarrow \mathcal Y^{2}_{odd, n}$.
    \item Let $S\in \mathcal Y^{2}_{odd,n}$.
    Then $p_1(S) = \mathcal P(2(X-z)+1 \cup 2(Y-z))$ where $(X,Y) = S^{sp}$.
    \item $L$ is of type $B_{(r^2-1)/2}$ for $r\ge 0$ odd such that $r^2\le N$.
    \item $W_L = W_{\frac12(N-r^2)}$.
    \item Let $S\in \mathcal Y^{2}_{d,n}$ for $d\ge 0$ odd. 
    Then $S$ belongs to the block corresponding to $r=d$.
    The corresponding bipartition is $\mathcal D(S)$.
\end{enumerate}

\subsubsection{$SO(N)$, $N$ even}\label{subsubsec:sospringereven}
Let $n=\lfloor N/2 \rfloor$.
\begin{enumerate}
    \item $\mathcal{N}_{o,l}\leftrightarrow \mathcal Y^{2}_{even, n}$.
    \item Let $S\in \mathcal Y^{2}_{even,n}$.
    Then $p_1(S) = \mathcal P(2(X-z)+1 \cup 2(Y-z))$ where $(X,Y) = S^{sp}$.
    \item $L$ is of type $D_{r^2/2}$ for $r\ge 0$ even such that $r^2\le N$.
    \item $W_L = W_{\frac12(N-r^2)}$ if $r>0$ and $W_n'$ if $r=0$.
    \item Let $S\in \mathcal Y^{2}_{d,n}$ for $d\ge 0$ even. 
    Then $S$ belongs to the block corresponding to $r=d$.
    The corresponding bipartition is $\mathcal D(S)$.
    The label of $\mathcal D(S)$ when $S$ is degenerate is not important for our purposes.
\end{enumerate}

\subsubsection{$Sp(2n)$}
\begin{enumerate}
    \item $\mathcal{N}_{o,l}\leftrightarrow \mathcal X^{1,1}_{odd, n}$.
    \item Let $S\in \mathcal X^{1,1}_{odd,n}$.
    Then $p_1(S) = \mathcal P(2(X-z) \cup 2(Y-z)+1)$ where $(X,Y) = S^{sp}$.
    \item $L$ is of type $C_{\frac 12r(r+1)}$ for $r\ge 0$ such that $\frac 12r(r+1)\le n$.
    \item $W_L = W_{n-\frac 12r(r+1)}$.
    \item Let $S\in \mathcal X^{1,1}_{d,n}$ for $d$ odd. 
    Then $S$ belongs to the block corresponding to $r=d-1$ if $d\ge 0$ and $r=-d$ if $d<0$.
    The corresponding bipartition is $\mathcal B(S)$ if $d\geq 0$ and the flip of $\mathcal{B}(S)$ otherwise.
\end{enumerate}

\subsubsection{$Spin(N)$, $N$ odd}\label{subsubsec:spinspringerodd}

\begin{enumerate}
    \item $\mathcal{N}_{o,l}^* \leftrightarrow \mathcal{P}_{ro}(N)$.
    \item $p_1(\lambda)=\lambda$.
    \item $L$ is of type $B_{(d-1)(2d+1)/2} + \frac{1}{4}(N-d(2d-1))A_1$ for $d\in \mathbb Z$, $d \equiv N \mod{4}$ and $d(2d-1) \leq N$. 
    \item $W_L= W_{(N-d(2d-1))/4}$.
    \item Let $\lambda \in \mathcal{P}_{ro}(N)$. Then $\lambda$ belongs to the block corresponding to $d=d(\lambda)$, cf. (\ref{eq:dlambda}), and the corresponding bipartition is $\rho(\lambda)$, cf. Proposition \ref{prop:spinnonrecursive}.
\end{enumerate}

\subsubsection{$Spin(N)$, $N$ even}\label{subsubsec:spinspringereven}

\begin{enumerate}
    \item $\mathcal{N}_{o,l}^* \leftrightarrow \mathcal{P}_{ro}(N) \sqcup \mathcal{P}_{ro}(N)$.
    \item $p_1(\lambda)=\lambda$.
    \item $L$ is of type $D_{d(2d-1)/2} + \frac{1}{4}(N-d(2d-1))A_1$ for $d\in \mathbb Z$, $d \equiv N \mod{4}$ and $d(2d-1)\leq N$.
    \item $W_L = W_{(N-d(2d-1))/4}$.
    \item Let $\lambda \in \mathcal{P}_{ro}(N)$. Then $\lambda$ belongs to the block corresponding to $d=d(\lambda)$, cf. (\ref{eq:dlambda}), and the corresponding bipartition is $\rho(\lambda)$, cf. Proposition \ref{prop:spinnonrecursive}.
\end{enumerate}

\subsection{The Arithmetic-Geometric correspondence}
Recall from Section \ref{sec:arithmeticgeometric} the definitions of $\bfG, G^\vee, \overline{\mathfrak S}, \overline{\mathfrak U}, \theta$ etc.
Recall from the previous section the definition of the set $\overline{\mathfrak C}(G^\vee)$ and note that there is a natural inclusion $\overline{\mathfrak C}(G^\vee)\to \overline{\mathfrak S}$.
In this section, we will give an explicit description of Lusztig's arithmetic-geometric correspondence (cf. Section \ref{sec:arithmeticgeometric}) for the geometric diagrams $\overline{\mathfrak C}(G^\vee)$ for the classical groups. 

Below we record 
\begin{enumerate}
    \item A list of all pairs $(L^{\vee},\chi)$ consisting of a Levi subgroup $L^{\vee} \subset G^\vee$ and a character $\chi$ of $Z(L^{\vee})$ such that there is a triple $(L^{\vee},\mathcal C^{\vee},\mathcal{E}^\vee) \in \overline{\mathfrak C}(G^{\vee})$ and $Z(L^{\vee})$ acts on the stalks of $\mathcal{E}^\vee$ by (a multiple of) $\chi$. In most cases, $\chi$ is uniquely determined by $L^{\vee}$; in such cases we omit it.
    \item The type of the affine Weyl group $\mathcal W^\vee(J^\vee,\mathcal C^\vee,\mathcal E^\vee)$.
    \item A description of $J$ and $\omega$ in $\theta(x) = (J,\bf E,\psi,\omega)$ for each $x\in \overline{\mathfrak C}(G^\vee)$.
    In all the cases we consider there is a unique choice of $\bf E$ and so we omit it.
\end{enumerate}
This information is taken from \cite[Section 7]{Lu-unip1}.

For $G^{\vee}=\mathrm{Spin}(N)$ we present this information in a slightly different manner. As discussed in Section \ref{subsec:Springer}, there is a decomposition
$$\overline{\mathfrak C}(Spin(N)) = \overline{\mathfrak C}(SO(N)) \sqcup \overline{\mathfrak C}^*(Spin(N)).$$
In (\ref{subsubsec:agSOodd}) and (\ref{subsubsec:agSOeven}), we describe the restriction of $\theta$ to $\overline{\mathfrak C}(SO(N))$, and in (\ref{subsubsec:agSpinodd}) and (\ref{subsubsec:agSpineven}), we describe the restriction of $\theta$ to $\overline{\mathfrak C}^*(Spin(N))$.

\subsubsection{$G^\vee=SL(N)$}
\begin{enumerate}
    \item $L^\vee$ is of type $(N/r)A_{r-1}$ for $r\mid N$, $\chi$ is of order $r$.
    \item $\tilde A_{N/r-1}$
    \item $J$ is $\emptyset$ and $\omega$ is of order $r$. 
\end{enumerate}

\subsubsection{$G^\vee = SO(N)$, $N$ odd}\label{subsubsec:agSOodd}
\begin{enumerate}
    \item $L^\vee$ is of type $B_{\frac 12(r^2-1)}$ for $r\ge 0$ odd and $\frac 12 (r^2-1) \le n$.
    \item $\tilde C_{\frac12 (N-r^2)}$ 
    \item $J$ is $C_{a^2+a}\times C_{a^2+a}$ where $a:=(r-1)/2$. $\omega = 1$.
\end{enumerate}

\subsubsection{$G^\vee = SO(N)$, $N$ even}\label{subsubsec:agSOeven}
\begin{enumerate}
    \item $L^\vee$ is of type $D_{r^2/2}$ for $r\ge 0$ even and $r^2/2 \le n$.
    \item $\tilde C_{\frac12 (N-r^2)}$ if $r>0$, $\tilde D_n$ if $r=0$.
    \item $J$ is $D_{a^2}\times D_{a^2}$ where $a:=r/2$. $\omega = 1$ if $a$ is even and $\omega $ of type I if $a$ is odd.
\end{enumerate}

\subsubsection{$G^\vee = Sp(2n)$}
\begin{enumerate}
    \item $L^\vee$ is of type $C_{\frac 12r(r+1)}$ for $r\ge 0$ and $\frac 12 r(r+1) \le n$.
    \item $\tilde C_{n-\frac 12r(r+1)}$ if $r>0$, $\tilde B_n$ if $r=0$..
    \item $J$ is $D_{a^2}\times B_{a^2+a}$ if r is even and $D_{a^2}\times B_{a^2-a}$ if $r$ is odd where $a:=\lfloor (r+1)/2\rfloor$. $\omega=1$ if $a$ is even and $\omega\ne 1$ if $a$ is odd.
\end{enumerate}

\subsubsection{$G^\vee = Spin(N)$, $N$ odd}\label{subsubsec:agSpinodd}
\begin{enumerate}
    \item $L^\vee$ is of type $B_{(d-1)(2d+1)/2} + \frac{1}{4}(N-d(2d-1))A_1$ for $d \equiv N \mod{4}$ and $d(2d-1) \leq N$. 
    \item $\tilde C_{(N-d(2d-1))/4}$
    \item $J$ is $C_{(d^2-1)/4}\times A_{d(d-1)/2-1}\times C_{(d^2-1)/4}$, $\omega \ne 1$.
\end{enumerate}

\subsubsection{$G^\vee = Spin(N)$, $N$ even}\label{subsubsec:agSpineven}
\begin{enumerate}
    \item $L^\vee$ is of type $D_{d(2d-1)/2} + \frac{1}{4}(N-d(2d-1))A_1$ for $d \equiv N \mod{4}$ and $d(2d-1)\leq N$. There are two possible $\chi$.
    \item $\tilde C_{(N-d(2d-1))/4}$
    \item $J$ is $D_{d^2/4}\times A_{d(d-1)/2-1}\times D_{d^2/4}$, $\omega$ of type II.
\end{enumerate}

\subsection{Faithful nilpotent orbits}
\label{subsec:faithful}
Recall the notation $\bfG, G^\vee, \overline{\mathfrak S}, \overline{\mathfrak U}, \theta$ from Section \ref{sec:arithmeticgeometric}, $\cN_{o,l},(\cN_{o,l})_x,W_x,\overline{\mathfrak C}$ from Section \ref{subsec:geometric}, $d_A, \mathcal K_I, \overline{\mathbb L}$ from Section \ref{sec:nil}, and $\unip, \unip_{\bfL,\bf E}, \WF(\phi)$ from Section \ref{subsec:unipotent}.

\begin{itemize}
    \item If $x=(J^\vee,\mathcal C^\vee,\mathcal E^\vee)\in \overline{\mathfrak S}$, and $\theta(x) = (K,\bf E,\psi,\omega)$, write $\mathcal W_x$ for the affine Weyl group $\mathcal W(K,\omega,\bf E)$ and $S_x$ for the Coxeter generators for $\mathcal W_x$.
    Recall from Section \ref{subsec:arithmetic} that $S_x\leftrightarrow (I - K)/\omega$. Write $W_x$ for the finite part of $\mathcal W_x$ (when $x\in \overline{\mathfrak C}(G^\vee)$ this agrees with the definition of $W_x$ given in Section \ref{subsec:Springer}). 
    Write $\pi:\mathcal W_x\to W_x$ for the projection map and $\mathcal I:(I-K)/\omega \xrightarrow{\sim} S_x$ for the bijection.
    \item Let $\mathcal J_x = \{J \mid K\subseteq J\subsetneq I, \ \omega J = J \}$. 
    \item Write $I_{0,K}$ for the element of $\mathcal J_x$ with $\mathcal I(I_{0,K}-K)$ equal to the Coxeter generators  of $W_x$.
    \item Let $\mathcal F_x = \{(J,\phi) \mid J\in \mathcal J_x, \ \phi \text{ a family of }\unip(\bfL_{c^\omega(J)}(\mathbb F_q)) \}$.
    \item For $J\in \mathcal J_x$ write $\mathcal W_{x,J}$ for $\langle y\mid y\in \mathcal I((J-K)/\omega)\rangle$ and $W_{x,J}$ for $\pi(\mathcal W_{x,J})$.
    %\item For $z\in S_x$ let $J_z=\mathcal I^{-1}(S_x\setminus\{z\})$.
    \item Let $J\in \mathcal J_x$.
    Then $\bfL_{c^\omega(K)}(\mathbb F_q)$ is a $\bfL_{c^\omega(J)}(\mathbb F_q)$-split Levi subgroup of $\bfL_{c^\omega(J)}(\mathbb F_q)$ affording the cuspidal representation $\bf E$, and $\mathcal H(\bfP_{c^\omega(J)(\mf o)},\bf E)$ (c.f. Section \ref{sec:multi}) has underlying coxeter system $\mathcal W_{x,J} \cong W_{x,J}$.
    Thus $\mathrm{Irr}(W_{x,J})$ (resp. $\mathrm{Irr}(\mathcal W_{x,J})$) parameterizes the irreducible representations in the Harish-Chandra series of $\bfL_{c^\omega(J)}(\mathbb F_q)$ determined by the cuspidal data $(\bfL_{c^\omega(K)},\bf E)$.
    Let $\mathcal U_{x,J}:\mathrm{Irr}(W_{x,J})\to \unip_{\bfL_{c^\omega(K)},\bf E}(\bfL_{c^\omega(J)}(\mathbb F_q))$ (resp. $\tilde{\mathcal U}_{x,J}:\mathrm{Irr}(\mathcal W_{x,J})\to \unip_{\bfL_{c^\omega(K)},\bf E}(\bfL_{c^\omega(J)}(\mathbb F_q))$) denote this parameterisation.
    We also write $\mathcal U_{x,J}$ for the composition 
    $$\mathrm{Irr}(W_{x,J})\xrightarrow{\sim} \unip_{\bfL_{c^\omega(K)},\bf E}(\bfL_{c^\omega(J)}(\mathbb F_q)) \hookrightarrow \unip(\bfL_{c^\omega(J)}(\mathbb F_q)).$$
\end{itemize}

%J depends on \OO^\vee
%Lemma:  K(\OO^\vee,\rho) \subseteq J for all \rho
%Defn: So it makes sense to talk about the W_x^J

\begin{definition}\label{def:faithful}
Let $x=(J^\vee,\mathcal C^\vee,\mathcal E^\vee)\in \overline{\mathfrak C}(G^\vee)$, $\theta(x) = (K,\bf E,\psi,\omega)$ and $\OO^{\vee} \in (\cN_o^\vee)_x$. We say that $\OO^{\vee}$ is \emph{$x$-faithful} if there exists a pair $(J, \phi) \in \mathcal{F}_{x}$ such that
\begin{itemize}
    \item[(i)] $\overline{\mathbb L}(J,\WF(\phi,\CC)) = d_A(\OO^{\vee},1)$,
\end{itemize}
and for every irreducible representation $\rho$ of $A(\OO^{\vee})$ such that $(\OO,\rho)\in (\cN_{o,l}^\vee)_x$,
\begin{itemize}
    \item[(ii)] there is a representation $F\in \mathrm{Irr}(W_{x,J})$ such that $\mathscr U_{x,J}(F) \in \phi$ and
     $$\Hom_{W_{x,J}}(F,E(\OO^\vee, \rho)|_{W_{x,J}}) \neq 0;$$
    \item[(iii)] for every $J'\in \mathcal J_x$ and $F'\in \mathrm{Irr}(W_{x,J'})$, if
     $$\Hom_{W_{x,J'}}(F',E(\OO^\vee, \rho)|_{W_{x,J'}}) \neq 0$$
     then 
     \begin{equation}
        \label{eq:bound}
        \overline{\mathbb L}(J',\WF(\mathcal U_{x,J'}(F'),\CC)) \le d_A(\OO^{\vee},1).
     \end{equation}
     %;
\end{itemize}
\end{definition}
Note that Equation \ref{eq:bound} is equivalent to the condition 
$$d_S(J',\WF(\mathcal U_{x,J'}(F'),\CC)) \ge \OO^\vee.$$
\begin{lemma}
    \label{lem:bound}
    It suffices to check (iii) in Definition \ref{def:faithful} for $J'$ maximal in $\mathcal J_x$ with respect to inclusion.
\end{lemma}
\begin{proof}
    Suppose condition (iii) of Definition \ref{def:faithful} holds for all $J'$ maximal.
    Let $J'\in \mathcal J_x$ and $F'\in \mathrm{Irr}(W_{x,J'})$ be such that
    $$\Hom_{W_{x,J'}}(F',E(\OO^\vee, \rho)|_{W_{x,J'}}) \neq 0.$$
    Let $J''\in \mathcal J_x$ be maximal among elements of $\mathcal J_x$ containing $J'$.
    Then there exists an $F''\in \mathrm{Irr}(W_{x,J''})$ such that 
    $$\Hom_{W_{x,J''}}(F'',E(\OO^\vee, \rho)|_{W_{x,J''}}) \neq 0$$
    and
    $$\Hom_{W_{x,J'}}(F',F''|_{W_{x,J'}}) \neq 0.$$
    By assumption
    $$d_S(J'',\WF(\mathcal U_{x,J''}(F''),\CC)) \ge \OO^\vee.$$
    Thus it suffices to prove
    $$d_S(J',\WF(\mathcal U_{x,J'}(F'),\CC)) \ge d_S(J'',\WF(\mathcal U_{x,J''}(F''),\CC)).$$
    By \cite[Proposition 2.31]{okada2021wavefront} and the remarks in the proof we have
    $$d_S(J',\WF(\mathcal U_{x,J'}(F'),\CC)) = d_S(J'',\mathrm{Sat}_{L_{J'}}^{L_{J''}}\WF(\mathcal U_{x,J'}(F'),\CC)).$$
    Thus it suffices to prove that for $F'\in \mathrm{Irr}(W_{x,J'}), F''\in \mathrm{Irr}(W_{x,J''})$
    \begin{equation}
        \label{eq:suff}
        \Hom_{W_{x,J'}}(F',F''|_{W_{x,J'}}) \neq 0\implies \mathrm{Sat}_{L_{J'}}^{L_{J''}}\WF(\mathcal U_{x,J'}(F'),\CC) \le \WF(\mathcal U_{x,J''}(F'')).
    \end{equation}
    By \cite[Theorem 5.9]{howlettlehrer}, 
    $$\mathcal U_{x,J''}\left(\mathrm{Ind}_{W_{x,J'}}^{W_{x,J''}}F'\right) = \mathrm{Ind}_{L_{J'}}^{L_{J''}}\mathcal U_{x,J'}(F')$$
    where $\mathrm{Ind}$ applied to unipotent characters means Harish-Chandra induction.
    Thus by Frobenius reciprocity we have that 
    \begin{equation}
        \label{eq:part1}
        \Hom_{W_{x,J'}}(F',F''|_{W_{x,J'}}) \neq 0\Leftrightarrow \mathcal U_{x,J''}(F'')\subseteq \mathrm{Ind}_{L_{J'}}^{L_{J''}}\mathcal U_{x,J'}(F').
    \end{equation}
    By \cite[Lemma 10.3]{taylorpham} (and noting that we may replace the Deligne-Lusztig induction in the lemma by Harish-Chandra induction since $\bfL_{c^\omega(J')}$ is a $\bfL_{c^\omega(J'')}$-split Levi of $\bfL_{c^\omega(J'')}$ by \cite[Section 2.1]{taylorpham}) we have that 
    \begin{equation}
        \label{eq:part2}
        \mathcal U_{x,J''}(F'')\subseteq \mathrm{Ind}_{L_{J'}}^{L_{J''}}\mathcal U_{x,J'}(F') \implies \mathrm{Sat}_{L_{J'}}^{L_{J''}}\WF(\mathcal U_{x,J'}(F'),\CC) \le \WF(\mathcal U_{x,J''}(F'')).
    \end{equation}
    Stringing together Equation \ref{eq:part1} and Equation \ref{eq:part2} gives Equation \ref{eq:suff} as required.
\end{proof}

\begin{theorem}
\label{thm:faithful}    
Suppose $\mathbf{G}$ is a simple group of adjoint type and let $\OO^{\vee} \in \cN_o^{\vee}$. Then for every $x \in \overline{\mathfrak C}(G^{\vee})$ such that $\OO^\vee\in (\cN_o)_x$, $\OO^{\vee}$ is $x$-faithful. 
\end{theorem}

\begin{proof}
We consider separately the cases of classical and exceptional groups.
If $\mathbf{G}$ is classical, the statement is proved in Section \ref{sec:faithfulclassical}. If $\mathbf{G}$ is exceptional, it is proved in Section \ref{sec:faithfulexceptional}.
\end{proof}

\subsection{Proof of faithfulness in classical types}\label{sec:faithfulclassical}
\subsubsection{$G^\vee = SL(N)$}
Let $n=N-1$.
Let $x\in \overline{\mathfrak C}(G^\vee)$ correspond to $\chi \in \mathrm{Irr}(\mathbb Z_N)$ of order $r\mid N$.
Then $\theta(x) = (\emptyset,1,\psi,\chi)$ and $\mathcal W_x$ is of type $\tilde A_{k-1}$ where $k = N/r$.
A nilpotent orbit $\OO^\vee\in \cN_o^\vee \leftrightarrow \lambda = (\lambda_1\le\lambda_2\le \dots \le \lambda_e)$ lies in $(\cN_o^\vee)_x$ iff $r\mid \lambda_i$ for all $i$.

We wish to show that $\lambda$ is $x$-faithful.
Let $J = rA_{k-1}$ and $\mu = (\lambda_1/r\le \lambda_2/r\le \dots \le \lambda_e/r)$. Then $\bfL_{c^\chi(J)}$ is isomorphic to $A_{k-1}(q)$ and so the families are parameterised by $\mathcal P(k)$.
Let $\phi$ be the family parameterised by the partition $\mu$.
Then $\bar {\mathbb L}(J,\WF(\phi,\CC)) = \mu^t \cup \mu^t\cup \cdots \cup\mu^t$ where there are $r$ terms in this expression.
But $\mu^t\cup\mu^t\cup\cdots\cup \mu^t = (\mu + \mu + \cdots + \mu)^t = \lambda^t = d(\OO^\vee,1)$.
Thus condition (i) holds.
Condition (ii) holds because $W_{x,J} = W_x$ and $E(\OO^\vee,\rho) = \mu$.
Finally condition (iii) holds by using Lemma \ref{lem:bound} to reduce to the case where $J\in \mathcal J_x$ is maximal and noting that for such $J$, $W_{x,J} = W_x$.
Condition (iii) then holds as a consequence of condition (ii) (the only possible $F'$ that can arise is $F' = F$).

\subsubsection{$SO(N)$, $N$ odd}
Let $n=\lfloor N/2\rfloor$, $G^\vee = Spin(N)$.
Fix $\OO^\vee\in \cN^\vee_o$. We show that $\OO^\vee$ is $x$-faithful for all $x\in \overline{\mathfrak C}(SO(N))$ such that $\OO^\vee\in (\cN_o^\vee)_x$ by exhibiting a single pair $(J,\phi)$ that works for all such blocks.
Let $\lambda = (\lambda_1\le\lambda_2\le \dots \le \lambda_e)$ be the orthogonal partition of $N$ corresponding to $\OO^\vee$.
Define $\mu(\lambda) \subseteq \lambda^t$ as follows
$$m_{\mu(\lambda)}(x) = \begin{cases} 
2 &\mbox{if } x \text{ is even and } m_{\lambda^t}(x) \text{ is}\\
&\mbox{even and nonzero}\\
1 & \mbox{if } x \text{ is even and } m_{\lambda^t}(x) \text{ is}\\
& \mbox{odd}\\
0 &\mbox{otherwise} \end{cases}$$
and define $\nu(\lambda) = d_{BV}(\lambda)\setminus \mu(\lambda)$.
By \cite[Lemma 4.8.4]{cmo} $\mu(\lambda)$ and $\nu(\lambda)$ are both special symplectic partitions.
We set $J$ to be the set $C_m\times C_{n-m}$ where $m = |\mu|/2$, and $\phi$ to be the family characterised by $\WF(\phi,\CC) = (\mu(\lambda),\nu(\lambda))$.
Since $J = C_m\times C_{n-m}$, the family $\phi$ decomposes as a product $\phi_1\times \phi_2$ where $\phi_1$ is a family of $C_m$, $\phi_2$ is a family of $C_{n-m}$ and
$$\WF(\phi_1) = \mu(\lambda), \quad \WF(\phi_2) = \nu(\lambda).$$

It follows from \cite[Lemma 4.8.6]{cmo} that condition (i) holds for this choice of $(J,\phi)$.
For condition (ii), let $x\in \overline{\mathfrak C}(SO(N))$ and suppose $(\OO^\vee,\rho)\in (\mathcal N_{o,l})_x$.
As described in Section \ref{subsubsec:sospringerodd} (3) let $r$ be the odd integer corresponding to $x$.
Then $\theta(x) = (C_{a^2+a}\times C_{a^2+a},\bf E,\psi,1)$, $a = (r-1)/2$, $\bf E$ is a uniquely determined cuspidal unipotent representation, and $\mathcal W_x$ is of type $\tilde C_{\frac 12(r^2-1)}$. 
By Section \ref{subsubsec:unipC}, $\unip_{\bfL_{c^\omega(K)},\bf E}(\bfL_{c^\omega(J)}(\mathbb F_q)) \leftrightarrow \mathcal Y^1_{r,m}\times\mathcal Y^1_{r,n-m}$.
Note that $K\subseteq J$ (and hence $J\in \mathcal J_x$) if and only if $\mathcal Y^1_{r,m}\times\mathcal Y^1_{r,n-m}$ is non-empty.
Under the bijection in Section \ref{subsubsec:sospringerodd} (1), $(\OO^\vee,\rho)$ corresponds to a $S\in \mathcal Y^2_{r,n}$ with $p_1(S) = \lambda$.
By \cite[Lemma 4.4.2]{cmo}, $S$ has the following description in terms of $S^{sp} = (X,Y)$.
Let $X=(X_1,\dots,X_c), Y = (Y_1,\dots,Y_c)$ be the refinements of $X$ and $Y$ (see paragraph preceeding \cite[Example 4.4.1]{cmo}).
Then $S=(A,B)$ where $A=(A_1,\dots,A_c),B=(B_1,\dots,B_c)$ and $\{A_i,B_i\} = \{X_i,Y_i\}$ for all $i$.
Let $S^{1,sp}= (X^1,Y^1)\in \mathcal Y_{1,m}^1, S^{2,sp}=(X^2,Y^2)\in \mathcal Y_{1,n-m}^1$ be the special symbols in the families $\phi_1$ and $\phi_2$ respectively.
Taking sufficiently large representatives in the relevant shift equivalence classes we have by condition (i) and \cite[Proposition 4.7.1]{cmo} that $S^{sp} = S^{1,sp} + S^{2,sp}$.
Let $S^{i,sp}=(X^i,Y^i)=((X_1^i,\dots,X_c^i),(Y_1^i,\dots,Y_c^i))$ be the decomposition into subsequences along the same indices as the decomposition of $X,Y$.
Define $S^i = ((A_1^i,\dots,A_c^i),(B_1^i,\dots,B_c^i))$ where
\begin{equation}
    (A_j^i,B_j^i) = \begin{cases}
        (X_j^i,Y_j^i) & \mbox{ if } (A_j,B_j)=(X_j,Y_j) \\
        (Y_j^i,X_j^i) & \mbox{ if } (A_j,B_j)=(Y_j,X_j).
    \end{cases}
\end{equation}
By construction $S^{i,sp}\sim S^i$ and $S = S^1+S^2$.
Moreover, by \cite[Lemma 4.5.2]{cmo} and \cite[Lemma 4.9.4]{cmo} we have that $S^1\in \mathcal Y^1_{r,m}, S^2\in \mathcal Y^1_{r,n-m}$.
Thus $\unip_{\bfL_{c^\omega(K)},\bf E}(\bfL_{c^\omega(J)}(\mathbb F_q)) \ne \emptyset$, so $J\in \mathcal J_x$, and
\begin{equation}
    \label{eq:sumSOodd}
    \mathcal D(S) = \mathcal D(S^1) + \mathcal D(S^2).
\end{equation}
Note that $\mathcal D(S)$ parameterises an irreducible representation of $W_x$ and $(\mathcal D(S^1),\mathcal D(S^2))$ parameterises an irreducible representation of $W_{x,J}$.
By the Littlewood-Richardson rule (c.f. the proof of \cite[Proposition 4.7.2]{cmo}) Equation \ref{eq:sumSOodd} implies that
$$\Hom_{W_{x,J}}(F,E(\OO^\vee, \rho)|_{W_{x,J}}) \neq 0.$$

Finally for condition (iii), let $S\in \mathcal Y^2_{r,n}$ of defect $r\ge 1$ correspond to $(\OO^\vee,\rho)$ as above and suppose $J\in \mathcal J_x$ is maximal and $F\in \mathrm{Irr}(W_{x,J'})$ satisfies
 $$\Hom_{W_{x,J}}(F,E(\OO^\vee, \rho)|_{W_{x,J}}) \neq 0.$$
 Since $J$ is maximal it is of the form $C_m\times C_{n-m}$ and $W_{x,J}$ is of the form $W_k\times W_{l-k}$.
 Let the pair of bipartitions $(B^1,B^2)$ correspond to $F$.
 Let $(S^1,S^2)$ be the corresponding symbols of defect $r$.
 By the Littlewood-Richardson rule (cf. \cite[Proposition 4.7.2]{cmo}) we have that $S^1+S^2\ge S$.
We wish to show that 
$$d_S(J,\WF(\mathcal U_{x,J}(F),\CC)) \ge \OO^\vee.$$
By \cite[Lemma 3.2]{acharaubert} it suffices to show that the inequality holds with respect to the dominance order for the respective s-symbols viewed as non-decreasing sequences i.e. partitions.
The s-symbol for $E(\OO^\vee,1)$ is just $S^{sp}$.
Let us compute the s-symbol for the left hand side.
Let $\lambda_1 = \mathcal P(2S_1^{1,sp}\cup 2S_2^{1,sp}+1)$ and $\lambda_2 = \mathcal P(2S_2^{2,sp}\cup 2S_2^{2,sp}+1)$ (these are both special symplectic partitions).
Then 
$$\WF(\mathcal U_{x,J}(F)) = \lambda_1^t\times \lambda_2^t.$$
Thus by \cite[Proposition 2.31]{okada2021wavefront}
$$d_S(J,\WF(\mathcal U_{x,J}(F),\CC)) = j_{W_J}^{W^\vee}(E(\lambda_1 ,1)\otimes E(\lambda_2,1)).$$
where $j$ is $j$-induction, see \cite[Chapter 4]{Lusztig1984}. By \cite[4.5 (a)]{Lusztig2009}, $W_J$ is the parahoric of $W$ (the Weyl group of $\bfG$) determined by $J$, and the s-symbol for the $j$-induced representation is
\begin{equation}
    \label{eq:ssymb}
    S^{1,sp}+S^{2,sp}.
\end{equation}
Since $S^1+S^2\ge S$, by Lemma \ref{lem:uniondominance} we have that $S^{1,sp}+S^{2,sp}\ge S$.
This is the required inequality.

\subsubsection{$SO(N)$, $N$ even}
\label{subsubsec:soeven}
Suppose $\lambda$ is not one of the edge cases listed in \cite[Proposition 4.8.2]{cmo}.
Then the proof is identical to the $N$ odd case except for a few minor technical differences which we note.

\begin{itemize}
    \item The partition $\mu(\lambda)$ is defined as
    $$m_{\mu(\lambda)}(x) = \begin{cases} 
    2 &\mbox{if } x \text{ is odd and } m_{\lambda^t}(x) \text{ is}\\
    &\mbox{even and nonzero}\\
    1 & \mbox{if } x \text{ is odd and } m_{\lambda^t}(x) \text{ is}\\
    & \mbox{odd}\\
    0 &\mbox{otherwise}. \end{cases}$$
    \item $J = D_m\times D_{n-m}$.
    \item $r$ is even. $\omega = 1$ if $4\mid r$, and $\omega$ is type I if $4\mid r-2$. $K=D_{a^2}\times D_{a^2}$ where $a = r/2$.
    Despite these differences, we still have $\unip_{\bfL_{c^\omega(K)},\bf E}(\bfL_{c^\omega(J)}(\mathbb F_q)) \leftrightarrow \mathcal Y^1_{r,m}\times\mathcal Y^1_{r,n-m}$.
    \item The rest of the argument follows through identically, except for the case $r=0$ (see Equation \ref{eq:mathcalD} for why this case is different) which is treated in \cite[Lemma 4.9.5]{cmo}.
\end{itemize}
Now suppose $\lambda$ is one of the edge cases listed in \cite[Porposition 4.8.2]{cmo}.
Recall that for these $\lambda$, $\mu(\lambda) = (1^2)$ which is a partition of $2$. 
The issue with these edge cases can be traced down to the fact that $m = 1$, but $D_1 \times D_{n-1}$ is not a subgraph of $\tilde D_n$ if we interpret $D_1$ as $A_1$.
If we interpret $D_1$ as the empty graph however, then as long as we interpret things correctly, the same argument as the one given for $\lambda$ not an edge case works.
We briefly give the details.

Suppose $\lambda$ is one of said edge cases. 
Then $\lambda$ is special, $r$ takes values in $\{0,2\}$, $\mu(\lambda) = (1^2)$ and $\nu(\lambda) = d_{BV}(\lambda)\setminus \mu(\lambda)$.
For both $r=0,2$ we have that $K = \emptyset$ (since \cite{Lu-unip1} uses the convention $D_1 = \emptyset$).
Thus we can take $J$ to be $D_{n-1}$.
Take $\phi$ to be the family characterised by $\WF(\phi,\CC) = \nu(\lambda)$.
Since $d_{BV}(\lambda) = \nu(\lambda) \cup \mu(\lambda) = \nu(\lambda) \cup (1^2)$, we have that $\nu(\lambda)$ saturates to $d_{BV}(\lambda)$ in $D_n$ and so since $\lambda$ is special we get that
$$\bar {\mathbb L}(D_{n-1},\nu(\lambda)) = \bar {\mathbb L}(D_n,d_{BV}(\lambda)) = d_A(\OO^\vee,1).$$
This establishes property (i).
For property (ii) we can construct as usual the a-symbols $S_1,S_2$ so that Equation \ref{eq:sumSOodd} holds.
Supposing $r\ne 0$ (the $r=0$ case has already been treated in \cite[Proposition 4.8.2]{cmo}), then $W_x = W_{l}$ where $l=\frac12(N-r^2)$, and $W_{x,J} = W_{l-1}$.
Since $W_{l-1}\le W_1\times W_{l-1}\le W_l$ and $\mathcal D(S^2) \subseteq \mathcal D(S^1)\otimes \mathcal D(S^2)|_{W_{l-1}}$, we get from Equation \ref{eq:sumSOodd} that 
$$\Hom_{W_{x,J}}(F,E(\OO^\vee, \rho)|_{W_{x,J}}) \neq 0$$
as required.
The proof for condition (iii) requires no modifications.

\subsubsection{$Sp(2n)$}
Suppose $\lambda$ is not one of the edge cases listed in \cite[Proposition 4.8.2]{cmo}.
The proof for this case is similar to the $SO(N)$, $N$ odd case with a few (more substantial) technical differences which we note.

\begin{itemize}
    \item $\lambda$ is a symplectic partition of $2n$. The partition $\mu(\lambda)$ is defined as
    $$m_{\mu(\lambda)}(x) = \begin{cases} 
    2 &\mbox{if } x \text{ is odd and } m_{\lambda^t}(x) \text{ is}\\
    &\mbox{even and nonzero}\\
    1 & \mbox{if } x \text{ is odd and } m_{\lambda^t}(x) \text{ is}\\
    & \mbox{odd}\\
    0 &\mbox{otherwise}. \end{cases}$$
    \item $J = D_m\times B_{n-m}$.
    \item $r$ has no parity restriction. $K=D_{a^2}\times B_{a^2+a}$ if $r$ is even and $K = D_{a^2}\times B_{a^2-a}$ if $r$ is odd where $a = \lfloor (r+1)/2\rfloor$.
    $\omega = 1$ if $a$ is even, and $\omega = -1$ if $a$ is odd.
    \item $S\in \mathcal X^{1,1}_{d,n}$ where $d=r+1$ if $r$ is even and $d=-r$ if $r$ is odd.
    \item We have $\unip_{\bfL_{c^\omega(K)},\bf E}(\bfL_{c^\omega(J)}(\mathbb F_q)) \leftrightarrow \mathcal Y^1_{d-1,m}\times\mathcal Y^1_{d,n-m}$.
    \item $S^{1,sp}$ lies in $\mathcal Y^1_{0,m}$.
    The additive identity is $S^{sp} = (S^{1,sp})^! + S^{2,sp}$.
    \item $(S^1)^!$ is constructed from $(S^{1,sp})^!$. We have $S^1\in \mathcal X_{d-1,m}^{1,0}, (S^1)^! \in \mathcal X_{d,m}^{0,1}, S^2\in \mathcal X_{d,n-m}^{1,0}$ and $S = (S^1)^!+S^2$. The $\bullet^!$ map reflects similarity classes and so $S^1 \sim S^{1,sp}$.
    \item The rest of the argument follows through identically.
\end{itemize}

If $\lambda$ is one of the edge cases listed in \cite[Proposition 4.8.2]{cmo} then if we interpret $D_1$ as the empty graph, the same argument as in Section \ref{subsubsec:soeven} follows through.

\subsubsection{$Spin(N)$, $N$ odd}
Let $n=\lfloor N/2\rfloor$, $G^\vee = Spin(N)$.
Fix $(\OO^\vee,\rho)\in \cN_{o,l}^*$ and let $\lambda = (\lambda_1^{a_1},\lambda_2^{a_2},\dots,\lambda_k^{a_k})\in \cN_{ro}$ for $\lambda_1<\cdots<\lambda_k$ and $a_i\ge 1, i=1,\dots,k$ be the corresponding partition under the bijection in Section \ref{subsubsec:spinspringerodd} (1).
Define $\mu(\lambda) \subseteq \lambda^t$ as follows
$$m_{\mu(\lambda)}(x) = \begin{cases} 
2 &\mbox{if } x \text{ is even and } 4\mid m_{\lambda^t}(x) \\
&\mbox{and } m_{\lambda^t}(x) \ne 0\\
1 & \mbox{if } x \text{ is even and } m_{\lambda^t}(x) \text{ is}\\
& \mbox{odd}\\
0 &\mbox{otherwise.} \end{cases}$$
Recall from Corollary \ref{cor:dlambda} that we can write $d_{BV}(\lambda) = \underline{d_{BV}(\lambda)}\cup \underline{d_{BV}(\lambda)}$.
\begin{lemma}
    $\mu(\lambda)\subseteq \underline{d_{BV}(\lambda)}$.
\end{lemma}

\begin{proof}
Arguing as in the proof of Lemma \ref{lem:dualro}, we can reduce to the case when $\lambda$ consists of only odd parts, i.e. $\lambda=(\lambda_1<\lambda_2<...<\lambda_k)$. Recall that
$$\lambda^t = (k^{\overline{\lambda}_1},(k-1)^{\overline{\lambda}_2},...,1^{\overline{\lambda}_k})$$
and
$$d_{BV}(\lambda) = (k^{\overline{\lambda}_1+\delta_1},(k-1)^{\overline{\lambda}_2+\delta_2},...,1^{\overline{\lambda}_k}+\delta_k)$$
where $\delta=(-1,2,-2,...,2,-2)$. If $x$ is even, then $m_{\lambda^t}(x) \geq 2$. So $m_{d_{BV}(\lambda)}(x) = m_{\lambda^t}(x)+2 \geq 4$. Hence, $m_{\underline{d_{BV}(\lambda)}}(x) = \frac{1}{2}m_{d_{BV}(\lambda)}(x) \geq 2$.
\end{proof}

Define $\nu(\lambda) = \underline{d_{BV}(\lambda)}\ \setminus \mu(\lambda)$.

Let $x\in \overline{\mathfrak C}^{*}(Spin(N))$ be such that $(\OO,\rho)\in (\cN_{o,l}^*)_x$.
$x$ is uniquely determined and corresponds to the block parameterised by $d = d(\lambda)$ in the description of blocks given in \ref{subsubsec:spinspringerodd} (3).
We have $\theta(x) = (C_{(d^2-1)/4}\times A_{d(d-1)/2-1}\times C_{(d^2-1)/4},\bf E,\psi,\omega)$ where $\omega = -1$, and the affine Weyl group is of type $\tilde C$.
Note that since $N$ is odd and $d(\lambda)\equiv N\pmod 2$ we have that $d$ is also odd.
We set $J$ to be the $\omega$-stable set $C_m\times A_{n-2m-1}\times C_m$ where $m = |\mu|/2$.
The reductive quotient has type $C_m(q)\times \hphantom{ }^2A_{n-2m-1}(q^2)$. 
By Section \ref{subsec:unipotent}, 
\begin{align*}
    %\unip_{\bfL_{c^\omega(K)},\bf E}(\bfL_{c^\omega(J)}(\mathbb F_q)) &\leftrightarrow \mathcal Y^1_{d,m}\times\mathrm{Irr}(W_{\frac12(n-2m-\frac12r(r+1))})
    \mathrm{Irr}(W_{m-(d^2-1)/4})\times\mathrm{Irr}(W_{\frac12(n-2m-\frac12r(r+1))}) &\leftrightarrow \unip_{\bfL_{c^\omega(K)},\bf E}(\bfL_{c^\omega(J)}(\mathbb F_q))
\end{align*}
where $r = d-1$ if $d>0$ and $r=-d$ if $d\le 0$.
It will be useful to note that 
$$\unip_{\bfL_{c^\omega(K)},\bf E}(\bfL_{c^\omega(J)}(\mathbb F_q)) \subseteq \unip(C_m(q))\times \unip(\hphantom{ }^2A_{n-2m-1}).$$
and that $\mathcal U_{x,J}$ is a product of maps 
$$\mathrm{Irr}(W_{m-(d^2-1)/4})\to \unip(C_m(q)), \quad \mathrm{Irr}(W_{\frac12(n-2m-\frac12r(r+1))})\to \unip(\hphantom{ }^2A_{n-2m-1}).$$
We call these maps $\mathcal U_{x,J,1}$ and $\mathcal U_{x,J,2}$ respectively.

We set $\phi$ to be the family characterised by $\WF(\phi,\CC) = (\mu(\lambda),\nu(\lambda))$.
Note the family $\phi$ decomposes as a product $\phi_1\times \phi_2$ where $\phi_1$ is a family of $C_m(q)$, $\phi_2$ is a family of $^2A_{n-2m-1}(q^2)$ and
$$\WF(\phi_1) = \mu(\lambda), \quad \WF(\phi_2) = \nu(\lambda).$$
The following Lemma establishes condition (i) for the pair $(J,\phi)$.
\begin{lemma}
    $\bar{\mathbb L}(J,\WF(\phi,\CC)) = d_A(\OO^\vee,1)$.
\end{lemma}
\begin{proof}
    We have 
    \begin{align*}
        \bar{\mathbb L}(J,\WF(\phi,\CC)) &= \bar{\mathbb L}(C_m\times A_{n-2m-1}\times C_m,(\mu(\lambda),\nu(\lambda),\mu(\lambda))) \\
        &=\bar{\mathbb L}(C_m\times C_{n-m},(\mu(\lambda),\nu(\lambda)\cup\nu(\lambda)\cup\mu(\lambda))) \\
        &=\bar{\mathbb L}(C_m\times C_{n-m},(\mu(\lambda),d_{BV}(\lambda)\setminus \mu(\lambda)) \\
        &=d_A(\OO^\vee,1)
    \end{align*}
    where the last equality follows from essentially the same argument given in \cite[Lemma 4.8.6]{cmo}.
\end{proof}
Recall the definition of $\rho$ from Section \ref{subsec:spincombo} and let $B = (B_1,B_2) = \rho(\lambda)$.
Note that there is a case distinction involving a flip depending on if $d(\lambda)>0$ or $\le 0$.
Let us assume for now that $d(\lambda)>0$. 
We will discuss the case $d(\lambda)\le0$ separately at the end.

To verify condition (ii) it suffices to find $B^1\in \mathcal B(m-(d^2-1)/4), B^2\in \mathcal B(\frac12(n-2m-\frac12r(r+1)))$ such that
$$B = B^1+B^2$$
(c.f. the proof of \cite[Proposition 4.7.2]{cmo}) and $\mathcal U_{x,J,1}(B^1)\in \phi_1,\mathcal U_{x,J,2}(B^2) \in \phi_2$.
Note that since the families in $^2A_{n-2m-1}$ are all singletons, the condition $\WF(\mathcal U_{x,J,2}(B^2)) = \nu(\lambda)$ in fact completely determines $B^2$ provided such a $B^2$ exists.
For this reason it is covenient to determine $B^2$ first.

Recall the definitions of $f, \delta, \bar\lambda_i,A_i$ from the end of Section \ref{subsec:spincombo}.
Let $(a,b]$ denote the set $\{a+1,\dots,b\}$ when $a<b$ and $\emptyset$ otherwise. Let $(a,b)$ denote the set $\{a+1,\dots,b-1\}$ if $a+1<b$ and $\emptyset$ otherwise.
Let $[P]$ denote the function that takes value $1$ if $P$ is true and $0$ if $P$ is false.
Define $\xi_i = [4\mid \bar\lambda_i], \zeta_i = [2\nmid \bar\lambda_i]$ and $\mathcal S = \bigcup_{i=1}^l(f(2i-1),f(2i)]$.
Since $a_i=1$ whenever $\lambda_i$ is odd, we have that $A_i$ is even if and only if $i\in \mathcal S$.
Therefore, we can write $\mu$ as $(A_i^{\zeta_i+2\xi_i})_{i\in \mathcal S}$ (this is a decreasing sequence in $i$).
It follows that 
$$\nu = (A_i^{\frac12(\bar\lambda_i+\delta_i)-[i\in \mathcal S](\zeta_i+2\xi_i)}).$$
A direct computation then shows that
$$\nu^t = (\eta_1^{a_1},\eta_2^{a_2},\dots,\eta_k^{a_k})$$
where $\eta_1\le\eta_2\le\cdots\le \eta_k$, 
$$\eta_i = \frac12(\lambda_i+[i\in \mathcal O_0]-[i\in \mathcal O_1])-E_i, \quad E_i = \sum_{j\le i, j\in \mathcal S}\zeta_i+2\xi_i$$
and 
$$\mathcal O_i = \{f(j) \mid 1\le j\le 2l,\ j \equiv i\pmod 2\}.$$
By Section \ref{subsubsec:unip2A} (2), the partition $\nu^t$ parameterises $\mathcal U_{x,J,2}(B^2)$ as en element of $\unip(\hphantom{ }^2A_{n-2m-1}(q^2))$. 
Let us now compute $B_2 = (B_1^2,B_2^2)$ using the algorithm in Section \ref{subsubsec:unip2A} (6).
Let $z_i = \lfloor \lambda_i/2 \rfloor$ so $z_i$ is even if $r_i = 0,1$ and is odd if $r_i = 2,3$.
In terms of $z_i$ we have
$$\eta_i = z_i + [i\in \mathcal O_0] - E_i.$$
The $\beta$-set $Z$ for $\nu^t$ of length $\#\lambda$ consists of blocks of sequences of the form
$$z_i+[i\in \mathcal O_0]-E_i+C_i+(0,1,\dots,a_i-1)$$
where $C_i = \sum_{j<i}a_j$.
Note that 
\begin{align*}
    &z_i + [i\in \mathcal O_0]-E_i+C_i \\
    \equiv &[r_i=2,3] + [i\in \mathcal S, \lambda_i \text{ is odd}] + [i\in \mathcal S, \lambda_i \text{ is even}] + [i\in \mathcal S] \pmod 2 \\
    \equiv &[r_i = 2,3] \pmod 2.
\end{align*}
Thus the $\beta$-set for $B_1^2$, which we denote $S_1^2$, consists of blocks of sequences of the form 
\begin{equation}
    \begin{cases}
        q_i+\frac12 \Delta_i+(0,1,\dots,a_i/2-1) & \mbox{ if $r_i=0$} \\
        q_i+\frac12\Delta_i+(0) & \mbox{ if $r_i=1$} \\
        q_i+\frac12 \Delta_i+ 1 +(0,1,\dots,a_i/2-1) & \mbox{ if $r_i=2$} \\
        \emptyset & \mbox{ if $r_i=3$} 
    \end{cases}
\end{equation}
and the $\beta$-set for $B_2^2$, which we denote $S_2^2$, consists of blocks of sequences of the form 
\begin{equation}
    \begin{cases}
        q_i+\frac12\Delta_i+(0,1,\dots,a_i/2-1) & \mbox{ if $r_i=0$} \\
        \emptyset & \mbox{ if $r_i=1$} \\
        q_i+\frac12\Delta_i+(0,1,\dots,a_i/2-1) & \mbox{ if $r_i=2$} \\
        q_i+\frac12\Delta_i+(0) & \mbox{ if $r_i=3$.}
    \end{cases}
\end{equation}
where $\Delta_i = [i\in \mathcal O_0]+C_i-E_i$ and $1\le i\le k$.
The partitions $B_1^2,B_2^2$ can then be obtained by applying $\mathcal P$ to the $\beta$-sets.
Note that $|S_1^2|-|S_2^2|=d$.

Let us now produce a bipartition $B_1$ satisfying $\mathcal U_{x,J,1}(B_1) \in \phi_1$ which, we will see, satisfies $B= B^1+B^2$.
First suppose such a $B^1$ exists.
Let $S = \mathcal D^{-1}(B^1)\in \mathcal Y_{d,m}^1$.
Using Section \ref{subsubsec:unipC} (3), we can compute $S^{sp}$ from the condition $\WF(\mathcal U_{x,J,1}(B^1)) = \mu(\lambda)$.
A direct computation shows that
$$\mu^t = (E_1^{a_1},E_2^{a_2},\dots,E_k^{a_k}).$$
Taking a $\beta$-set of length $\#\lambda$ one computes that $S^{sp} = (X,Y)$ where $X$ consists of blocks of sequences of the form
\begin{equation}
    \begin{cases}
        \emptyset & \mbox{ if $i\in \mathcal O_0$} \\
        (C_i+E_i)/2 + (0) & \mbox{ if $i\in \mathcal O_1$} \\
        (C_i+E_i)/2 + (0,1,\dots,a_i/2-1) & \mbox{ if $\lambda_i$ is even}
    \end{cases}
\end{equation}
and $Y$ consists of blocks of sequences of the form
\begin{equation}
    \begin{cases}
        (C_i+E_i-1)/2 + (0) & \mbox{ if $i\in \mathcal O_0$} \\
        \emptyset & \mbox{ if $i\in \mathcal O_1$} \\
        (C_i+E_i)/2 + (0,1,\dots,a_i/2-1) & \mbox{ if $\lambda_i$ is even.}
    \end{cases}
\end{equation}
Thus we know that if such a $S$ exists, it must have entries coming from the above special symbol.
We now define such an $S$. 
Let $S=(S_1^1,S_2^1)$ where $S_1^1$ consists of blocks of sequences of the form
\begin{equation}
    \begin{cases}
        (C_i+E_i-[i\in \mathcal O_0])/2 + (0) & \mbox{ if $r_i=1$} \\
        \emptyset & \mbox{ if $r_i=3$} \\
        (C_i+E_i)/2 + (0,1,\dots,a_i/2-1) & \mbox{ if $r_i=0,2$}
    \end{cases}
\end{equation}
and $S_2^1$ consists of blocks of sequences of the form
\begin{equation}
    \begin{cases}
        \emptyset & \mbox{ if $r_i=1$} \\
        (C_i+E_i-[i\in \mathcal O_0])/2 + (0) & \mbox{ if $r_i=3$} \\
        (C_i+E_i)/2 + (0,1,\dots,a_i/2-1) & \mbox{ if $r_i=0,2$.}
    \end{cases}
\end{equation}
Note as a sanity check that $C_i+E_i \equiv [i\in \mathcal S] + [i\in \mathcal S,\lambda_i \text{ even}] \equiv [i\in \mathcal O_0] \pmod 2$ since $\mathcal O_0 = \{i\in \mathcal S\mid \text{$\lambda_i$ is odd}\}$ and so all the entries are integers.
It is clear that $S$ has the same entries as $S^{sp}$ and has defect $d$.
To show that $S$ lies in $\mathcal Y_{d,m}^1$ it suffices to check that $X',Y'$ are both strictly increasing.
Clearly within each block the sequence is strictly increasing so we only need to compare endpoints of blocks.
We prove this for $S_1^1$. The proof for $S_2^1$ is analogous.
Let $o_i = \lfloor (C_i+E_i)/2\rfloor = (C_i+E_i-[i\in\mathcal O_0])/2$, $e_i=(C_i+E_i)/2$ and $e_i' = (C_{i+1}+E_i)/2-1$.
Here $e_i$ and $e_i'$ are the entries at the start and at the end of the $i$th block when $\lambda_i$ is even.
Let $i$ and $j$ index consecutive blocks of $S_1^1$ (so $i<j$ and $r_l=3$ for all $i<l<j$).
There are a number of cases to consider: 
\begin{itemize}
    \item $r_{i}=1,r_j=1$ $\implies$ $o_{i}<o_j$,
    \item $r_{i}=1,r_j\in\{0,2\}$ $\implies$ $o_{i}<e_j$,
    \item $r_i\in\{0,2\},r_j=1$ $\implies$ $e_{i}'<o_j$,
    \item $r_i\in\{0,2\},r_j\in\{0,2\}$ $\implies$ $e_{i}'<e_j$.
\end{itemize}
We prove each case as follows:
\begin{itemize}
    \item Note that $C_j = C_i + (j-i)$ and so $o_i<o_j$ certainly holds if $j-i\ge 2$.
    Thus it remains to verify the case when $j=i+1$.
    Suppose $j=i+1$. Then $C_j=C_i+1$.
    Moreover, since $r_{j-1}=r_i=r_j$ we have $\xi_j = 1$ and so $E_j=E_i+2[j\in \mathcal S]$.
    Thus 
    $$o_j = o_i + (1+2[j\in \mathcal S]+[j-1\in \mathcal O_0]-[j\in \mathcal O_0])/2 = o_i+1>o_i$$
    as required.
    \item By a similar reduction to the one performed in the previous case we may assume that $j=i+1$.
    Then $C_j = C_i+1$ and
    $$E_j = E_i + [j\in \mathcal S] = E_i + 1-[j-1\in \mathcal O_0].$$
    Thus 
    $$e_j = (C_j+E_j)/2 = (C_i+E_i-[i\in \mathcal O_0])/2+1>o_i$$
    as required.
    \item Suppose $j-i\ge 2$. Then $C_j = C_{i+1}+(j-i-1)\ge C_{i+1}+1$ and $E_j \ge E_i$ so
    $$o_j = (C_j+E_j-[j\in \mathcal O_0])/2 \ge (C_{i+1}+E_i+1-[j\in \mathcal O_0])/2 > e_i'.$$
    If $j=i+1$ then $E_j = E_i+[j\in \mathcal O_0]$ so $o_j = e_i'+1 > e_i'$ as required.
    
    \item Since $j>i$ we have $C_j\ge C_{i+1}$ and $E_j\ge E_i$ and so
    $$e_i' \le (C_j+E_j)/2-1 < e_j.$$
\end{itemize}

It thus now remains to check that $\mathcal P(S_y^1) + \mathcal P(S_y^2) = B_y$ for $y=1,2$.
For this we prove the equivalent statement that 
\begin{equation}
   \label{eq:sum} 
    S_y^1+S_y^2 = B_y + 2\cdot(0,1,2,\dots).
\end{equation}
Let $\lambda^i$ denote the partition $(\lambda_1^{a_1},\lambda_2^{a_2},\dots,\lambda_{i-1}^{a_{i-1}})$.
We note that by Proposition \ref{prop:spinnonrecursive}, $B_1$ consists of blocks of the form
\begin{equation}
    \begin{cases}
        q_i-d(\lambda^i) + (0,0,\dots,0) & \mbox{ if $r_i=0$} \\
        q_i-d(\lambda^i) + (0) & \mbox{ if $r_i=1$} \\
        q_i-d(\lambda^i)+1+ (0,0,\dots,0) & \mbox{ if $r_i=2$} \\
        \emptyset & \mbox{ if $r_i=3$} 
    \end{cases}
\end{equation}
where in the $r_i = 0,2$ cases the blocks have length $a_i/2$, and $B_2$ consists of blocks of sequences of the form
\begin{equation}
    \begin{cases}
        q_i+d(\lambda^i)+(0,0,\dots,0) & \mbox{ if $r_i=0$} \\
        \emptyset & \mbox{ if $r_i=1$} \\
        q_i+d(\lambda^i) + (0,0,\dots,0) & \mbox{ if $r_i=2$} \\
        q_i+d(\lambda^i) + (0) & \mbox{ if $r_i=3$} 
    \end{cases}
\end{equation}
where in the $r_i = 0,2$ cases the blocks have length $a_i/2$.
We prove Equation \ref{eq:sum} for $y=1$. The case $y=2$ is analogous.
Note that the block structures of $S_i^1,S_i^2$, and $B_i$ all coincide (by which we mean they have the same number of blocks and the blocks have the same size) and so we prove the equality on an arbitrary block.
For this it will be necessary to decompose the sequence $(0,1,\dots)$ into blocks of the same size.
We have for $y=1$, the sequence $(0,1,\dots)$ decomposes into the blocks
\begin{equation}
    \begin{cases}
        F_i + (0,1,\dots,a_i/2-1) & \mbox{ if $r_i=0$} \\
        F_i + (0) & \mbox{ if $r_i=1$} \\
        F_i + (0,1,\dots,a_i/2-1) & \mbox{ if $r_i=2$} \\
        \emptyset & \mbox{ if $r_i=3$} 
    \end{cases}
\end{equation}
where $F_i = \sum_{j<i}([2\mid \lambda_j]a_j/2+[r(\lambda_j) = 1])$.
Unpacking Equation \ref{eq:sum} for a block with $r_i = 0,2$ we see that we must show that
\begin{align*}
    &q_i+\frac12\Delta_i+[r_i=2]+(0,1,\dots,a_i/2-1)+(C_i+E_i)/2+(0,1,\dots,a_i/2-1) \\
    = &q_i-d(\lambda^i)+[r_i=2]+2F_i+2(0,1,\dots,a_i/2-1)
\end{align*}
After cancelling equal terms we get that we must show
\begin{equation}
    \label{eq:red}
    \frac12\Delta_i + (C_i+E_i)/2 = 2F_i-d(\lambda^i).
\end{equation}
Now the left hand side simplifies to $C_i$, whereas the right hand side simplifies to
\begin{align*}
    &\sum_{j<i}\left([2\mid a_j]a_j+2[r(\lambda_j)=1]\right)-\sum_{j<i}\left([r(\lambda_j)=1]-[r(\lambda_j)=3]\right) \\
    = &\sum_{j<i} \left([2\mid a_j]a_j + [r(\lambda_j)=1] + [r(\lambda_j)=3] \right) = C_i.
\end{align*}
Thus we have the equality that we desire. For blocks with $r_i = 1$ we must show that 
\begin{align*}
    &q_i+\frac12\Delta_i+(0) + (C_i+E_i-[i\in \mathcal O_0])/2 + (0) \\
    = &q_i-d(\lambda^i)+2F_i+2(0).
\end{align*}
Cancelling equal terms yields exactly Equation \ref{eq:red} which we have just established.
Finally, the $r_i = 3$ case is vacuous.
This concludes our proof that condition (ii) holds when $d(\lambda) >0$.
The case $d(\lambda)\le 0$ is analogous after making the following modifications
\begin{itemize}
    \item In terms of $B_1,B_2$ as defined above, $B = (B_2,B_1)$;
    \item In terms of $B_1^2,B_2^2$ as defined above, $B_2 = (B_2^2,B_1^2)$;
    \item In terms of $S_1^1,S_2^1$ as defined above, we define $S = (S_2^1,S_1^1)$.
\end{itemize}
Since all the terms are just flipped the proof follows through identically.

Finally we show that condition (iii) holds.
We first need a preparatory lemma.
For $r\in \{0,1,2,3\}$ and $a\ge 0$, define $\gamma(r,a)$ to be the sequence of length $a$ of the form
\begin{equation}
    \begin{cases}
        (0,1,0,1,\dots) & \mbox{if } r=0,2 \\
        (0,0,0,0,\dots) & \mbox{if } r=1,3.
    \end{cases}
\end{equation}
Define
\begin{align*}
    \gamma(\lambda) &= (\gamma(r_1,a_1),\gamma(r_2,a_2),\dots,\gamma(r_k,a_k)).
\end{align*}
\begin{lemma}
    \label{lem:trivlocsys}
    Let $\lambda\in \cN_{ro}$.
    Let $\Lambda\in \mathcal Y^2_{\bullet,n}$ be the symbol with $\mathcal D(\Lambda) = E(\lambda,1)$.
    Then 
    $$Z^*(\Lambda) =  2q(\lambda)+z+\epsilon(\lambda)-2\delta(\lambda).$$
\end{lemma}
\begin{proof}
    Write $\lambda = (\lambda_1^{a_1},\dots,\lambda_k^{a_k})$ with $\lambda_1<\lambda_2<\dots<\lambda_k$.
    Let $1\le n_1<\cdots<n_l\le k$ denote the indicies for which $\lambda_i$ is odd. Let $n_0 = 0$, $\lambda_0=0$, and $a_0 = 0$.
    We can break $\lambda$ up into subsequences of the form $(\lambda_{n_i}^{a_{n_i}},\dots, \lambda_{n_{i+1}-1}^{a_{n_{i+1}-1}})$.
    Note that for $i>0$, $a_{n_i} = 1$.
    Let $A_i = \sum_{j<i}a_j$.
    Then $\lambda + z$ is block-wise of the form 
    \begin{equation}
        \label{eq:lplusz}
        A_{n_i} + (\lambda_{n_i}^{a_{n_i}},\dots, \lambda_{n_{i+1}-1}^{a_{n_{i+1}-1}}) + (0,1,\dots).
    \end{equation}
    Since $A_{n_i} \equiv i-1 \pmod 2$ for $i>0$ and $A_{n_0} = 0$, considering this sequence in Equation \ref{eq:lplusz} modulo $2$ we see that it is of the form $i+(0,0,1,0,1,\dots) \pmod 2$ for $i>0$ and $(0,1,0,1,\dots) \pmod 2$ for $i=0$.
    Let 
    Thus the a-symbol consists of blocks of the form 
    $$\begin{pmatrix}
        \lambda_{n_i}/2,& \dots, & \lambda_{n_{i+1}-1}/2, & \dots \\
        \lambda_{n_i}/2,& \dots, & \lambda_{n_{i+1}-1}/2, & \dots 
    \end{pmatrix}+\begin{pmatrix}
        0, & 1, & \dots \\
        0, & 1, & \dots 
    \end{pmatrix}$$
    for $i=0$,
    $$A_{n_i}/2 + \begin{pmatrix}
        \lfloor\lambda_{n_i}/2\rfloor,&\lambda_{n_i+1}/2,& \dots, & \lambda_{n_{i+1}-1}/2, & \dots \\
        &\lambda_{n_i+1}/2,& \dots, & \lambda_{n_{i+1}-1}/2, & \dots
    \end{pmatrix}+\begin{pmatrix}
        0, & 0, & 1,& \dots \\
         & 1, & 2,& \dots 
    \end{pmatrix}$$
    for $i>0$ and odd, and
    $$(A_{n_i}-1)/2 + \begin{pmatrix}
        &\lambda_{n_i+1}/2,& \dots, & \lambda_{n_{i+1}-1}/2 & \dots \\
        \lfloor\lambda_{n_i}/2\rfloor,&\lambda_{n_i+1}/2,& \dots, & \lambda_{n_{i+1}-1}/2, & \dots
    \end{pmatrix}+\begin{pmatrix}
          & 1, & 2,& \dots \\
        1,& 1, & 2,& \dots 
    \end{pmatrix}$$
    for $i>0$ and even.
    The s-symbol thus consists of blocks of the form
    $$\begin{pmatrix}
        \lambda_{n_i}/2,& \dots, & \lambda_{n_{i+1}-1}/2, & \dots \\
        \lambda_{n_i}/2,& \dots, & \lambda_{n_{i+1}-1}/2, & \dots
    \end{pmatrix}+\begin{pmatrix}
        0, & 2, & \dots \\
        0, & 2, & \dots 
    \end{pmatrix}$$
    for $i=0$,
    $$A_{n_i} + \begin{pmatrix}
        \lfloor\lambda_{n_i}/2\rfloor,&\lambda_{n_i+1}/2,& \dots, & \lambda_{n_{i+1}-1}/2, & \dots \\
        &\lambda_{n_i+1}/2,& \dots, & \lambda_{n_{i+1}-1}/2, & \dots
    \end{pmatrix}+\begin{pmatrix}
        0, & 1, & 3,& 5,& \dots \\
         & 1, & 3,& 5,& \dots 
    \end{pmatrix}$$
    for $i>0$ and odd, and
    $$A_{n_i} + \begin{pmatrix}
        &\lambda_{n_i+1}/2,& \dots, & \lambda_{n_{i+1}-1}/2 & \dots \\
        \lfloor\lambda_{n_i}/2\rfloor,&\lambda_{n_i+1}/2,& \dots, & \lambda_{n_{i+1}-1}/2, & \dots
    \end{pmatrix}+\begin{pmatrix}
          & 1, & 3,& 5,& \dots \\
        0,& 1, & 3,& 5,& \dots 
    \end{pmatrix}$$
    for $i>0$ and even.
    Thus as a non-decreasing sequence, the s-symbol is equal to 
    $$\lfloor\lambda/2\rfloor + z - \gamma(\lambda).$$
    But $\lfloor \lambda/2\rfloor = 2\lfloor \lambda/4\rfloor + [r=2,3]$.
    Since $[r=2,3] - \gamma(\lambda) = \epsilon(\lambda)-2\delta(\lambda)$ we get that
    $$\lfloor\lambda/2\rfloor + z - \gamma(\lambda) = 2q(\lambda)+z+\epsilon(\lambda)-2\delta(\lambda).$$

\end{proof}

We handle the case $d(\lambda)>0$ first.
Suppose $J\in \mathcal J_x$ is maximal and $F\in \mathrm{Irr}(W_{x,J})$ is such that
\begin{equation}
    \label{eq:res}
    \Hom_{W_{x,J}}(F, E(\OO^\vee,\rho)|_{W_{x,J}})\ne0.
\end{equation}
Since $J$ is maximal $W_{x,J} = W_{i}\times W_{l-i}$ for $l = (N-d(2d-1))/4$ and some $0\le i \le l$.
Let $(B^1,B^2)$ be the pair of bipartitions parameterising $F$ (so $B^1$ is a bipartition of $i$, and $B^2$ is a bipartition of $l-i$).
We wish to show that 
$$d_S(J,\WF(\mathcal U_{x,J}(F),\CC)) \ge \OO^\vee.$$
By \cite[Lemma 3.2]{acharaubert} it suffices to show that the inequality holds for the respective s-symbols viewed as non-decreasing sequences.
By Lemma \ref{lem:trivlocsys} the multiset of entries of the s-symbol for $E(\OO^\vee,1)$ viewed as a non-decreasing sequence is 
$$2q(\lambda) + z + \epsilon(\lambda) -2\delta(\lambda).$$
Let us compute the s-symbol attached to $d_S(J,\WF(\mathcal U_{x,J}(F),\CC))$.
Let $S^i \in \mathcal X_{d(\lambda)}^{1,0}$ be the symbol associated to $B^i$.
Write $S^i = (S_1^i,S_2^i)$.
Let $\lambda_1 = \mathcal P(2S_1^{1,sp}\cup 2S_2^{1,sp}+1)$ (a special symplectic partition) and $\lambda_2 = \mathcal P(2S_1^2\cup 2S_2^2+1)$.
Then 
$$\WF(\mathcal U_{x,J,1})(B^1) = \lambda_1^t, \quad \WF(\mathcal U_{x,J,2}(B^2)) = \lambda_2^t.$$
Thus by \cite[Proposition 2.31]{okada2021wavefront}
$$d_S(J,\WF(\mathcal U_{x,J}(F),\CC)) = j_{W_J}^{W^\vee}(E(\lambda_1,1)\otimes E(\lambda_2,1)\otimes E(\lambda_1,1))$$
where $W_J$ is the parahoric of $W$ (the Weyl group of $\bfG$) determined by $J$.
By \cite[4.5 (a)]{Lusztig2009}, $Z^*$ of the s-symbol for the $j$-induced representation is
\begin{equation}
    \label{eq:ssymb}
    2S^{1,sp}+\widetilde{S^{2}}-z.
\end{equation}
Thus we must show that
\begin{equation}
    \label{eq:ineq}
    2S^{1,sp}+\widetilde{S^{2}}-z\ge 2q(\lambda) + z + \epsilon(\lambda) -2\delta(\lambda).
\end{equation}
For this we first observe that by the Littlewood-Richardson rule (c.f. \cite[Proposition 4.7.2]{cmo}), Equation \ref{eq:res} implies that $B^1+B^2 \ge B$ where $B = \rho(\lambda)$.
Let $q=q(\lambda),\epsilon=\epsilon(\lambda),\delta=\delta(\lambda)$ and $\Lambda$ be the symbol with linear presentation $(q+z-\delta,\epsilon)$.
By Proposition \ref{prop:ssymbspin}, $\Lambda = \rho(B) + 2(z,z)$.
Thus 
$$B^1+B^2 \ge B \Rightarrow S^1+S^2 \ge \Lambda.$$
Since $\epsilon$ has exactly $d(\lambda)$ many $1$'s, there is a unique sequence $S^{i,r}$ such that $(S^{i,r},\epsilon)$ is a linear presentation for $S^i$.
Write $\Lambda^r$ for $q+z-\delta$.
We thus get that
$$S^{1,r}+S^{2,r} - \Lambda^r \ge 0.$$
By Lemma \ref{lem:linearpresentation} we have that 
$$S^{2,sp}+\epsilon^*(S^2)\ge S^{2,r}+\epsilon.$$
In particular
$$S^{2,sp}+\epsilon^*(S^2) - S^{2,r} \ge \epsilon.$$
Finally we have by Lemma \ref{lem:unwind} that 
$$2S^{1,sp}+\widetilde{S^2}-z = 2(S^{1,sp}+S^{2,sp})+\epsilon^*(S^2)-z$$
Therefore
\begin{align*}
    &2S^{1,sp}+\widetilde{S^2}-z \\
    = &2(S^{1,sp}+S^{2,sp})+\epsilon^*(S^2)-z \\
    = &2(S^{1,sp}+S^{2,sp}-S^{1,r}-S^{2,r}+S^{1,r}+S^{2,r}-\Lambda^r+\Lambda^r)+\epsilon^*(S^2)-z \\ 
    = &2(S^{1,sp}-S^{1,r}) + (S^{2,sp}-S^{2,r}) + (S^{2,sp}+\epsilon^*(S^2)-S^{2,r}) + 2(S^{1,r}+S^{2,r}-\Lambda) + 2q + z - 2\delta \\
    \ge &2\cdot 0 + 0 + \epsilon + 2\cdot 0 + 2q + z - 2\delta \\
    = &2q+z + \epsilon - 2\delta.
\end{align*}
where we use Lemma \ref{lem:uniondominance} to deduce that $S^{i,sp}-S^{i,r}\ge0$.
This proves Equation \ref{eq:ineq} as required.
For the case $d(\lambda)\le 0$ we note that $\Lambda$ and $S^2$ get replaced by their flips, $\Lambda^f$ and $(S^2)^f$. 
The linear presentation for $\Lambda^f$ is just $(q+z-\delta,1-\epsilon)$.
Also $\widetilde{(S^2)^f} = 2S^{2,sp}+1-\epsilon^*(S^2)$.
In particular
$$S^{2,sp} + 1-\epsilon^*(S^2)-S^{2,r} \ge 1-\epsilon$$
where $S^{2,r}$ is now the linear presentation with respect to to $1-\epsilon$, and so the rest of the argument now follows through identically.

\subsubsection{$Spin(N)$, $N$ even}
When $N$ is even, $\mathcal{N}_{o,l}^*$ consists of two copies of $\mathcal{P}_{ro}(N)$ instead of one.
However the associated Hecke algebras are the same for both copies and so we treat both cases the same.
Otherwise the proof is essentially the same as the $N$ odd case and all the relevant lemmas have been proved in sufficient generality for the even case.
We do not repeat the proof, but mention the main technical differences.

\begin{itemize}
    \item The partition $\mu(\lambda)$ is defined as
    $$m_{\mu(\lambda)}(x) = \begin{cases} 
    2 &\mbox{if } x \text{ is odd and } 4\mid m_{\lambda^t}(x) \\
    &\mbox{and } m_{\lambda^t}(x) \ne 0\\
    1 & \mbox{if } x \text{ is odd and } m_{\lambda^t}(x) \text{ is}\\
    & \mbox{odd}\\
    0 &\mbox{otherwise.} \end{cases}$$
    \item $J = D_m\times A_{n-2m-1} \times D_m$.
    \item $d$ is even. $\omega$ is type II. $K=D_{d^2}\times A_{d(d-1)/2-1}\times D_{d^2}$.
\end{itemize}

\subsection{Proof of faithfulness in exceptional types}\label{sec:faithfulexceptional}

Let $G^\vee$ be a complex exceptional group and $x\in \overline{\mathfrak C}(G^\vee)$.
Let $\theta(x) = (K,\bf E,\psi,\omega)$.
The block $x$ takes one of the following three forms:
\begin{enumerate}
    \item $L^\vee=T$ and $\mathcal E^\vee$ is the trivial local system;
    \item $L^\vee=G^\vee$ and $\mathcal E^\vee$ is a cuspidal local system;
    \item $L^\vee \ne G^\vee,T^\vee$.
\end{enumerate}
The first case is covered in \cite[Section 4.10, Lemma 3.0.2]{cmo}, and the second case in \cite{cmo2}.
We now cover the third case.
There are very few possibilities for this case: we must have $G^\vee$ is either of type $E_6$ or $E_7$.
If $G^\vee = E_6$ there are two possibilities for $x$.
In both cases however the relative Weyl group is of type $G_2$ and $\mathcal J_x$ is the same set for both.
We can treat these two cases simultaneously.
If $G^\vee = E_7$ there is only one choice for $x$ and the relative Weyl group is of type $F_4$.
We use GAP to exhibit explict pairs $(J,\phi)\in \mathcal F_x$ that realise faithfulness.
We summarise the results in the next paragraph.

We treat the cases $E_6,E_7$ simultaneously.
Suppose $\OO^\vee \in (\cN_o^\vee)_x$.
If $\OO^\vee$ appears in Table \ref{table:e6} or Table \ref{table:e7} take the indicated $(J,\phi)$ (recall that $\WF(\phi,\CC)$ completely characterises $\phi$).
If $\OO^\vee$ does not appear then $\OO^\vee$ has the following properly: $\OO^\vee$ is special and $d(\OO^\vee)$ intersects $\mathbb L_{c^\omega(I_{0,K})}(\overline{\mathbb F_q})$ in a single $F_\omega$-stable special orbit which we denote $\OO_0$.
We take $(J,\phi)$ to be such that $J=I_{0,K}$ and $\WF(\phi,\CC) = \OO_0$.

\begin{table}[H]
    \begin{tabular}{ |c||c|c|c|  }
    \hline
    $\OO^\vee$ & $J$ & Type of $\mathbf L_{c^\omega(J)}(\mathbb F_q)$ & $\mathrm{WF}(\phi,\CC)$\\
    \hline
    %$E_6$ & \dynkin[extended,labels={,,\times,\times,\times,\times,}] E6 & $^3D_4$ & $(1^8)$ \\
    %$E_6(a_1)$ & \dynkin[extended,labels={,,\times,\times,\times,\times,}] E6 & $^3D_4$ & $(2^21^4)$ \\
    %$E_6(a_3)$ & \dynkin[extended,labels={,,\times,\times,\times,\times,}] E6 & $^3D_4$ & $(3^21^2)$ \\
    $A_5$ & \dynkin[extended,labels={\times,\times,,,\times,,\times}] E6 & $A_1\times A_1$ & $(2)\times (2)$ \\
    $2A_2+A_1$ & \dynkin[extended,labels={\times,\times,\times,\times,,\times,\times}] E6 & $A_2$ & $(3)$ \\
    %$2A_2$ & \dynkin[extended,labels={,,\times,\times,\times,\times,}] E6 & $^3D_4$ & $(71)$ \\
    \hline
    \end{tabular} 
    \caption{$(J,\phi)$ for $G^\vee = E_6$}
    \label{table:e6}
\end{table}

%\dynkin[extended,labels={,\times,\times,\times,\times,\times,\times,}] E7 & $^2E_6$
%\dynkin[extended,labels={\times,,\times,\times,\times,\times,,\times}] E7 & $A_1\times \hphantom{ }^2D_4$
%\dynkin[extended,labels={\times,\times,\times,,\times,,\times,\times}] E7 & $A_2\times A_2$
%\dynkin[extended,labels={\times,\times,\times,\times,,\times,\times,\times}] E7 & $A_3\times A_1$
%\dynkin[extended,labels={\times,\times,,\times,\times,\times,\times,\times}] E7 & $^2A_7$
%

\begin{table}[H]
    \begin{tabular}{ |c||c|c|c|  }
    \hline
    $\OO^\vee$ & $J$ & Type of $\mathbf L_{c^\omega(J)}(\mathbb F_q)$ & $\mathrm{WF}(\phi,\CC)$\\
    \hline
    $D_6$ & \dynkin[extended,labels={\times,\times,,\times,\times,\times,\times,\times}] E7 & $^2A_7$ & $(2^4)$ \\
    $E_7(a_4)$ & \dynkin[extended,labels={\times,,\times,\times,\times,\times,,\times}] E7 & $A_1\times \hphantom{ }^2D_4$ & $(2)\times (3^21^2)$ \\
    $D_6(a_2)$ & \dynkin[extended,labels={\times,\times,,\times,\times,\times,\times,\times}] E7 & $^2A_7$ & $(42^2)$ \\
    $D_5(a_1)+A_1$ & \dynkin[extended,labels={\times,\times,,\times,\times,\times,\times,\times}] E7 & $^2A_7$ & $(431)$ \\
    $A_5+A_1$ & \dynkin[extended,labels={\times,\times,\times,,\times,,\times,\times}] E7 & $A_2\times A_2$ & $(3)\times(3)$ \\
    $D_4+A_1$ & \dynkin[extended,labels={\times,\times,,\times,\times,\times,\times,\times}] E7 & $^2A_7$ & $(4^2)$ \\
    $A_3+A_2+A_1$ & \dynkin[extended,labels={\times,\times,,\times,\times,\times,\times,\times}] E7 & $^2A_7$ & $(53)$ \\
    $A_3+2A_1$ & \dynkin[extended,labels={\times,\times,,\times,\times,\times,\times,\times}] E7 & $^2A_7$ & $(62)$ \\
    $A_2+3A_1$ & \dynkin[extended,labels={\times,\times,,\times,\times,\times,\times,\times}] E7 & $^2A_7$ & $(71)$ \\
    $4A_1$ & \dynkin[extended,labels={\times,\times,,\times,\times,\times,\times,\times}] E7 & $^2A_7$ & $(8)$ \\
    \hline
    \end{tabular} 
    \caption{$(J,\phi)$ for $G^\vee = E_7$}
    \label{table:e7}
\end{table}

\begin{rmk}
    We make a brief remark on how to interpret the table.
    Take for example $G^\vee = E_7, \OO^\vee = E_7(a_4)$.
    $J$ takes the form
    $$\dynkin[extended,labels={\times,,\times,\times,\times,\times,,\times}] E7.$$
    The reductive quotient $\mathbf L_{c^\omega(J)}(\overline{\mathbb F_q})$ is of type $A_1\times D_4\times A_1$ and comes with a Frobenius $F_\omega$ that identifies the two $A_1$ factors and folds the $D_4$ factor.
    Thus the $\mathbb F_q$ points $\mathbf L_{c^\omega(J)}(\mathbb F_q)$ has type $A_1\times \hphantom{ }^2D_4$.
    When we say $\WF(\phi,\CC) = (2)\times (3^21^2)$ we are listing the geometric orbit in $A_1\times D_4$.
    The actual geometric orbit in $\mathbf L_{c^\omega(J)}(\overline{\mathbb F_q})$ is $(2)\times (3^21^2)\times (2)$.
\end{rmk}

\begin{rmk}
    We remark that there appears to be a mistake in table for the generalised Springer correspondence for $E_7$ in \cite{spaltensteingeneralised} (and hence in GAP).
    It is listed that 
    $$E(E_7(a_4),\epsilon'') = \chi_{8,3}, \quad E(D_5+A_1,-1) = \chi_{2,3}.$$
    However, it should be
    $$E(E_7(a_4),\epsilon'') = \chi_{2,3}, \quad E(D_5+A_1,-1) = \chi_{8,3}.$$
    This was verified in \cite[Table 5]{ciubotaru}, where the $W(F_4)$-structure of the simple tempered modules with real infinitesimal character (and as a byproduct the generalized Springer correspondence) was computed for the graded affine Hecke algebra attached to the cuspidal local system on the principal nilpotent orbit in the Levi subgroup $(3A_1)'\subset E_7$.
\end{rmk}

\section{Main results}\label{sec:main}

Our main result is a general formula for the (canonical unramified) wavefront set of an irreducible unipotent representation with real infinitesimal character. 

%\begin{definition}
%\label{def:wWF}
%Let $x \in \overline{\mathfrak C}(G^\vee)$, $(\OO^\vee,\mathcal L)\in (\cN_{o,l})_x$, and  $E = E(\OO^\vee,\mathcal L)\in \mathrm{Irr}(W_x)$.
%Define
%%
%$$\WF(E) = \max\{\bar {\mathbb L}(J,\WF(\mathcal U_{x,J}(F),\CC)) : J\in \mathcal J_x, \Hom_{W_{x,J}}(F,E|_{W_{x,J}}) \ne 0\}.$$
%%
%\end{definition}
%
%
%\begin{prop}
%\label{prop:wWF}
%Let $(\OO^\vee,\mathcal L)\in \cN_{o,l}$. Then
%%
%$$\WF(E(\OO^\vee,\mathcal L)) = d_A(\OO^{\vee},1).$$
%%
%\end{prop}
%\begin{proof}
    %Let $x\in \overline{\mathfrak C}(G^\vee)$ be such that $(\OO^\vee,\mathcal L)\in (\cN_{o,l})_x$.
    %By Theorem \ref{thm:faithful} $\OO^\vee$ is $x$-faithful.
    %By Definition \ref{def:faithful} (iii) we must have 
    %
    %$$\WF(E(\OO^\vee,\mathcal L))\le d_A(\OO^\vee,1).$$
    %%
    %Conditions (i) and (ii) guarantee equality is attained.
%\end{proof}

%We now turn to the task of computing wavefront sets of unipotent representations. 
For this we first need an analogue of \cite[Theorem 3.3]{okada2021wavefront}.
Let $X$ be a unipotent representation of $\bfG^\omega(\sfk)$ and let $x = \mathsf{Gcusp}(X) = (J^\vee,\mathcal C^\vee,\mathcal E^\vee)$ and $\theta(x) 
 = \mathsf{Acusp}(X)=(K,\mathbf E,\psi,\omega)$ (c.f. Section \ref{subsec:arithmetic} and Section \ref{subsec:geometric}]). 
Recall from Section \ref{subsec:LLC} that we can associate to $X$ a module $m_{\bf E}(X)$ of the Hecke algebra $\mathcal H = \mathcal H(\bfG^\omega(\sfk),\bf E)$.
Our first order of business is to describe the structure of $X^{\bfU_c(\mf o)}$ in terms of restrictions of $m_{\bf E}(X)$ to subalgebras of $\mathcal H$.

Recall from Section \ref{subsec:wavefrontsets} the definition of the local wavefront sets $^K\WF_c(X)$ for $c\in \mathcal B(\bfG^\omega,\sfk)$.
\begin{lemma}
    \label{lem:canonwf}
    We have 
    $$^K\WF(X) = \max_{J}\hphantom{ }^K\WF_{c^\omega(J)}(X)$$
    where $J$ ranges over $\omega$-stable subsets of $I$ containing $K$.
\end{lemma}
\begin{proof}
    By \cite[Lemma 5.1]{reederhecke}, if $X^{\bfU_{c^\omega(J)}(\mf o)} \ne 0$ then there exists a $\sigma\in \Omega$ such that $\sigma(J)\supseteq K$.
    But since $^K\WF_{c^\omega(J)}$ is constant on $\Omega$-orbits the result follows from the analogue of \cite[Theorem 1.18]{okada2021wavefront} for canonical unramified wavefront sets (this is essentially the content of \cite[Lemma 2.36]{okada2021wavefront}).
\end{proof}

Recall from Section \ref{subsec:faithful} the definition of $\mathcal J_x, \tilde{\mathcal U}_{x,J}$, $\mathcal W_x,W_x,\mathcal W_{x,J}$ and $W_{x,J}$.
Recall from Section \ref{sec:multi} the definition of the $\mathcal W_x$ representation $\widetilde \sigma_0(\theta^*(\mathsf m_{\mathbf E}(V))$.
Write $E$ for this representation.

\begin{lemma}
    \label{lem:wfformula}
    We have 
    $$^K\WF(X) = \max\{\overline{\mathbb L}(J,\WF(\tilde{\mathcal U}_{x,J}(F),\CC)) \mid J\in \mathcal J_x,F\in \mathrm{Irr}(\mathcal W_{x,J}), \Hom_{\mathcal W^\vee_{J}}(F, E) \ne 0\}.$$
\end{lemma}
\begin{proof}
    Let $J\in \mathcal J_x$ .
    By \cite{moyprasad} the representation $X^{\bfU_{c^\omega(J)}(\mf o)}$ of $\bfL_{c^\omega(J)}(\mf o)$ is a direct sum of representations in $\unip_{\bfL_{c^\omega(K)},\bf E}(\bfL_{c^\omega(J)}(\mathbb F_q))$.
    Moreover $(\theta^*\circ m_{\bf E})_{q\to 1}$ furnishes a bijection between $\unip_{\bfL_{c^\omega(K)}(\mathbb F_q),\bf E}(\bfL_{c^\omega(J)}(\mathbb F_q))$ and $\mathrm{Irr}(\mathcal W_{J})$ with inverse $\tilde{\mathcal U}_{x,J}$.
    The result then follows from Equation \ref{e:mult-1} and Lemma \ref{lem:canonwf}.
\end{proof}

\begin{theorem}
\label{thm:realwf}
Let $X=X(s,n,\rho)$ be an irreducible unipotent $\mathbf G^\omega(\mathsf k)$-representation with $s\in T^\vee_\mathbb R$. Write $\OO^{\vee}_X = G^{\vee}n \subset \mathfrak{g}^{\vee}$. Then $^K\WF(\AZ(X))$, $\hphantom{ }^{\bar{\sfk}}\WF(\AZ(X))$ are singletons, and
\begin{align*}
^K\WF(\mathsf{AZ}(X)) &= d_A(\OO^{\vee}_X,1)\\
\hphantom{ }^{\bar{\sfk}}\WF(\AZ(X)) &= d(\OO^{\vee}_X).
\end{align*}
\end{theorem}

\begin{proof}
Arguing as in the proof of \cite[Theorem 3.0.2]{cmo2}, we can reduce to the case when $\mathbf{G}$ is simple and adjoint.

By Equation (\ref{eq:wfcompatibility}) and the identity
\begin{equation}
    \pr_1(d_A(\OO^\vee_X,1)) = d(\OO^\vee_X)
\end{equation}
the formula for $\hphantom{ }^{\bar{\sfk}}\WF(\AZ(X))$ follows from the formula for $^K\WF(\AZ(X))$. So it suffices to show that $^K\WF(\AZ(X)) = d_A(\OO^{\vee}_X,1)$. 

By inverting the character identities (\ref{e:std-irred}), we see that there is an equality in $R(\mathbf{G}^\omega(\sfk))$
\[X = Y(s,n,\rho) + \sum_{n \in \partial (G^{\vee}n')} m_{s,n',\rho'}Y(s,n',\rho')
\]
for standard modules $Y(s,n',\rho')$ and $m_{s,n',\rho'} \in \ZZ$.
Let $x = \mathsf{Gcusp}(X(s,n,\rho)) = (J^\vee,\mathcal C^\vee,\mathcal E^\vee)$ and $\theta(x) = \mathsf{Acusp}(X(s,n,\rho))=(K,\mathbf E,\psi,\omega)$.
By Remark \ref{r:real-geom-diag} $x$ lies in $\overline{\mathfrak C}(G^\vee)$.

If $J\in \mathcal J_x$ and $\tau\in \unip_{\bfL_{c^\omega(K)},\bf E}(\bfL_{c^\omega(J)}(\mathbb F_q))$, let $\tau|_{q\to 1}$ be the corresponding irreducible $\mathcal W_{x,J}$-representation (this was denoted $\theta^*(m_{\mathbf E}(\tau))_{q\to 1}$ in section \ref{sec:multi}). By Corollary \ref{c:mult-real},
    \begin{equation}\label{eq:restr-1}
   \begin{aligned}
   \Hom_{\mathbf P_{c^\omega(J)}(\mathfrak o)}(\tau,Y(s,n',\rho'))&=\Hom_{W_{x,J}}(\tau_{q\to 1}, Y_{\mathbf H}(s,n',\rho')|_{W_{x}}),\\
   %&=\Hom_{\mathcal W^\vee_{{J'}^\vee}}(\theta^*(m_{\mathbf E}(\tau))_{q\to 1}, H_\bullet (\mathcal P_{n},\dot{\mathcal F}^\vee)^\rho)   ,
   \end{aligned}
   \end{equation}
   where $Y_{\mathbf H}(s,n,\rho)$ is the standard module for the graded affine Hecke algebra $\mathbf H(G^\vee,G^\vee_{J^\vee},\mathcal C^\vee,\mathcal F^\vee)$. To simplify the notation, let us denote by $\Res_{\bfL_{c^\omega(J)}(\mathbb F_q)}(X)$ the character of the parahoric restriction of an admissible representation $V$. Then we can rewrite (\ref{eq:restr-1}) in the following form:
   \begin{equation}
       \Res_{\bfL_{c^\omega(J)}(\mathbb F_q)}(Y(s,n',\rho'))=\Res^{W_x}_{W_{x,J}}( Y_{\mathbf H}(s,n',\rho')),
   \end{equation}
by which we understand that the multiplicity of $\tau$ in the left hand side equals the multiplicity of $\tau_{q\to 1}$ in the right hand side.

By Proposition \ref{p:stand-graded-struct}, we have the following identity of $W_x$-characters:
  \[
    Y_\bH(s,n',\rho')|_{W_x}=\sum_{\tilde\rho\in \widehat{A(n')}_{\CL}} m(\rho',\tilde\rho) \mu(n',\tilde\rho) + \sum_{(n'',\tilde\rho'')} m_{(n',\tilde\rho),(n'',\tilde\rho')} \mu(n'',\tilde\rho'),
    \]
    where \[m(\rho',\tilde\rho)=\dim \Hom_{A(s,n')}[\rho',\tilde\rho|_{A(s,n')}],\] $m_{(n',\tilde\rho),(n'',\tilde\rho')}\in \mathbb Z_{\ge 0}$, $\tilde\rho'\in \widehat{A(n'')}_{\CL}$, and $n''$ ranges over a set of representatives of $G^\vee$-orbits in $\mathcal N^\vee$ such that $n'\in\partial {G^\vee n''}$. 

This means that
\begin{equation}
     \Res_{\bfL_{c^\omega(J)}(\mathbb F_q)}(X)=\sum_{\tilde\rho\in \widehat{A(n)}_{\CL}} m(\rho,\tilde\rho) \mu(n,\tilde\rho)|_{W_{x,J}} + \sum_{\substack{\mu'\in\Irr(W_x)\\n\in \partial\OO^\vee(\mu',\mathbb C)}} m_{\mu'}\mu'|_{W_{x,J}},
\end{equation}
for some integers $m_{\mu'}$. Consequently (see Remark \ref{r:tensor-sign}),
\begin{equation}\label{e:tensor-sign}
     \Res_{\bfL_{c^\omega(J)}(\mathbb F_q)}(\AZ(X))=\sum_{\tilde\rho\in \widehat{A(n)}_{\CL}} m(\rho,\tilde\rho) ~(\mu(n,\tilde\rho)\otimes\mathrm{sgn})|_{W_{x,J}} + \sum_{\substack{\mu'\in\Irr(W_x)\\n\in \partial\OO^\vee(\mu',\mathbb C)}} m_{\mu'}~(\mu'\otimes\mathrm{sgn})|_{W_{x,J}},
\end{equation}

Notice that our normalization is such that $\mu(n,\tilde\rho)\otimes\mathrm{sgn}$ is precisely the irreducible generalized Springer representation attached to the pair $(n,\tilde\rho)$ i.e. $E(\OO^\vee_n,\tilde\rho)$.  
Thus
\begin{equation}
     \Res_{\bfL_{c^\omega(J)}(\mathbb F_q)}(\AZ(X))=\sum_{\tilde\rho\in \widehat{A(n)}_{\CL}} m(\rho,\tilde\rho) E(\OO^\vee_X,\tilde\rho)|_{W_{x,J}} + \sum_{\substack{\OO^\vee < \OO'^\vee,\\ (\OO'^\vee,\tilde\rho')\in (\cN_{o,l})_x}} m_{\OO'^\vee,\tilde\rho'} E(\OO^\vee_{1},\tilde\rho')|_{W_{x,J}}.
\end{equation}
Let $F$ be an irreducible constituent of $\Res_{\bfL_{c^\omega(J)}(\mathbb F_q)}(\AZ(X))$.
Then $\Hom_{W_{x,J}}(F,E(\OO'^\vee,\tilde \rho'))\ne 0$ for some $\OO'^\vee\ge \OO^\vee_X$.
By condition (iii) of faithfulness, we have
$$\overline{\mathbb L}(J,\WF(\mathcal U_{x,J}(F),\CC)) \le d_A(\OO'^\vee,1)\le d_A(\OO^\vee_X,1).$$
Thus by Lemma \ref{lem:wfformula} $^K\WF(X)\le d_A(\OO^\vee_X,1)$.
To establish equality note that there exists at least one $\tilde\rho$ such that $m(\rho,\tilde\rho)\neq 0$.
By conditions (i) and (ii) of faithfulness there exists a $J$ and $F\in \mathrm{Irr}(W_{x,J})$ such that $\Hom_{W_{x,J}}(F,E(\OO^\vee_X,\tilde\rho))\ne 0$ and
$$\overline{\mathbb L}(J,\WF(\mathcal U_{x,J}(F),\CC)) = d_A(\OO^\vee_X,1).$$
But
$$\Hom_{W_{x,J}}(F,E(\OO^\vee_X,\tilde\rho))\ne 0 \implies \Hom_{W_{x,J}}(F,\Res_{\bfL_{c^\omega(J)}(\mathbb F_q)}(\AZ(X)))\ne 0$$
since $m(\rho,\tilde\rho)\neq 0$ and so by Lemma \ref{lem:wfformula} we get $^K\WF(\AZ(X)) = d_A(\OO^\vee_X,1)$ as desired.

\end{proof}

\begin{rmk}\label{r:tensor-sign}
    In the above proof, we have used implicitly (see (\ref{e:tensor-sign})) that the restrictions of $\AZ(X)$ are obtained from those of $X$ by tensoring with $\mathrm{sgn}$. In the Iwahori-spherical case, we presented the details in \cite[Section 2.7]{cmo}. In the general case, the argument is identical using Kato's duality for affine Hecke algebras with unequal parameters \cite{Kato} (see also \cite[Section 6]{chan-duality-2016} for graded Hecke algebras), and the known compatibility between the parabolic induction (respectively, Jacquet restriction) functors for $\bfG^\omega(\mathsf k)$-representations  in a Bernstein component and the parabolic induction (respectively, restriction) functors for the modules of the corresponding Hecke algebras (see \cite{bushnellkutzko-types}). 
\end{rmk}

Now let $\OO^{\vee} \subset \mathfrak{g}^{\vee}$ be a nilpotent orbit. Choose an $\mathfrak{sl}(2)$-triple $(e^{\vee},f^{\vee},h^{\vee})$ with $e^{\vee} \in \OO^{\vee}$.

\begin{cor}\label{cor:wfbound}
Let $X=X(\frac{1}{2}h^{\vee},n,\rho)$ be an irreducible unipotent $\mathbf{G}^{\omega}(\sfk)$-representation. Then
\begin{align}\label{eq:WFsetbound}
\begin{split}
d_A(\OO^{\vee}, 1) &\leq_A \hphantom{ } ^K\WF(X)\\
d(\OO^{\vee}) &\leq \hphantom{ }^{\bar{\sfk}}\WF(X)
\end{split}
\end{align}
\end{cor}

\begin{proof}
The second inequality follows from the first by applying $\mathrm{pr}_1$ to both sides. So it suffices to show that $d_A(\OO^{\vee}, 1) \leq_A \hphantom{ } ^K\WF(X)$.

Let $X'=\AZ(X)$. Note that $X' = X(\frac{1}{2}h^{\vee},n',\rho')$, for some $n'$ and $\rho'$. By \cite[Lemma 3.0.9]{cmo}, we have $\OO^{\vee}_{X'} = G^{\vee}n' \leq \OO^{\vee}$. In particular, $(\OO^\vee_{X'},1)\le_A(\OO^\vee,1)$ and so $d_A(\OO^\vee,1)\le_A d_A(\OO^\vee_{X'},1)$ since $d_A$ is order-reversing. But by Theorem \ref{thm:realwf}, $d_A(\OO^\vee_{X'},1) = \hphantom{ }^K\WF(X)$. This completes the proof.
\end{proof}

\begin{sloppypar} \printbibliography[title={References}] \end{sloppypar}

\end{document}